\documentclass[11pt]{article}
\usepackage{epic,latexsym,tikz}
\usepackage{epsfig,enumerate,amsmath,amsfonts,amsbsy,amssymb,amsthm,mathrsfs,lineno,ifpdf, mathtools}
\usepackage[mathscr]{euscript}
\usepackage[numbers,sort&compress]{natbib}
\usepackage{hyperref, url}
\usepackage{setspace,graphicx}
\usepackage[capitalize]{cleveref}
\usepackage{pgfplots}
\pgfplotsset{compat=1.15}
\usepackage{mathrsfs}
\usetikzlibrary{arrows}
\usepackage{float}
\usepackage[usenames,dvipsnames]{pstricks}
\usepackage[margin=1.2in]{geometry}
\usepackage{comment}

\newtheorem{theorem}{Theorem}
\newtheorem{lemma}{Lemma}

\newtheorem{obs}{Observation}

\newtheorem{claim}{Claim}

\usepackage[normalem]{ulem}
\usepackage{enumitem}
\usepackage{etoolbox}

\numberwithin{equation}{section}

\hypersetup{
	colorlinks=true,
	linkcolor=blue,
	filecolor=magenta,
	pdfpagemode=FullScreen,
}
\DeclarePairedDelimiter\floor{\lfloor}{\rfloor}

\begin{document}

\title{Disjunctive domination in maximal outerplanar graphs}
	
\author{$^1$Michael A. Henning\thanks{Research supported in part by the South African National Research Foundation under grant number 129265 and the University of Johannesburg}, $^{2}$Paras Vinubhai Maniya, and $^2$Dinabandhu Pradhan\thanks{Corresponding author.}  \\ \\
$^{1}$Department of Mathematics and Applied Mathematics\\
University of Johannesburg \\
Auckland Park, South Africa\\
\small \tt Email: mahenning@uj.ac.za
\\ \\
$^{2}$Department of Mathematics \& Computing\\
Indian Institute of Technology (ISM) \\
Dhanbad, India \\
\small \tt Email: maniyaparas9999@gmail.com \\
\small \tt Email: dina@iitism.ac.in}

\date{}
\maketitle

\begin{abstract}
A disjunctive dominating set of a graph $G$ is a set $D \subseteq V(G)$ such that every vertex in $V(G)\setminus D$ has a neighbor in $D$ or has at least two vertices in $D$ at distance $2$ from it. The disjunctive domination number of $G$, denoted by $\gamma_2^d(G)$, is the minimum cardinality of a disjunctive dominating set of $G$. In this paper, we show that if $G$ is a maximal outerplanar graph of order $n \ge 7$ with $k$ vertices of degree $2$, then $\gamma_2^d(G)\le \floor*{\frac{2}{9}(n+k)}$, and this bound is sharp.
\end{abstract}
	

\section{Introduction}

All the graphs considered in this paper are finite, simple, and undirected. For a graph $G$, we use $V(G)$ and $E(G)$ to denote the vertex set and the edge set of $G$, respectively. Two vertices $u$ and $v$ of $G$ are \emph{adjacent} if $uv\in E(G)$. Two adjacent vertices are called \emph{neighbors}. The \emph{open neighborhood} $N_G(v)$ of a vertex $v$ in $G$ is the set of neighbors of $v$, while the \emph{closed neighborhood} of $v$ is the set $N_G[v] = \{v\} \cup N_G(v)$. The \emph{degree} of a vertex $v$ in $G$ is the number of vertices adjacent to $v$ in $G$, and is denoted by $\deg_G(v)$, and so $\deg_G(v) = |N_G(v)|$. A vertex of degree~$1$ in $G$ is  called a \emph{leaf} (and also called a \emph{pendant vertex} in the literature). The distance between vertices $u$ and $v$ in $G$ is the minimum length of a path between $u$ and $v$, and is denoted by $d_G(u,v)$.  For a given positive integer $l$, we use the notation $[l]$ to denote the set $\{1,2,\ldots,l\}$.

A \emph{rooted tree} $T$ distinguishes one vertex $r$ called the \emph{root}. Let $T$ be a tree rooted at vertex $r$.  For each vertex  $v \ne r$ of $T$, the \emph{parent} of $v$ is the neighbor of $v$ on the unique $(r,v)$-path, while a \emph{child} of $v$ is any other neighbor of $v$. A \emph{descendant} of $v$ is a vertex $u \ne v$ such that the unique $(r,u)$-path contains $v$. Thus, every child of $v$ is a descendant of $v$. We let $C(v)$ and $D(v)$ denote the set of children and descendants, respectively, of $v$, and we define $D[v] = D(v) \cup \{v\}$. The \emph{maximal subtree} rooted at $v$ is the subtree of $T$ induced by $D[v]$, and is denoted by $T_v$. A \emph{diametrical vertex} of $T$ is a leaf that belongs to a longest path in $T$.

In this paper, we study planar graphs. A \emph{plane embedding} of a planar graph $G$ is an embedding of $G$ in a plane such that the edges of $G$ do not intersect each other except at the endpoints. A planar graph with a plane embedding is called a \emph{plane graph}. A \emph{triangulated disk} or \emph{near-triangulation} is a $2$-connected plane graph all of whose interior faces are triangles. A \emph{maximal outerplanar graph} $G$, abbreviated \emph{mop}, is a plane graph such that all vertices lie on the boundary of the outer face (unbounded face) and all inner faces are triangles.

Throughout our discussion, we refer to inner faces of a maximal outerplanar graph as triangles. Two faces are adjacent if they share a common edge. A triangle of $G$ that is not adjacent to the outer face is called an \emph{internal triangle} of $G$. An edge on the outer face (unbounded face) is called an \emph{outer edge} of $G$, while any other edge of $G$ is called a \emph{diagonal} of $G$. A \emph{region} $R \, \colon v_1v_2 \ldots v_k$ of $G$ is a maximal outerplanar subgraph of $G$ such that one outer edge of $R$ is diagonal of $G$ and all other outer edges of $R$ are outer edges of $G$. For an edge $e=xy$ of $G$, the contraction of an edge $e$  of $G$ is the graph obtained from $G$ by deleting $x$ and $y$ (and all incident edges), adding a new vertex $v$, and adding edges between $v$ and each vertex in $(N_G(x)\cup N_G(y)) \setminus \{x,y\}$.

\begin{lemma}[\cite{Ro-83}]\label{key}
	If $H$ is obtained by the contraction of an outer edge $e$ in a mop $G$ of order~$n \ge 4$, then $H$ is also a mop.
\end{lemma}

A set $D\subseteq V(G)$ is a \emph{dominating set} of $G$ if every vertex in $V(G) \setminus D$ is adjacent to at least one vertex in $D$. The \emph{domination number} of $G$, denoted by $\gamma(G)$, is the minimum cardinality among all dominating sets of $G$. Domination and its variants are well-explored topics in graph theory, with existing literature thoroughly reviewed in~\cite{HaHeHe-20,HaHeHe-21,HaHeHe-23,HeYe-book}.

Goddard, Henning, and McPillan introduced the concept of disjunctive domination in graphs inspired by distance domination and exponential domination. In a graph $G$, a set $S\subseteq V(G)$ is called a \emph{disjunctive dominating set}, abbreviated 2DD-set, if every vertex in $V(G)\setminus S$ is adjacent to a vertex in $S$ or has at least two vertices in $S$ at distance~$2$ from it. A vertex $x$ in $G$ is said to be \emph{disjunctive dominated} by the set $S$ if it is adjacent to at least one vertex of $S$ or has at least two vertices in $S$ at distance~$2$ from it. The \emph{disjunctive domination number} of $G$,  denoted by $\gamma_2^d(G)$, is the minimum cardinality among all disjunctive dominating sets of $G$. A 2DD-set of cardinality $\gamma_2^d(G)$ is called a $\gamma_2^d$-set of $G$. The concept of disjunctive domination in graphs has been examined in \cite{GoHeMc-14,HaMa-15,HaMa-16a,HaMa-16b,Zh-23}.

In this paper, we study disjunctive domination in planar graphs. The study of domination in maximal outerplanar graphs has been extensively studied since 1975. In a seminal paper~\cite{Ch-75}, Chv\'{a}tal showed  that the domination number of a maximal outerplanar graph of order $n$ is at most $n/3$. Campos and Wakabayashi~\cite{CaWa-13} demonstrated that for a mop $G$ of order~$n$, $\gamma(G)\le \floor*{\frac{1}{4}(n+k)}$, where $k$ is the number of vertices with degree~$2$. Tokunaga independently confirmed this result in~\cite{To-13}. For additional variants of domination in maximal outerplanar
graphs, we refer the reader to the references ~\cite{Aita24,Alvarado18,Araki18,Dorfling16,Lemanska17}.

\section{Main result}
\label{Sect:main result}

Since every dominating set is a disjunctive dominating set of a graph $G$, we note that $\gamma_2^d(G) \le \gamma(G)$. In particular, for a mop $G$ of order $n$ with $k$ vertices of degree~$2$, we infer that $\gamma_2^d(G) \le \gamma(G) \le \floor*{\frac{1}{4}(n+k)}$. A natural problem is to determine whether this bound on $\gamma_2^d(G)$ can be improved, and if so, what is a tight bound in the sense that it is achievable for mops. Our main result is the following improved upper bound on the disjunctive domination  number of a mop.

\begin{theorem}
	\label{thm:main}
	If $G$ is a mop of order $n \ge 7$ with $k$ vertices of degree~$2$, then $\gamma_2^d(G) \le \floor*{\frac{2}{9}(n+k)}$.
\end{theorem}

We proceed as follows. In Section~\ref{Sect:prelim}, we present observations of maximal outerplanar graphs and present preliminary lemmas on the disjunctive domination  number of a mop. Thereafter in Section~\ref{Sect:main-proof}, we present a proof of our main result, namely Theorem~\ref{thm:main}. In Section~\ref{Sect:conclude}, we present examples that demonstrate the tightness of the given bound.

\section{Preliminary results and lemmas}
\label{Sect:prelim}

In the following, we present properties of maximal outerplanar graphs, which are well-known or easy to observe.

\begin{obs}
	\label{5vertex}
	If $G$ is a mop of order~$5$, then there exists a vertex adjacent to all other vertices of $G$.
\end{obs}

\begin{obs}
	\label{obs3}		
	If $G$ is a mop of order~$n$ where $5 \le n \le 8$ and if $v_1v_2\ldots v_nv_1$ represent the boundary of the outer face of $G$, then $\{v_i,v_{i+4}\}$ is a 2DD-set of $G$ for any $i \in [n]$, where $i+4$ is calculated modulo~$n$.
\end{obs}

\begin{obs}\label{obs1}
	Let $G$ be a mop of order~$n \ge 4$ and let $v_1v_2\ldots v_nv_1$ represent the boundary of the outer face of $G$. If $v_iv_j$ is a diagonal of $G$, where $i < j$, then $v_i$ and $v_j$ have two common neighbors $v_k$ and $v_l$, where $k\in \{i+1,i+2,\ldots,j-1\}$ and $l\in \{j+1,\ldots,n,1,\ldots,i-1\}$.
\end{obs}

Using~\cref{obs1}, we have the following observation.

\begin{obs}\label{obs2}		
	Let $G$ be a mop of order $n \ge 4$ that does not contain any internal triangles, and let $v_1v_2\ldots v_nv_1$ represent the boundary of the outer face of $G$. If $v_iv_j$ is a diagonal of $G$, where $i< j$, then $v_i$ and $v_j$ share a common neighbor, which can be either $v_{i+1}$ or $v_{j-1}$. Similarly, they also share a common neighbor, which can be either $v_{j+1}$ or $v_{i-1}$.
\end{obs}

We state next two known lemmas from the literature.

\begin{lemma}[\cite{CaWa-13}]\label{internal}
	If $G$ is a mop of order $n \ge 4$, then $G$ has at least two vertices of degree~$2$. Furthermore, if $G$ has $k$ internal triangles, then it has $k+2$ vertices of degree $2$.
\end{lemma}

\begin{lemma}[\cite{Ch-75}]\label{partition}
	If $G$ is a mop of order $n \ge 6$, then $G$ has a diagonal~$d$ that partitions it into two mops $G_1$ and $G_2$ such that $G_1$ has a exactly $4$, $5$, or $6$ outer edges of $G$.
\end{lemma}

The following lemma shows that for mops $G$ of a small order $n$, the bound $\gamma_2^d(G) \le \lfloor \frac{2}{9}(n+k) \rfloor$ is satisfied, where $k$ is the number of vertices of degree~$2$ in $G$.

\begin{lemma}\label{7to12mops}
	If $G$ is a mop of order~$n$ where $7 \le n \le 12$ with $k$ vertices of degree~$2$, then $\gamma_2^d(G) \le \floor*{\frac{2}{9}(n+k)}$.
\end{lemma}
\begin{proof}
	By \cref{internal}, the mop $G$ has at least two vertices of degree~$2$. If $7 \le n \le 8$, then by \cref{obs3}, we have $\gamma_2^d(G) \le 2 \le  \floor*{\frac{2}{9}(n+k)}$. Hence, we may assume that $n \ge 9$. Suppose $n=9$. Let $v_1v_2\ldots v_9v_1$ be the boundary of the outer face of $G$. Since $G$ has at least two vertices of degree~$2$, without loss of generality, assume that $\deg_G(v_2)=2$. Since $G$ is a mop and $N_G(v_2)=\{v_1,v_3\}$, $v_1v_3\in E(G)$. In this case, $\{v_1,v_6\}$ is a 2DD-set of $G$, and so  $\gamma_2^d(G)\le 2  \le  \floor*{\frac{2}{9}(n+k)}$. Hence we may assume that $n \ge 10$, for otherwise the desired upper bound follows.
	
	Suppose $n=10$. Let $v_1v_2\ldots v_{10}v_1$ be the boundary of the outer face of $G$. By Lemma~\ref{partition}, the mop $G$ has a diagonal $d=v_iv_j$ that partitions it into mops $G_1$ and $G_2$ such that $V(G_1)\cap V(G_2)=\{v_i,v_j\}$, $E(G_1)\cap E(G_2)=\{d\}$, and $G_1$ has exactly $4$, $5$ or $6$ outer edges of $G$. Let $G_1$ has exactly $p$ outer edges of $G$ for some $p\in \{4,5,6\}$. Without loss of generality, assume that $d=v_1v_{p+1}$ and $V(G_1)=\{v_1,v_2,\ldots,v_{p+1}\}$. Suppose $p=4$. By \cref{5vertex}, there exists a vertex $v_i\in V(G_1)$ such that $v_i$ is adjacent to all other vertices of $G_1$ for some $i\in [5]$. Thus, $\{v_i,v_8\}$ is a 2DD-set of $G$, and so $\gamma_2^d(G) \le 2  \le  \floor*{\frac{2}{9}(n+k)}$. Suppose $p=6$. In this case, $G_2$ has exactly four outer edges of $G$. Thus, this case is identical to the case $p=4$. So we assume that $p=5$. Thus both $G_1$ and $G_2$ have exactly five outer edges of $G$ and the diagonal $d=v_1v_6$. We note that there exists an internal triangle $F$ with vertex set $\{v_1,v_6,v_i\}$ in $G_1$ for some $i\in \{3,4\}$, for otherwise, there exists a mop $G'_1$ with exactly four outer edges of $G$. We will present similar arguments as in the case of $p=4$. Similarly, $G_2$ also has an internal triangle with vertex set $\{v_1,v_6,v_i\}$ in $G_2$ for some $i\in \{8,9\}$. Thus, $\{v_1,v_6\}$ is a 2DD-set of $G$, and so $\gamma_2^d(G)\le 2  \le  \floor*{\frac{2}{9}(n+k)}$. Hence we may assume that $n \ge 11$.
	
	Suppose $n=11$. Let $v_1v_2\ldots v_{11}v_1$ be the boundary of the outer face of $G$. If $G$ has at least one internal triangle, then by \cref{internal}, the mop $G$ has at three vertices of degree~$2$. Thus, $\{v_1,v_5,v_9\}$ is a 2DD-set of $G$, and so $\gamma_2^d(G)\le 3  \le  \floor*{\frac{2}{9}(n+k)}$. Hence we assume that $G$ has no internal triangle. By Lemma~\ref{partition}, the mop $G$ has a diagonal $d=v_iv_j$ that partitions it into mops $G_1$ and $G_2$ such that $V(G_1)\cap V(G_2)=\{v_i,v_j\}$, $E(G_1)\cap E(G_2)=\{d\}$, and $G_1$ has a exactly $4$, $5$ or $6$ outer edges of $G$. Let $G_1$ has exactly $p$ outer edges of $G$ for some $p\in \{4,5,6\}$. Without loss of generality, assume that $d=v_1v_{p+1}$ and $V(G_1)=\{v_1,v_2,\ldots,v_{p+1}\}$.

	Suppose firstly that $p=4$. By \cref{5vertex}, there exists a vertex $v_i\in V(G_1)$ such that $v_i$ is adjacent to all other vertices of $G_1$ for some $i\in [5]$. We note that $v_3$ is not adjacent to the remaining vertices of $G_1$, for otherwise $G$ has an internal triangle. If $v_1$ is adjacent to remaining vertices of $G_1$, then $\{v_1,v_8\}$ of $G$. If $v_5$ is adjacent to the remaining vertices of $G_1$, then $\{v_5,v_9\}$ of $G$, and so $\gamma_2^d(G)\le 2  \le  \floor*{\frac{2}{9}(n+k)}$. Hence, $v_2$ or $v_4$ is adjacent to every vertex in $G_1$. By symmetry, we may assume that $v_2$ is adjacent to every vertex in $G_1$. Therefore, $v_2$ is at distance~$2$ from $v_6$ and $v_{11}$. We note that $v_5v_6\ldots v_{11}v_1v_5$ is the boundary of the outer face of $G_2$. Since $G$ has no internal triangle, $G_2$ also has no internal triangle. By \cref{obs2}, we have either $v_1v_6\in E(G)$ or $v_5v_{11}\in E(G)$. Without loss of generality, we assume that $v_1v_6\in E(G)$. Again since $G_2$ has no internal triangle and by \cref{obs2}, either $v_1v_7\in E(G)$ or $v_6v_{11}\in E(G)$.

	Suppose firstly that $v_1v_7\in E(G)$. Since $G_2$ has no internal triangle and by \cref{obs2}, either $v_1v_8\in E(G)$ or $v_7v_{11}\in E(G)$. Suppose $v_1v_8\in E(G)$. Therefore there exists a mop $G_3$ with exactly four outer edges of $G$. We note that $V(G_3)=\{v_1,v_8,v_9,v_{10},v_{11}\}$. By \Cref{5vertex}, there exists a vertex $v_i\in V(G_3)$ such that $v_i$ is adjacent to all other vertices of $G_3$ for some $i\in \{1,8,9,10,11\}$. Thus, $\{v_i,v_2\}$ is a 2DD-set of $G$, and so $\gamma_2^d(G)\le 2  \le  \floor*{\frac{2}{9}(n+k)}$. Hence we may assume that $v_7v_{11}\in E(G)$. Therefore there exists a mop $G_4$ with exactly four outer edges of $G$. We note that $V(G_4)=\{v_7,v_8,v_9,v_{10},v_{11}\}$. By \Cref{5vertex}, there exists a vertex $v_i\in V(G_4)$ such that $v_i$ is adjacent to all other vertices of $G_4$ for some $i\in \{7,8,9,10,11\}$. Thus, $\{v_i,v_2\}$ is a 2DD-set of $G$, and so $\gamma_2^d(G)\le 2  \le  \floor*{\frac{2}{9}(n+k)}$, as desired.
	
	Hence we may assume that $v_6v_{11}\in E(G)$. Since $G_2$ has no internal triangle and by \cref{obs2}, either $v_6v_{10}\in E(G)$ or $v_7v_{11}\in E(G)$. Suppose $v_6v_{10}\in E(G)$. Therefore there exists a mop $G_3$ with exactly four outer edges of $G$. We note that $V(G_3)=\{v_6,v_7,v_8,v_9,v_{10}\}$. By \Cref{5vertex}, there exists a vertex $v_i\in V(G_3)$ such that $v_i$ is adjacent to all other vertices of $G_3$ for some $i\in \{6,7,8,9,10\}$. Thus, $\{v_i,v_2\}$ is a 2DD-set of $G$, and so  $\gamma_2^d(G)\le 2  \le  \floor*{\frac{2}{9}(n+k)}$. Hence we may assume that $v_7v_{11}\in E(G)$. Therefore there exists a mop $G_4$ with exactly four outer edges of $G$. We note that $V(G_4)=\{v_7,v_8,v_9,v_{10},v_{11}\}$. By \Cref{5vertex}, there exists a vertex $v_i\in V(G_4)$ such that $v_i$ is adjacent to all other vertices of $G_4$ for some $i\in \{7,8,9,10,11\}$. Thus, $\{v_i,v_2\}$ is a 2DD-set of $G$, and so $\gamma_2^d(G)\le 2  \le  \floor*{\frac{2}{9}(n+k)}$, as desired. Hence we have shown that if $p = 4$, then the desired bound holds.

	Suppose next that $p=5$. Since $G$ has no internal triangle and by \cref{obs2}, either $v_1v_{5}\in E(G)$ or $v_2v_{6}\in E(G)$. Therefore there exists a mop graph $G'_1$ with exactly four outer edges of $G$. Present similar arguments as in the case of $p=4$, we infer that the desired bound $\gamma_2^d(G)\le 2  \le  \floor*{\frac{2}{9}(n+k)}$ holds. Hence we assume that $p=6$. Thus, $G_2$ has exactly five outer edges of $G$, and so this case is identical to the case $p=5$ analyzed earlier. Hence we have shown that if $n=11$, then  the desired bound holds.

	Suppose that $n=12$. Let $v_1v_2\ldots v_{12}v_1$ be the boundary of the outer face of $G$. Thus, $\{v_1,v_5,v_9\}$ is a 2DD-set of $G$, and so  $\gamma_2^d(G)\le 3  \le  \floor*{\frac{2}{9}(n+k)}$. This completes the proof of \cref{7to12mops}.
\end{proof}

Let $G$ be a mop of order $n \ge 4$. Thus all vertices lie on the boundary of the outer face (unbounded face) of $G$ and all inner faces are triangles. Let $T$ be the graph whose vertices correspond to the triangles of $G$, and where two vertices in $T$ are adjacent if their corresponding triangles in $G$ share an edge. If $T$ contains a cycle, it would imply that a vertex is enclosed by triangles within the graph, which contradicts the outerplanarity of $G$. Hence, $T$ is necessarily a tree. We refer the tree $T$ as the tree associated with the mop $G$. The tree $T$ has maximum degree at most~$3$, and a triangle of $G$ corresponding to a vertex of degree~$3$ in $T$ is necessarily an internal triangle of $G$. Next, we will investigate the maximum possible distance between a leaf in $T$ and its nearest vertex of degree~$3$.

\begin{lemma}\label{lem1}
	If $G$ be a mop of order $n \ge 7$ with $k$ vertices of degree~$2$ and $T$ is a tree associated with~$G$, then either $\gamma_2^d(G) \le \floor*{\frac{2}{9}(n+k)}$ or the following conditions hold where $x$ is a leaf of $T$.  \\ [-24pt]
	\begin{enumerate}[label=\rm{(\alph*)}]
		\item \label{npath} $T$ is not a path graph.\\ [-22pt]
		\item \label{pdistance} If $y$ is a nearest vertex of degree~$3$ from $x$ in $T$, then $d_T(x,y)=i$, where $i\in\{1,2,5,6\}$.\\ [-22pt]
		\item \label{5distance} If $y$ is a nearest vertex of degree~$3$ from $x$ in $T$ and $d_T(x,y)=5$, then the subgraph of $G$ associated with the path between $x$ and $y$ in $T$ corresponds to the region $H_1$ or $H_2$ illustrated in {\rm{\Cref{5-distance}(a)-(b)}}.\\ [-22pt]
		\item \label{6distance} If $y$ is a nearest vertex of degree~$3$ from $x$ in $T$ and $d_T(x,y)=6$, then the subgraph of $G$ associated with the path between $x$ and $y$ in $T$ corresponds to the region $H_5$, $H_6$, $H_7$, or $H_8$ illustrated in {\rm{\Cref{6-distance}(a)-(d)}}.\\ [-22pt]
	\end{enumerate}
\end{lemma}
\begin{proof}
	If $7 \le n\le 12$, then by \cref{7to12mops}, $\gamma_2^d(G) \le  \floor*{\frac{2}{9}(n+k)}$. Hence we assume that the mop $G$ has order $n \ge 13$ otherwise the desired result follows. Suppose that $\gamma_2^d(G) > \floor*{\frac{2}{9}(n+k)}$. Among all such mops $G$, let $G$ be chosen to have minimum order $n \ge 13$ where as before $G$ has $k$ vertices of degree~$2$. By the minimality of $G$, if $G'$ is a mop of order~$n'$ where $7 \le n' < n$, with $k'$ vertices of degree~$2$, then $\gamma_2^d(G') \le \lfloor \frac{2}{9}(n'+k') \rfloor$. We will now show that tree $T$ corresponding to a mop $G$ satisfies the conditions (a), (b), (c), and (d) mentioned in the statement of the lemma.

	Let $t_1$ be a leaf of $T$ and $F_1$ be a triangle  in $G$ corresponding to the vertex $t_1$ in $T$. Further, let $V(F_1)= \{u_1,u_2,u_3\}$. Let $t_2$ be the support vertex of $T$ adjacent to the leaf $t_1$, and let $F_2$ be the triangle in $G$ corresponding to the vertex $t_2$ in $T$. Renaming vertices of $F_1$ if necessary, we may assume that $V(F_2)=\{u_2,u_3,u_4\}$, and so $u_2u_3$ is the common edge of the triangles $F_1$ and $F_2$. If $\deg_T(t_2)=1$, then the order of $G$ is $4$, a contradiction to fact that $G$ is a mop of order $n \ge 13$. Hence, $2 \le \deg_{T}(t_2)\le 3$ and $d_{T}(t_1,t_2)=1$. If $\deg_{T}(t_2)=3$, then \Cref{lem1}\ref{pdistance} holds. Hence we assume that $\deg_{T}(t_2)=2$.
	
	Let $t_3 \in V(T)$ be the neighbor of $t_2$ different form $t_1$, and let $F_3$ be the triangle in $G$ corresponding to the vertex $t_3$ in $T$. Let $u_5$ be the vertex in $F_3$ that is not in $F_2$. Renaming the vertices $u_2$ and $u_3$ necessary, we assume that $V(F_3)=\{u_2,u_4,u_5\}$. Since $\deg_T(t_1)=1$ and $\deg_T(t_2)=2$, we note that there are no further edges incident with $u_1$ and $u_3$ in $G$, and $\deg_G(u_1)=2$ and $\deg_G(u_3)=3$.
	
	If $\deg_T(t_3)=1$, then the order of $G$ is~$5$, a contradiction to the fact that $G$ is a mop of order $n \ge 13$. Hence, $2 \le \deg_{T}(t_3)\le 3$ and $d_{T}(t_1,t_3)=2$. If $\deg_{T}(t_3)=3$, then \Cref{lem1}\ref{pdistance} holds. Hence we may assume that $\deg_{T}(t_3)=2$. Let $t_4 \in V(T)$ be the neighbor of $t_3$ in $T$ different from $t_2$ in $T$ and let $F_4$ be the triangle in $G$ corresponding to the vertex $t_4$ in $T$. Thus, $d_{T}(t_1,t_4)=3$. Further let $u_6$ be the vertex in $F_4$ that is not in $F_3$. We note that either $V(F_4)=\{u_2,u_5,u_6\}$ or $V(F_4)=\{u_4,u_5,u_6\}$ (see Figure~\ref{triangle}(a)-(b)).  If $\deg_T(t_4)=1$, then $n=6$, a contradiction to the fact that $G$ is a mop of order $n \ge 13$. Hence, $2 \le \deg_{T}(t_4)\le 3$.

	\begin{figure}[htb]
		\begin{center}
			\includegraphics[scale=.22]{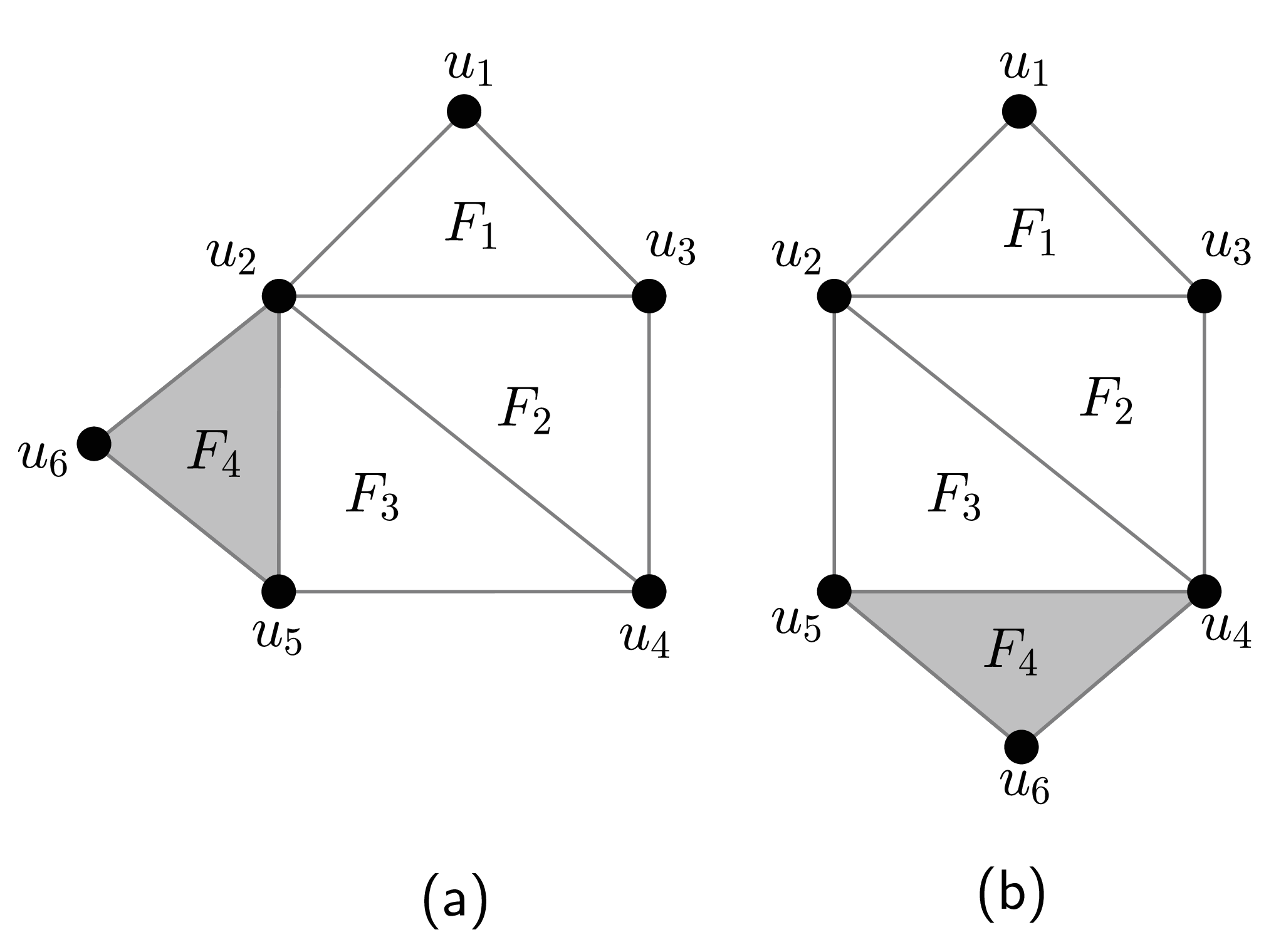}
			\caption{Possible (shaded) triangle adjacent to triangle $F_3$}\label{triangle}
		\end{center}
	\end{figure}
	
	\smallskip
	\begin{claim}\label{3-dist}
		$\deg_{T}(t_4) = 2$.
	\end{claim}
	\begin{proof}[Proof of \cref{3-dist}]
		Suppose, to the contrary, that $\deg_{T}(t_4)= 3$, implying that triangle $F_4$ is an internal triangle of $G$. We have shown the shaded triangle $F_4$ of $G$ corresponding to vertex $t_4$ in Figure~\ref{triangle}. In the following, we consider two cases depending on whether $V(F_4)=\{u_2,u_5,u_6\}$ or $V(F_4)=\{u_4,u_5,u_6\}$.
		
		Suppose firstly that $V(F_4)=\{u_2,u_5,u_6\}$. Let $G' = G - \{u_1,u_3,u_4\}$ be a graph of order $n'$ obtained by deleting the vertices $u_1$, $u_3$ and $u_4$. The resulting graph $G'$ is a mop of order $n' = n - 3 \ge 10$ with $k-1$ number of vertices of degree~$2$ since $F_4$ is an internal triangle of $G$. The edge $u_2u_5$ is an outer edge of $G'$. Let $G_1$ be graph of order $n_1$ obtained from $G'$ by contracting the edge $u_2u_5$ to form a new vertex $x$ in $G_1$, and let $G_1$ have $k_1$ vertices of degree~$2$. By Lemma~\ref{key}, $G_1$ is a mop. Since $n' \ge 10$, we note that $n_1 = n' - 1 \ge 9$. By the minimality of the mop $G$, we have $\gamma_2^d(G_1) \le \floor*{\frac{2}{9}(n_1+k_1)} \le \floor*{\frac{2}{9}(n-4+k-1)} \le \floor*{\frac{2}{9}(n+k)}-1$. Let $D_1$ be a $\gamma_2^d$-set of $G_1$. If $x \in D_1$, then let $D = (D_1\setminus \{x\}) \cup \{u_2,u_5\}$. If $x \notin D_1$, then let $D = D_1 \cup \{u_2\}$. In both cases, the set $D$ is a 2DD-set of $G$, and so $\gamma_2^d(G) \le |D| \le |D_1| + 1 \le \floor*{\frac{2}{9}(n+k)}$, a contradiction.

		Hence, $V(F_4)=\{u_4,u_5,u_6\}$. We now let $G' = G - \{u_1,u_2,u_3\}$ be a graph of order $n'$ obtained by deleting the vertices $u_1$, $u_2$ and $u_3$. The resulting graph $G'$ is a mop of order $n' = n - 3 \ge 10$ with $k-1$ number of vertices of degree~$2$ since $F_4$ is an internal triangle of $G$. The edge $u_4u_5$ is an outer edge of $G'$. Let $G_1$ be graph of order $n_1$ obtained from $G'$ by contracting the edge $u_4u_5$ to form a new vertex $x$ in $G_1$, and let $G_1$ have $k_1$ vertices of degree~$2$. By Lemma~\ref{key}, $G_1$ is a mop. Since $n' \ge 10$, we note that $n_1 = n' - 1 \ge 9$. By the minimality of the mop $G$, we have $\gamma_2^d(G_1) \le \floor*{\frac{2}{9}(n_1+k_1)}  \le \floor*{\frac{2}{9}(n-4+k-1)} \le \floor*{\frac{2}{9}(n+k)}-1$. Let $D_1$ be a $\gamma_2^d$-set of $G_1$. If $x \in D_1$, then let $D = (D_1\setminus \{x\}) \cup \{u_4,u_5\}$. If $x \notin D_1$, then let $D = D_1 \cup \{u_2\}$. In both cases, the set $D$ is a 2DD-set of $G$, and so $\gamma_2^d(G) \le |D| \le |D_1| + 1 \le \floor*{\frac{2}{9}(n+k)}$, a contradiction.
	\end{proof}

	By \cref{3-dist}, we have $\deg_{T}(t_4) \ne 3$, implying that $\deg_{T}(t_4)=2$. Recall that either $V(F_4)=\{u_2,u_5,u_6\}$ or $V(F_4)=\{u_4,u_5,u_6\}$.
	
	\smallskip
	\begin{claim}\label{notf4}
		$V(F_4)=\{u_4,u_5,u_6\}$.
	\end{claim}
	\begin{proof}[Proof of Claim~\ref{notf4}]
		Suppose, to the contrary, that $V(F_4)=\{u_2,u_5,u_6\}$. Let $t_5 \in V(T)$ be the neighbor of $t_4$ in $T$ different from $t_3$, and let $F_5$ be the triangle in $G$ corresponding to the vertex $t_5$ in $T$. If $\deg_T(t_5)=1$, then $n=7$, a contradiction to the fact that $G$ is a mop of order $n \ge 13$. Hence, $2 \le \deg_{T}(t_5)\le 3$. Let $u_7$ be the vertex in $F_5$ that is not in $F_4$. We note that either $V(F_5) = \{u_2,u_6,u_7\}$ or $V(F_5)=\{u_5,u_6,u_7\}$ (see \Cref{3-distance}(a)-(b)). In the following, we consider two cases depending on whether $V(F_5)=\{u_2,u_6,u_7\}$ or $V(F_5)=\{u_5,u_6,u_7\}$.

		\begin{figure}[htb]
			\begin{center}
				\includegraphics[scale=.22]{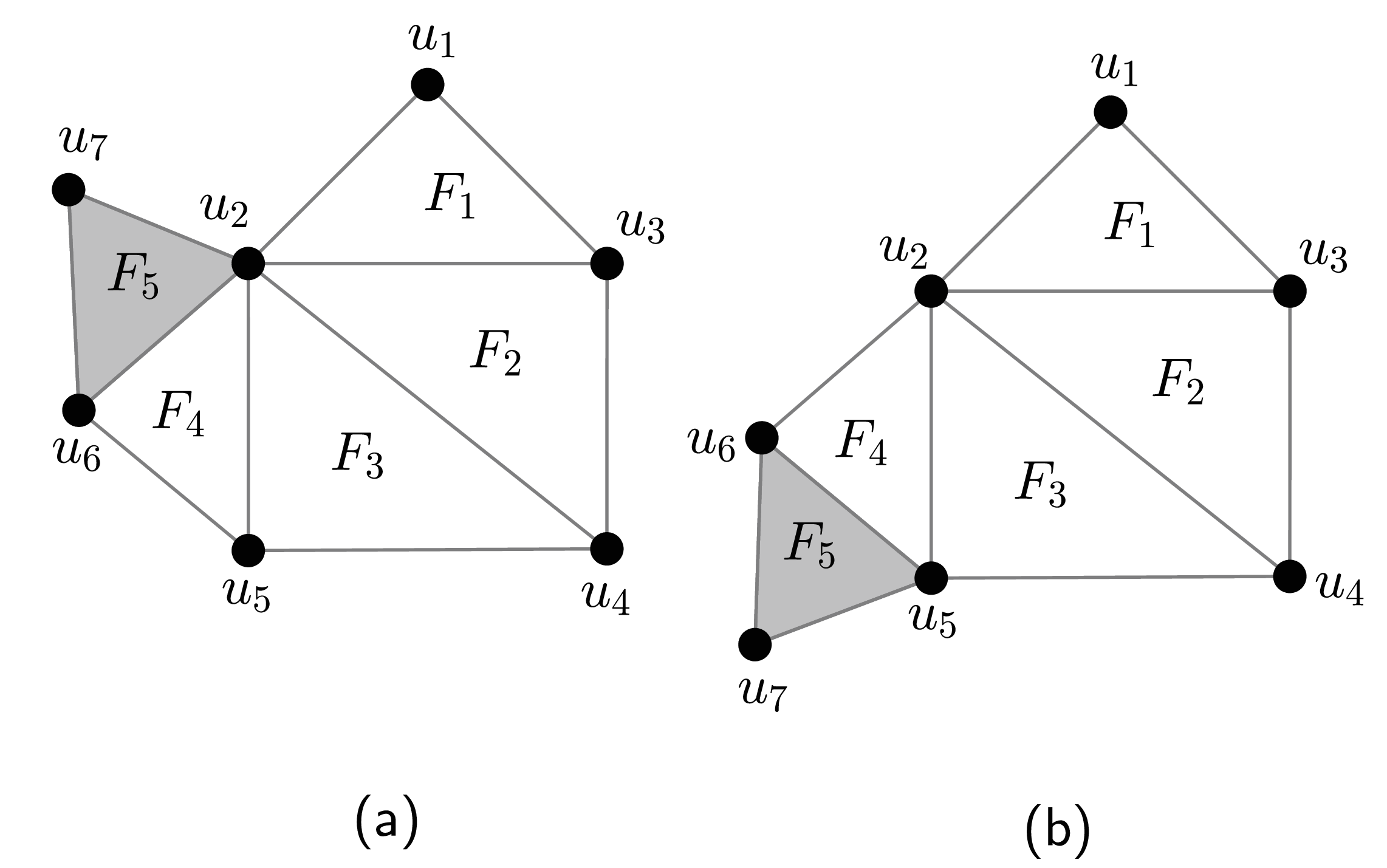}
				\caption{Possible (shaded) triangle adjacent to triangle $F_4$ when $V(F_4)=\{u_2,u_5,u_6\}$.}\label{3-distance}
			\end{center}
		\end{figure}

		Suppose firstly that $V(F_5)=\{u_2,u_6,u_7\}$. Let $G' = G - \{u_1,u_3,u_4,u_5\}$ and let $G'$ have order~$n'$. We note that $G'$ is a mop of order $n' = n - 4 \ge 9$. The edge $u_2u_6$ is an outer edge of $G'$. Let $G_1$ be a graph of order $n_1$ obtained from $G'$ by contracting the edge $u_2u_6$ to form a new vertex $x$ in $G_1$, and let $G_1$ have $k_1$ vertices of degree~$2$. By Lemma~\ref{key}, $G_1$ is a mop. Since $n' \ge 9$, we note that $n_1 = n' - 1 \ge 8$. By the minimality of the mop $G$, we have $\gamma_2^d(G_1) \le \floor*{\frac{2}{9}(n_1+k_1)}  \le \floor*{\frac{2}{9}(n-5+k)} \le \floor*{\frac{2}{9}(n+k)}-1$. Let $D_1$ be a $\gamma_2^d$-set of $G_1$. If $x \in D_1$, then let $D = (D_1\setminus \{x\}) \cup \{u_2,u_6\}$. If $x \notin D_1$, then let $D = D_1 \cup \{u_2\}$. In both cases, the set $D$ is a 2DD-set of $G$, and so $\gamma_2^d(G) \le |D| \le |D_1| + 1 \le \floor*{\frac{2}{9}(n+k)}$, a contradiction.

		Hence, $V(F_5)=\{u_5,u_6,u_7\}$. We now let $G' = G - \{u_1,u_2,u_3,u_4\}$ and let $G'$ have order~$n'$. We note that $G'$ is a mop of order $n' = n - 4 \ge 9$. The edge $u_5u_6$ is an outer edge of $G'$. Let $G_1$ be a graph of order $n_1$ obtained from $G'$ by contracting the edge $u_5u_6$ to form a new vertex $x$ in $G_1$, and let $G_1$ have $k_1$ vertices of degree~$2$. By Lemma~\ref{key}, $G_1$ is a mop. Since $n' \ge 9$, we note that $n_1 = n' - 1 \ge 8$. By the minimality of the mop $G$, we have $\gamma_2^d(G_1) \le \floor*{\frac{2}{9}(n_1+k_1)}  \le \floor*{\frac{2}{9}(n-5+k)} \le \floor*{\frac{2}{9}(n+k)}-1$. Let $D_1$ be a $\gamma_2^d$-set of $G_1$. If $x \in D_1$, then let $D = (D_1\setminus \{x\}) \cup \{u_5,u_6\}$. If $x \notin D_1$, then let $D = D_1 \cup \{u_2\}$. In both cases, the set $D$ is a 2DD-set of $G$, and so $\gamma_2^d(G) \le |D| \le |D_1| + 1 \le \floor*{\frac{2}{9}(n+k)}$, a contradiction.
	\end{proof}

	By \cref{notf4}, we have $V(F_4)=\{u_4,u_5,u_6\}$. Since $\deg_{T}(t_4)=2$, it has a neighbor different from $t_3$. Let $t_5\in T$ be the neighbor of $t_4$ different from $t_3$, and let $F_5$ be the triangle in $G$ corresponding to the vertex $t_5$ in $T$. Thus, $d_{T}(t_1,t_5)=4$. If $\deg_T(t_5)=1$, then $n=7$, a contradiction to the fact that $G$ is a mop of order $n \ge 13$. Hence, $2 \le \deg_{T}(t_5) \le 3$. Let $u_7$ be the vertex in $F_5$ that is not in $F_4$. We note that either $V(F_5)=\{u_4,u_6,u_7\}$ or $V(F_5)=\{u_5,u_6,u_7\}$ (see \Cref{4-distance}(a)-(b)).
	
	\smallskip
	\begin{claim}\label{4-dist}
		$\deg_{T}(t_5) = 2$.
	\end{claim}
	\begin{proof}[Proof of \cref{4-dist}]
		Suppose, to the contrary, that $\deg_{T}(t_5)= 3$, implying that triangle $F_5$ is an internal triangle of $G$. We have shown the shaded triangle $F_5$ of $G$ corresponding to vertex $t_5$ in \Cref{4-distance}(a)-(b). In the following, we consider two cases depending on whether $V(F_5)=\{u_4,u_6,u_7\}$ or $V(F_5)=\{u_5,u_6,u_7\}$.

		\begin{figure}[htb]
			\begin{center}
				\includegraphics[scale=.22]{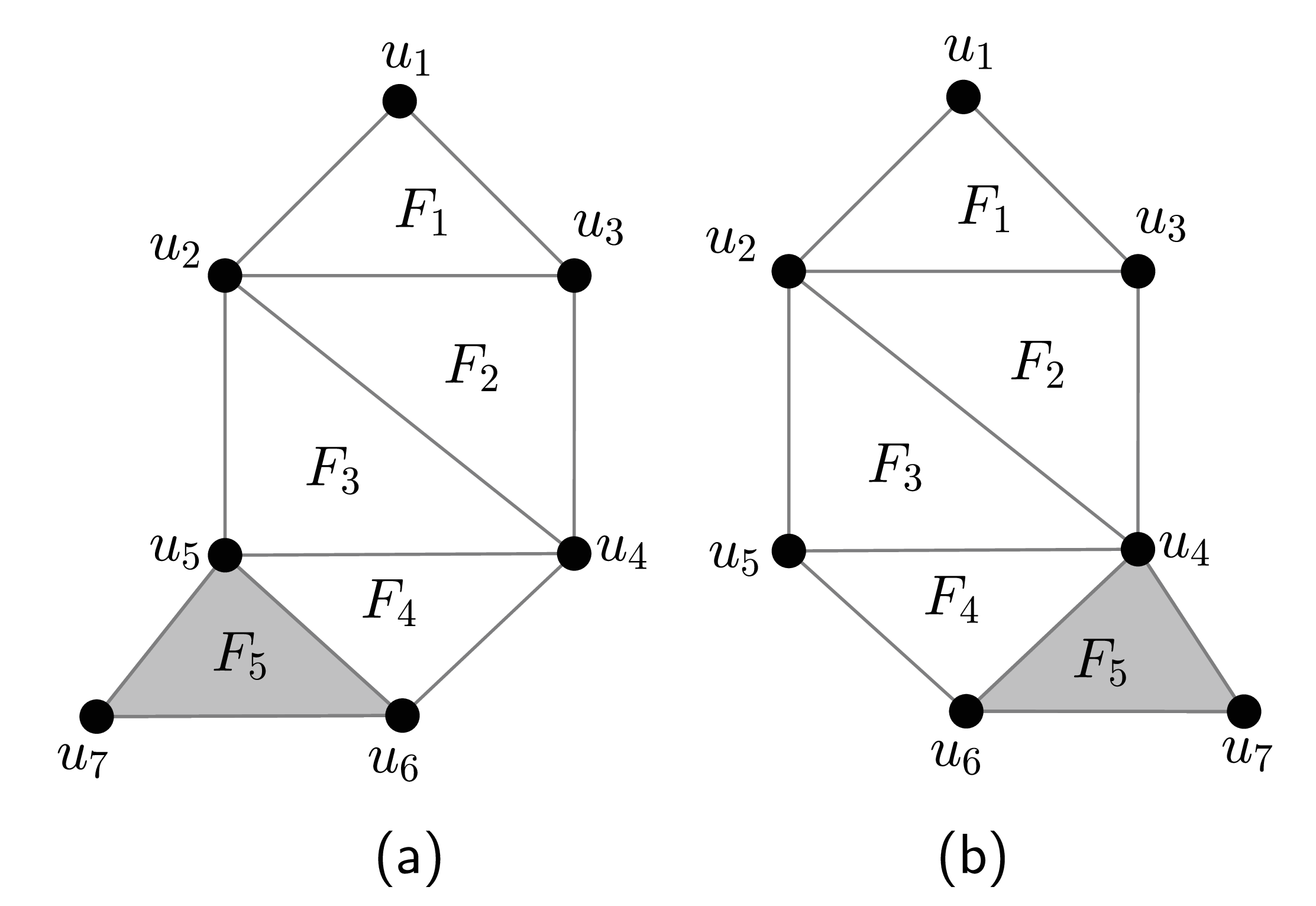}
				\caption{Possible (shaded) triangle adjacent to triangle $F_4$ when $V(F_4)=\{u_4,u_5,u_6\}$.}\label{4-distance}
			\end{center}
		\end{figure}

		If $V(F_5)=\{u_4,u_6,u_7\}$, then we let $G' = G - \{u_1,u_2,u_3,u_5\}$, and if $V(F_5)=\{u_5,u_6,u_7\}$, then we let $G' = G - \{u_1,u_2,u_3,u_4\}$. Let $G'$ have order $n'$ with $k'$ vertices of degree~$2$. In both cases, we note that $G'$ is a mop of order $n' = n - 4 \ge 9$ with $k' = k-1$ since $F_5$ is an internal triangle of $G$. By the minimality of the mop $G$, we have $\gamma_2^d(G') \le \floor*{\frac{2}{9}(n'+k')}  \le \floor*{\frac{2}{9}(n-4+k-1)} \le \floor*{\frac{2}{9}(n+k)}-1$. Let $D'$ be a $\gamma_2^d$-set of $G'$. In both cases, we let $D = D' \cup \{u_2\}$. The set $D$ is a 2DD-set of $G$, and so $\gamma_2^d(G) \le |D| \le |D'| + 1 \le \floor*{\frac{2}{9}(n+k)}$, a contradiction.
	\end{proof}

	By \cref{4-dist}, we have $\deg_{T}(t_5) = 2$. Let $t_6 \in V(T)$ be the neighbor of $t_5$ in $T$ different from $t_4$, and let $F_6$ be the triangle in $G$ corresponding to the vertex $t_6$ in $T$. Thus, $d_T(t_1,t_6)=5$. If $\deg_T(t_6)=1$, then $n=8$, a contradiction to the fact that $G$ is a mop of order $n \ge 13$. Hence, $2 \le \deg_{T}(t_6)\le 3$. Let $u_8$ be the vertex in $F_6$ that is not in $F_5$. We note that either $V(F_6)=\{u_5,u_7,u_8\}$ or  $V(F_6)=\{u_4,u_7,u_8\}$ or  $V(F_6)=\{u_6,u_7,u_8\}$ (see \Cref{5-distance}(a)-(d)).
	
	\begin{figure}[htb]
		\begin{center}
			\includegraphics[scale=.22]{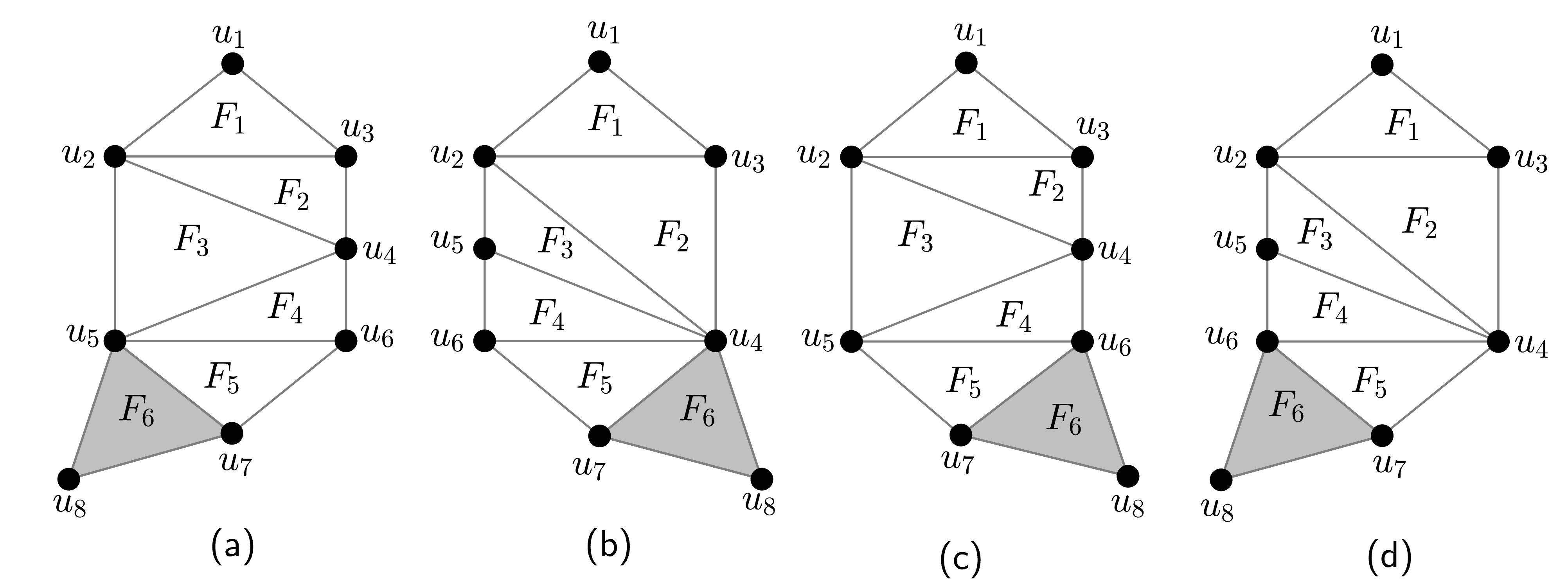}
			\caption{(a)~$H_1$, (b)~$H_2$, (c)~$H_3$, and (d)~$H_4$. Possible (shaded) triangle adjacent to triangle $F_5$.}\label{5-distance}
		\end{center}
	\end{figure}

	\begin{claim}\label{5-dist}
		$V(F_6) \ne \{u_6,u_7,u_8\}$.
	\end{claim}
	\begin{proof}[Proof of \cref{5-dist}]
		Suppose, to the contrary, that $V(F_6)=\{u_6,u_7,u_8\}$. There are two possible cases that may occur, as shown in \Cref{5-distance}(c)-(d).  In the following, we present arguments that work in both cases. Let $G'  = G - \{u_1,u_2,u_3,u_4,u_5\}$  and let $G'$ have order~$n'$ with $k'$ vertices of degree~$2$. We note that $G'$ is a mop of order $n' = n - 5 \ge 8$. The edge $u_6u_7$ is an outer edge of $G'$. We note that face $F_6$ may not be an internal triangle of $G$. By the minimality of the mop $G$, we have $\gamma_2^d(G') \le \floor*{\frac{2}{9}(n'+k')}  \le \floor*{\frac{2}{9}(n-5+k)} \le \floor*{\frac{2}{9}(n+k)}-1$. Let $D'$ be a $\gamma_2^d$-set of $G'$ and let $D = D' \cup \{u_2\}$. The set $D$ is a 2DD-set of $G$, and so $\gamma_2^d(G) \le |D| \le |D'| + 1 \le \floor*{\frac{2}{9}(n+k)}$, a contradiction.
	\end{proof}

	\begin{figure}[htbp]
		\begin{center}
			\includegraphics[scale=.24]{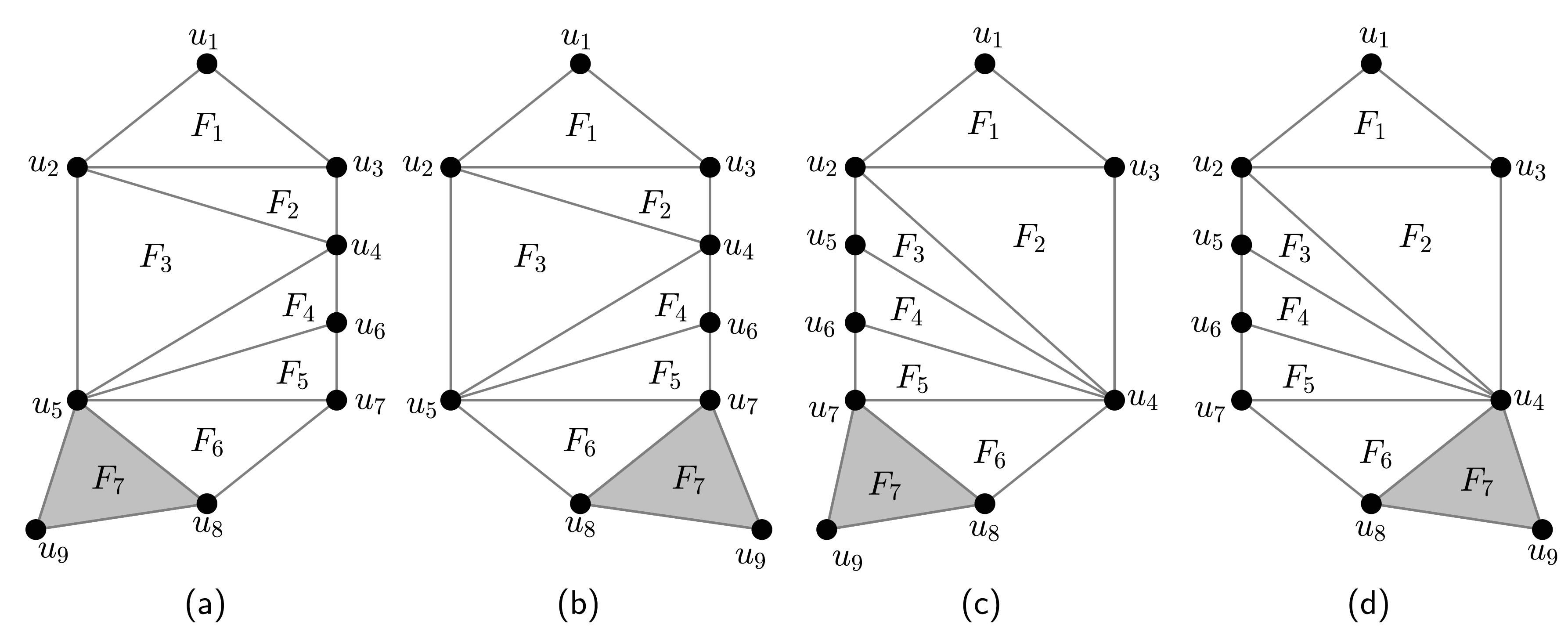}
			\caption{(a)~$H_5$, (b)~$H_6$, (c)~$H_7$, and (d)~$H_8$. Possible (shaded) triangle adjacent to triangle $F_6$.}\label{6-distance}
		\end{center}
	\end{figure}

	By \Cref{5-dist}, the triangle $F_6$ corresponding to vertex $t_6$ is either $V(F_6)=\{u_4,u_7,u_8\}$ or $V(F_6)=\{u_5,u_7,u_8\}$. If $\deg_T(t_6)=3$, then the subgraph of $G$ associated with the path between $t_1$ and $t_6$ corresponds to the region $H_1$ or $H_2$ illustrated in \Cref{5-distance}(a)-(b), and so \Cref{lem1}\ref{5distance} holds. Hence we may assume that $\deg_{T}(t_6) \ne 3$, implying that $\deg_{T}(t_6)=2$. Let $t_7 \in V(T)$ be the neighbor of $t_6$ in $T$ different from $t_5$, and let $F_7$ be the triangle in $G$ corresponding to the vertex $t_7$ in $T$. Thus, $d_T(t_1,t_7)=6$. If $\deg_T(t_7)=1$, then $n=9$, a contradiction to the fact that $G$ is a mop of order $n \ge 13$. Hence, $2 \le \deg_{T}(t_7)\le 3$. Let $u_9$ be the vertex in $F_7$ that is not in $F_6$. We note that either $V(F_7)=\{u_5,u_8,u_9\}$ or $V(F_7)=\{u_7,u_8,u_9\}$ or $V(F_7)=\{u_4,u_8,u_9\}$ (see \Cref{6-distance}(a)-(d)).

	If $\deg_T(t_7)=3$, then the subgraph of $G$ associated with the path between $t_1$ and $t_7$ corresponds to the region $H_5$, $H_6$, $H_7$, or $H_8$ illustrated in \Cref{6-distance}(a)-(d), and so \Cref{lem1}\ref{6distance} holds. Hence we may assume that $\deg_{T}(t_7) \ne 3$, implying that $\deg_{T}(t_7)=2$. Let $t_8 \in V(T)$ be the neighbor of $t_7$ in $T$ different from $t_9$, and let $F_8$ be the triangle in $G$ corresponding to the vertex $t_8$ in $T$. Thus, $d_T(t_1,t_8)=7$. If $\deg_T(t_8)=1$, then $n=10$, a contradiction to the fact that $G$ is a mop of order $n \ge 13$. Hence, $2 \le \deg_{T}(t_8)\le 3$. Let $u_{10}$ be the vertex in $F_8$ that is not in $F_7$. We note that either $V(F_8)=\{u_5,u_9,u_{10}\}$ or $V(F_8)=\{u_8,u_9,u_{10}\}$ or 	$V(F_8)=\{u_7,u_9,u_{10}\}$ or $V(F_8)=\{u_4,u_9,u_{10}\}$ (see \Cref{7-distance}(a)-(h)).
	
	\begin{figure}[htb]
		\begin{center}
			\includegraphics[scale=.24]{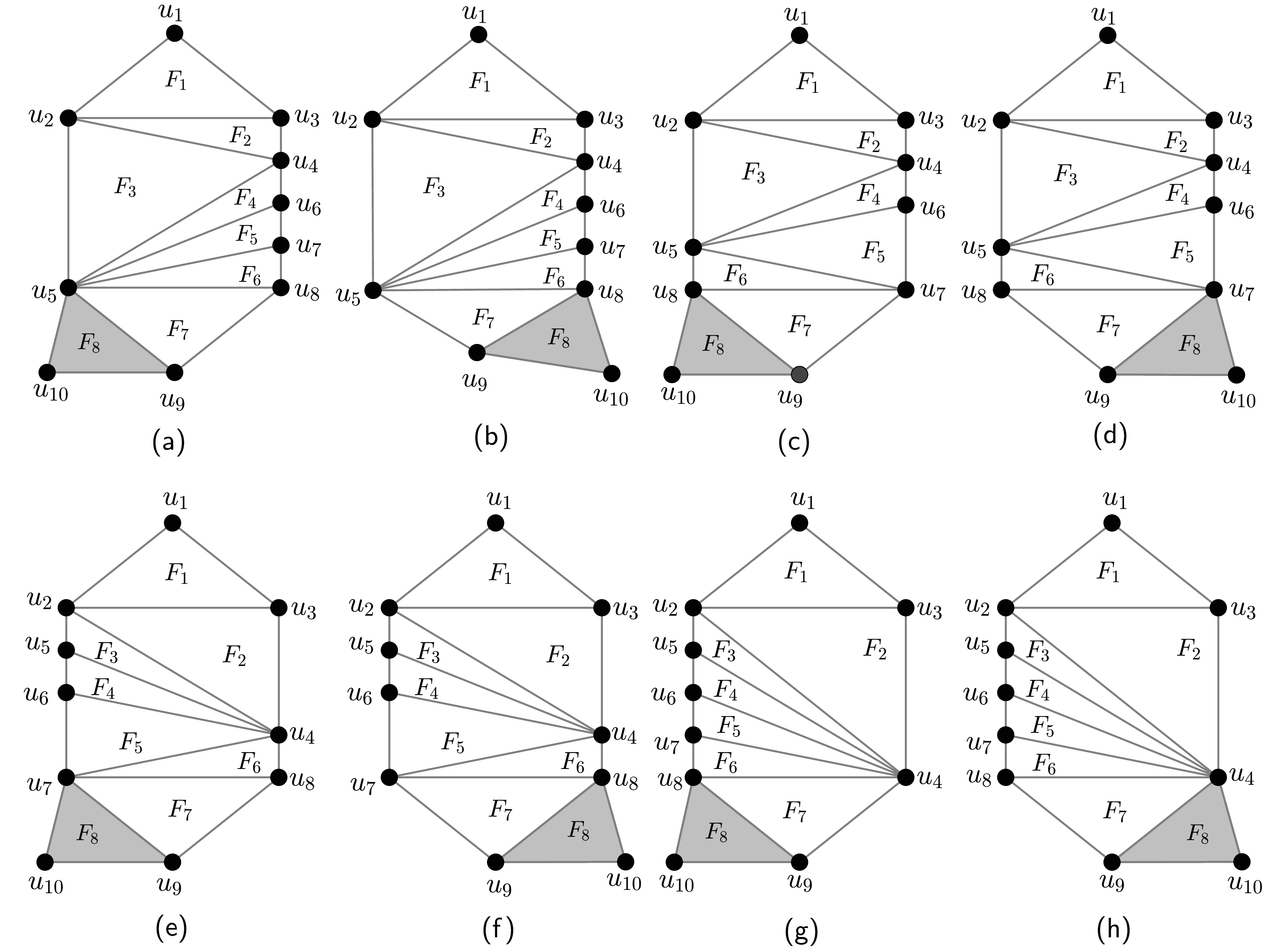}
			\caption{Possible (shaded) triangle adjacent to triangle $F_7$.}\label{7-distance}
		\end{center}
	\end{figure}
	
	\begin{claim}\label{7-dist}
		$\deg_{T}(t_8) = 2$.
	\end{claim}
	\begin{proof}[Proof of \Cref{7-dist}]
		Suppose, to the contrary, that $\deg_{T}(t_8)= 3$, implying that the triangle $F_8$ is an internal triangle of $G$. We have shown the shaded triangle $F_8$ of $G$ corresponding to vertex $t_8$ in \Cref{7-distance}(a)-(h). In the following, we consider four cases depending on whether $V(F_8)=\{u_5,u_9,u_{10}\}$ or $V(F_8)=\{u_8,u_9,u_{10}\}$ or $V(F_8)=\{u_7,u_9,u_{10}\}$ or $V(F_8)=\{u_4,u_9,u_{10}\}$.

		Suppose firstly that $V(F_8)=\{u_5,u_9,u_{10}\}$. Let $G' = G - \{u_1,u_2,u_3,u_4,u_6,u_7,u_8\}$ and let $G'$ have order~$n'$ with $k'$ vertices of degree~$2$. We note that $G'$ is a mop of order $n' = n - 7 \ge 6$. Moreover, $k' = k-1$ since $F_8$ is an internal triangle of $G$. The edge $u_5u_9$ is an outer edge of $G'$. If $6 \le n' \le 7$, then $\{u_2,u_5,u_{10}\}$ or $\{u_2,u_9,u_{10}\}$ is a 2DD-set of $G$, and so $\gamma_2^d(G) \le 3 \le  \floor*{\frac{2}{9}(n+k)}$, a contradiction. Hence, $n' \ge 8$. Let $G_1$ be a graph of order $n_1$ obtained from $G'$ by contracting the edge $u_5u_9$ to form a new vertex $x$ in $G_1$, and let $G_1$ have $k_1$ vertices of degree~$2$. By Lemma~\ref{key}, $G_1$ is a mop. Since $n' \ge 8$, we note that $n_1 = n' - 1 \ge 7$. By the minimality of the mop $G$, we have $\gamma_2^d(G_1) \le \floor*{\frac{2}{9}(n_1+k_1)} \le \floor*{\frac{2}{9}(n-8+k-1)} \le \floor*{\frac{2}{9}(n+k)}-2$. Let $D_1$ be a $\gamma_2^d$-set of $G_1$. If $x \in D_1$, then let $D = (D_1\setminus \{x\}) \cup \{u_2,u_5,u_9\}$. If $x \notin D_1$, then let $D = D_1 \cup \{u_2,u_5\}$. In both cases, the set $D$ is a 2DD-set of $G$, and so $\gamma_2^d(G) \le |D| \le |D_1| + 2 \le \floor*{\frac{2}{9}(n+k)}$, a contradiction.
		
		Suppose secondly that $V(F_8)=\{u_8,u_9,u_{10}\}$. There are four possible cases that may occur, as shown in \Cref{7-distance}(b),(c),(f), and (g).  In the following, we present arguments that work in each case. Let $G' = G - \{u_1,u_2,u_3,u_4,u_5,u_6,u_7\}$ and let $G'$ have order~$n'$ with $k'$ vertices of degree~$2$. We note that $G'$ is a mop of order $n' = n - 7 \ge 6$. Moreover, $k' = k-1$ since $F_8$ is an internal triangle of $G$. The edge $u_8u_9$ is an outer edge of $G'$. If $6 \le n' \le 7$, then $\{u_2,u_8,u_{10}\}$ or $\{u_2,u_9,u_{10}\}$ is a 2DD-set of $G$, and so $\gamma_2^d(G)\le 3 \le  \floor*{\frac{2}{9}(n+k)}$, a contradiction. Hence $n' \ge 8$. Let $G_1$ be a graph of order $n_1$ obtained from $G'$ by contracting the edge $u_8u_9$ to form a new vertex $x$ in $G_1$. By Lemma~\ref{key}, $G_1$ is a mop. Since $n' \ge 8$, we note that $n_1 = n' - 1 \ge 7$. By the minimality of the mop $G$, we have $\gamma_2^d(G_1) \le \floor*{\frac{2}{9}(n_1+k_1)} \le \floor*{\frac{2}{9}(n-8+k-1)} \le \floor*{\frac{2}{9}(n+k)}-2$. Let $D_1$ be a $\gamma_2^d$-set of $G_1$. If $x \in D_1$, then let $D = (D_1\setminus \{x\}) \cup \{u_2,u_8,u_9\}$. If $x \notin D_1$, then let $D = D_1 \cup \{u_2,u_8\}$. In both cases, the set $D$ is a 2DD-set of $G$, and so $\gamma_2^d(G) \le |D| \le |D_1| + 2 \le \floor*{\frac{2}{9}(n+k)}$, a contradiction.

		Suppose next that $V(F_8)=\{u_7,u_9,u_{10}\}$. There are two possible cases that may occur, as shown in \Cref{7-distance}(d)-(e).  In the following, we present arguments that work in both cases. Let $G' = G - \{u_1,u_2,u_3,u_4,u_5,u_6,u_8\}$ and let $G'$ have order~$n'$ with $k'$ vertices of degree~$2$. We note that $G'$ is a mop of order $n' = n - 7 \ge 6$. Moreover, $k' = k-1$ since $F_8$ is an internal triangle of $G$. The edge $u_8u_9$ is an outer edge of $G'$. If $6 \le n' \le 7$, then $\{u_2,u_7,u_{10}\}$ or $\{u_2,u_9,u_{10}\}$ is a 2DD-set of $G$, and so $\gamma_2^d(G)\le 3 \le  \floor*{\frac{2}{9}(n+k)}$, a contradiction. Hence $n' \ge 8$. Let $G_1$ be a graph of order $n_1$ obtained from $G'$ by contracting the edge $u_8u_9$ to form a new vertex $x$ in $G_1$. By Lemma~\ref{key}, $G_1$ is a mop. Since $n' \ge 8$, we note that $n_1 = n' - 1 \ge 7$. By the minimality of the mop $G$, we have $\gamma_2^d(G_1) \le \floor*{\frac{2}{9}(n_1+k_1)} \le \floor*{\frac{2}{9}(n-8+k-1)} \le \floor*{\frac{2}{9}(n+k)}-2$. Let $D_1$ be a $\gamma_2^d$-set of $G_1$. If $x \in D_1$, then let $D = (D_1\setminus \{x\}) \cup \{u_2,u_7,u_9\}$. If $x \notin D_1$, then let $D = D_1 \cup \{u_2,u_7\}$. In both cases, the set $D$ is a 2DD-set of $G$, and so $\gamma_2^d(G) \le |D| \le |D_1| + 2 \le \floor*{\frac{2}{9}(n+k)}$, a contradiction.
		
		Suppose finally that $V(F_8)=\{u_4,u_9,u_{10}\}$ (see \Cref{7-distance}(h)). Let $G' = G - \{u_1,u_2,u_3,u_5,u_6,u_7,u_8\}$ and let $G'$ have order~$n'$ with $k'$ vertices of degree~$2$. We note that $G'$ is a mop of order $n' = n - 7 \ge 6$. Moreover, $k' = k-1$ since $F_8$ is an internal triangle of $G$. The edge $u_4u_9$ is an outer edge of $G'$. If $6 \le n' \le 7$, then $\{u_2,u_4,u_{10}\}$ or $\{u_2,u_9,u_{10}\}$ is a 2DD-set of $G$, and so $\gamma_2^d(G)\le 3 \le  \floor*{\frac{2}{9}(n+k)}$, a contradiction. Hence $n' \ge 8$. Let $G_1$ be a graph of order $n_1$ obtained from $G'$ by contracting the edge $u_4u_9$ to form a new vertex $x$ in $G_1$. By Lemma~\ref{key}, $G_1$ is a mop. Since $n' \ge 8$, we note that $n_1 = n' - 1 \ge 7$. By the minimality of the mop $G$, we have $\gamma_2^d(G_1) \le \floor*{\frac{2}{9}(n_1+k_1)} \le \floor*{\frac{2}{9}(n-8+k-1)} \le \floor*{\frac{2}{9}(n+k)}-2$. Let $D_1$ be a $\gamma_2^d$-set of $G_1$. If $x \in D_1$, then let $D = (D_1\setminus \{x\}) \cup \{u_2,u_4,u_9\}$. If $x \notin D_1$, then let $D = D_1 \cup \{u_2,u_4\}$. In both cases, the set $D$ is a 2DD-set of $G$, and so $\gamma_2^d(G) \le |D| \le |D_1| + 2 \le \floor*{\frac{2}{9}(n+k)}$, a contradiction.	
	\end{proof}
	
	\begin{figure}[htbp]
		\begin{center}
			\vspace{-0cm}
			\includegraphics[scale=.25]{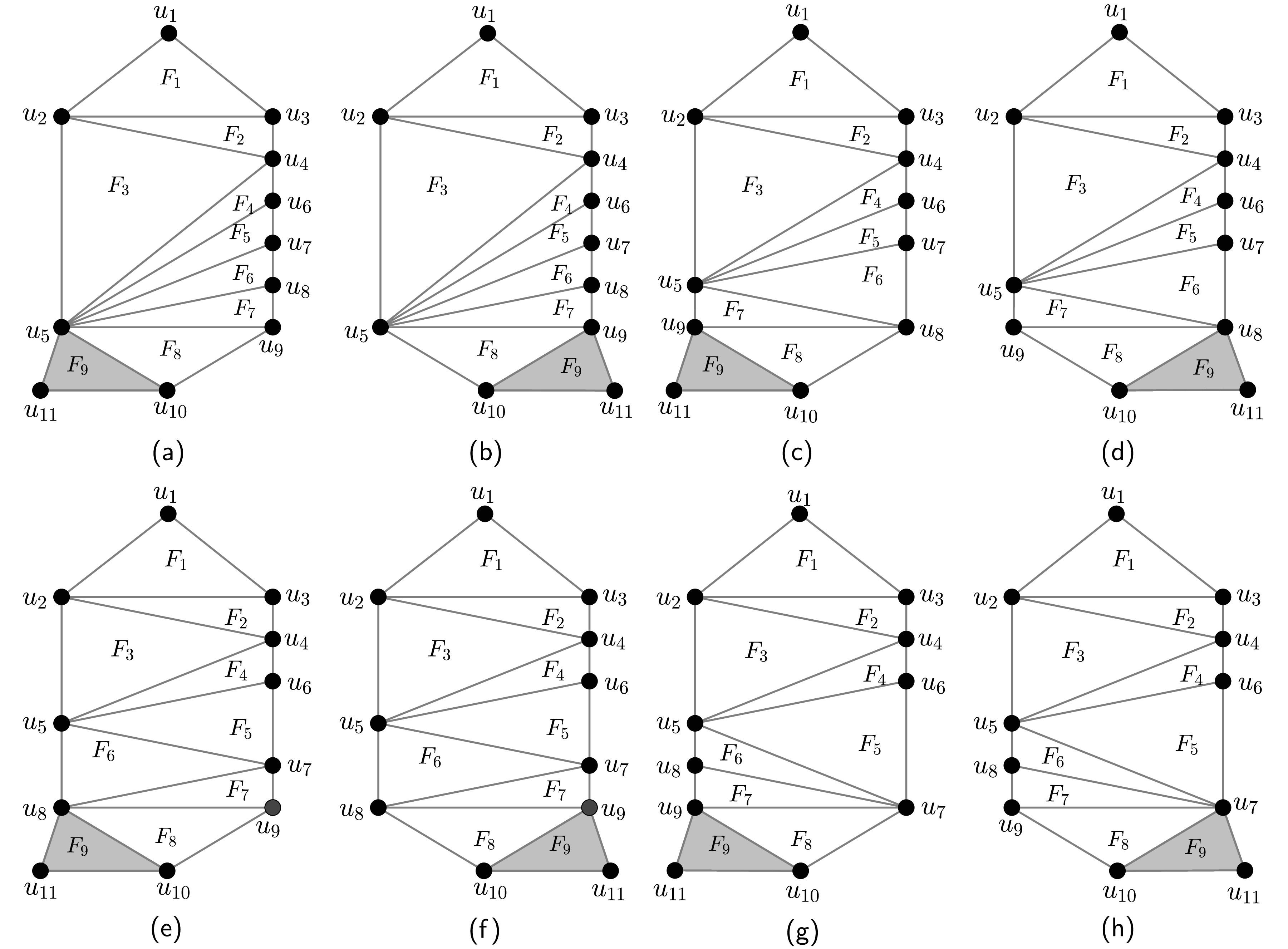}
			\includegraphics[scale=.25]{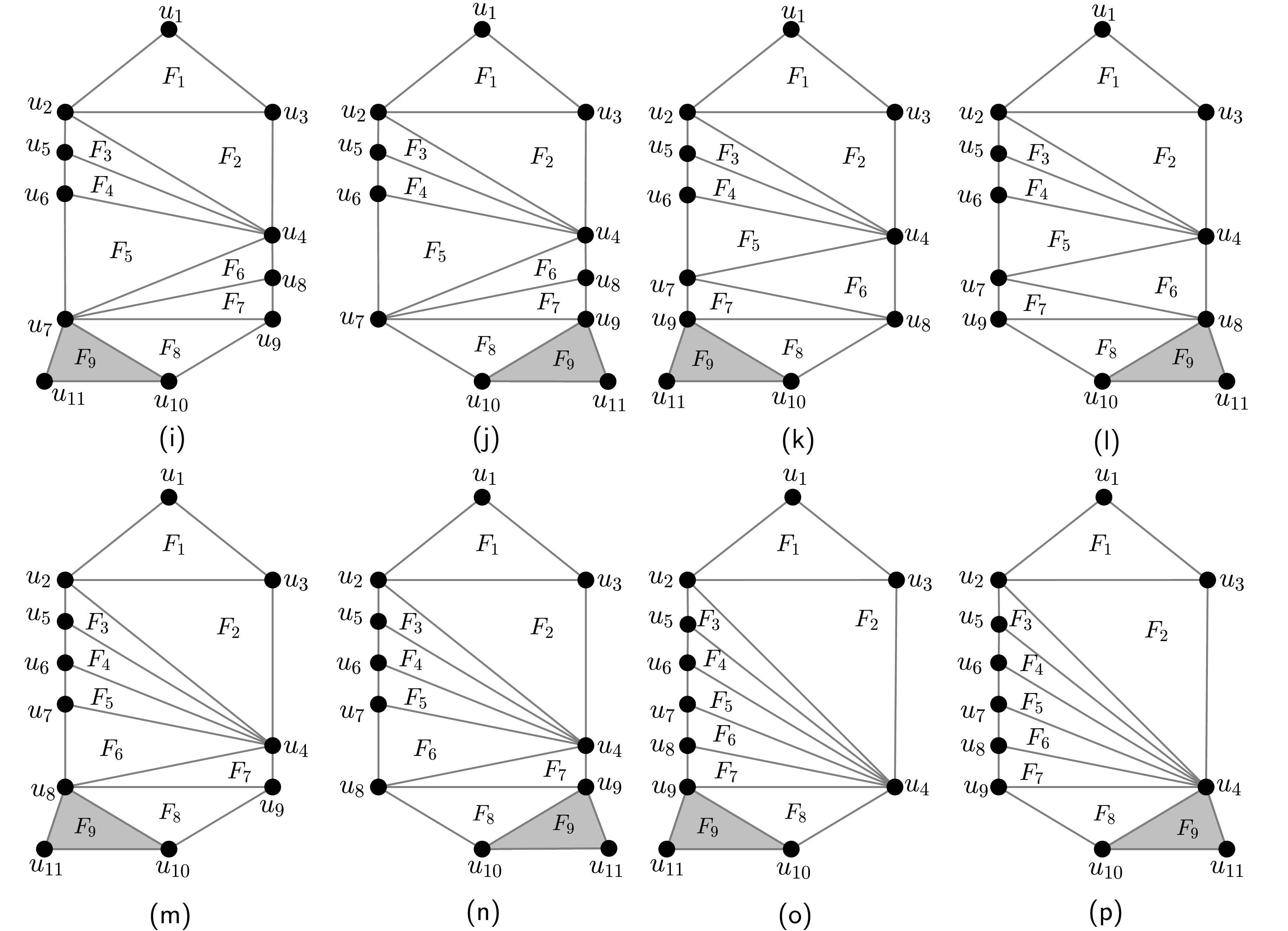}
			\caption{Possible (shaded) triangle adjacent to triangle $F_8$.}\label{8-distance}
		\end{center}
	\end{figure}

	By \cref{7-dist}, we have $\deg_{T}(t_8)\ne 3$, and so $\deg_{T}(t_8)=2$. Let $t_9 \in V(T)$ be the neighbor of $t_8$ different from $t_7$, and let $F_9$ be the triangle in $G$ corresponding to the vertex $t_9$ in $T$. Thus, $d_T(t_1,t_9)=8$. If $\deg_T(t_9)=1$, then $n=11$, a contradiction to the fact that $G$ is a mop of order $n \ge 13$. Hence, $2 \le \deg_{T}(t_9)\le 3$. Let $u_{11}$ be the vertex in $F_9$ that is not in $F_8$. We note that either $V(F_9)=\{u_5,u_{10},u_{11}\}$ or 	$V(F_9)=\{u_9,u_{10},u_{11}\}$ or $V(F_9)=\{u_8,u_{10},u_{11}\}$ or $V(F_9)=\{u_7,u_{10},u_{11}\}$ or $V(F_9)=\{u_4,u_{10},u_{11}\}$ (see \Cref{8-distance}(a)-(p)). In the following, we consider each of these five cases in turn. We show that each case yields a contradiction. We note that $F_9$ may not be an internal triangle of $G$. This shows that $T$ is not a path graph.
	
	\smallskip
	\emph{Case 1. $V(F_9)=\{u_5,u_{10},u_{11}\}$.} Let $G' = G - \{u_1,u_2,u_3,u_4,u_6,u_7,u_8,u_9\}$ and let $G'$ have order~$n'$ with $k'$ vertices of degree~$2$. We note that $G'$ is a mop of order $n' = n - 8 \ge 5$ and $k' \le k$. The edge $u_5u_{10}$ is an outer edge of $G'$. If $5 \le n' \le 7$, then by \cref{obs3}, there exists a 2DD-set $D'$ of $G'$ such that $u_5 \in D'$ and $|D'| = 2$. Therefore, $D' \cup \{u_2\}$ is a 2DD-set of $G$, and so $\gamma_2^d(G) \le 3\le  \floor*{\frac{2}{9}(n+k)}$, a contradiction. Hence, $n' \ge 8$. Let $G_1$ be a graph of order $n_1$ obtained from $G'$ by contracting the edge $u_5u_{10}$ to form a new vertex $x$ in $G_1$. By Lemma~\ref{key}, $G_1$ is a mop. Since $n' \ge 8$, we note that $n_1 = n' - 1 \ge 7$. Moreover, $n_1 = n - 9$. By the minimality of the mop $G$, we have $\gamma_2^d(G_1) \le \floor*{\frac{2}{9}(n_1+k_1)} \le \floor*{\frac{2}{9}(n-9+k)} \le \floor*{\frac{2}{9}(n+k)}-2$. Let $D_1$ be a $\gamma_2^d$-set of $G_1$. If $x \in D_1$, then let $D = (D_1\setminus \{x\}) \cup \{u_2,u_5,u_{10}\}$. If $x \notin D_1$, then let $D = D_1 \cup \{u_2,u_5\}$. In both cases, the set $D$ is a 2DD-set of $G$, and so $\gamma_2^d(G) \le |D| \le |D_1| + 2 \le \floor*{\frac{2}{9}(n+k)}$, a contradiction.
	
	\smallskip
	\emph{Case 2. $V(F_9)=\{u_9,u_{10},u_{11}\}$.} There are eight possible cases that may occur, as shown in \Cref{7-distance}(b),(c),(f),(g),(j),(k),(n), and (o). In the following, we present arguments that work in each case. Let $G' = G - \{u_1,u_2,u_3,u_4,u_5,u_6,u_7,u_8\}$ and let $G'$ have order~$n'$ with $k'$ vertices of degree~$2$. We note that $G'$ is a mop of order $n' = n - 8 \ge 5$ and $k' \le k$. The edge $u_9u_{10}$ is an outer edge of $G'$. If $5 \le n' \le 7$, then by \cref{obs3}, there exists a 2DD-set $D'$ of $G'$ such that $u_9\in D'$ and $|D'|=2$. Therefore, $D' \cup \{u_2\}$ is a 2DD-set of $G$, and so $\gamma_2^d(G)\le 3\le  \floor*{\frac{2}{9}(n+k)}$, a contradiction. Hence, $n' \ge 8$. Let $G_1$ be a graph of order $n_1$ obtained from $G'$ by contracting the edge $u_9u_{10}$ to form a new vertex $x$ in $G_1$. By Lemma~\ref{key}, $G_1$ is a mop. Since $n' \ge 8$, we note that $n_1 = n' - 1 \ge 7$. Moreover, $n_1 = n - 9$. By the minimality of the mop $G$, we have $\gamma_2^d(G_1) \le \floor*{\frac{2}{9}(n_1+k_1)} \le \floor*{\frac{2}{9}(n-9+k)} \le \floor*{\frac{2}{9}(n+k)}-2$. Let $D_1$ be a $\gamma_2^d$-set of $G_1$. If $x \in D_1$, then let $D = (D_1\setminus \{x\}) \cup \{u_2,u_9,u_{10}\}$. If $x \notin D_1$, then let $D = D_1 \cup \{u_2,u_9\}$. In both cases, the set $D$ is a 2DD-set of $G$, and so $\gamma_2^d(G) \le |D| \le |D_1| + 2 \le \floor*{\frac{2}{9}(n+k)}$, a contradiction.
	
	\smallskip
	\emph{Case 3. $V(F_9)=\{u_8,u_{10},u_{11}\}$.} There are four possible cases that may occur, as shown in \Cref{7-distance}(d),(e),(l), and (m). In the following, we present arguments that work in each case. Let $G' = G - \{u_1,u_2,u_3,u_4,u_5,u_6,u_7,u_9\}$ and let $G'$ have order~$n'$ with $k'$ vertices of degree~$2$. We note that $G'$ is a mop of order $n' = n - 8 \ge 5$ and $k' \le k$. The edge $u_8u_{10}$ is an outer edge of $G'$. If $5 \le n' \le 7$, then by \cref{obs3}, there exists a 2DD-set $D'$ of $G'$ such that $u_8\in D'$ and $|D'|=2$. Therefore, $D' \cup \{u_2\}$ is a 2DD-set of $G$, and so $\gamma_2^d(G)\le 3\le  \floor*{\frac{2}{9}(n+k)}$, a contradiction. Hence, $n' \ge 8$. Let $G_1$ be a graph of order $n_1$ obtained from $G'$ by contracting the edge $u_8u_{10}$ to form a new vertex $x$ in $G_1$. By Lemma~\ref{key}, $G_1$ is a mop. Since $n' \ge 8$, we note that $n_1 = n' - 1 \ge 7$. Moreover, $n_1 = n - 9$. By the minimality of the mop $G$, we have $\gamma_2^d(G_1) \le \floor*{\frac{2}{9}(n_1+k_1)} \le \floor*{\frac{2}{9}(n-9+k)} \le \floor*{\frac{2}{9}(n+k)}-2$. Let $D_1$ be a $\gamma_2^d$-set of $G_1$. If $x \in D_1$, then let $D = (D_1\setminus \{x\}) \cup \{u_2,u_8,u_{10}\}$. If $x \notin D_1$, then let $D = D_1 \cup \{u_2,u_8\}$. In both cases, the set $D$ is a 2DD-set of $G$, and so $\gamma_2^d(G) \le |D| \le |D_1| + 2 \le \floor*{\frac{2}{9}(n+k)}$, a contradiction.
	
	\smallskip
	\emph{Case 4. $V(F_9)=\{u_7,u_{10},u_{11}\}$.} There are two possible cases that may occur, as shown in \Cref{7-distance}(h)-(i). In the following, we present arguments that work in both cases. Let $G' = G - \{u_1,u_2,u_3,u_4,u_5,u_6,u_8,u_9\}$ and let $G'$ have order~$n'$ with $k'$ vertices of degree~$2$. We note that $G'$ is a mop of order $n' = n - 8 \ge 5$ and $k' \le k$. The edge $u_7u_{10}$ is an outer edge of $G'$. If $5 \le n' \le 7$, then by \cref{obs3}, there exists a 2DD-set $D'$ of $G'$ such that $u_7\in D'$ and $|D'|=2$. Therefore, $D' \cup \{u_2\}$ is a 2DD-set of $G$, and so $\gamma_2^d(G)\le 3\le  \floor*{\frac{2}{9}(n+k)}$, a contradiction. Hence, $n' \ge 8$. Let $G_1$ be a graph of order $n_1$ obtained from $G'$ by contracting the edge $u_7u_{10}$ to form a new vertex $x$ in $G_1$. By Lemma~\ref{key}, $G_1$ is a mop. Since $n' \ge 8$, we note that $n_1 = n' - 1 \ge 7$. Moreover, $n_1 = n - 9$. By the minimality of the mop $G$, we have $\gamma_2^d(G_1) \le \floor*{\frac{2}{9}(n_1+k_1)} \le \floor*{\frac{2}{9}(n-9+k)} \le \floor*{\frac{2}{9}(n+k)} - 2$. Let $D_1$ be a $\gamma_2^d$-set of $G_1$. If $x \in D_1$, then let $D = (D_1\setminus \{x\}) \cup \{u_2,u_7,u_{10}\}$. If $x \notin D_1$, then let $D = D_1 \cup \{u_2,u_7\}$. In both cases, the set $D$ is a 2DD-set of $G$, and so $\gamma_2^d(G) \le |D| \le |D_1| + 2 \le \floor*{\frac{2}{9}(n+k)}$, a contradiction.

	\smallskip
	\emph{Case 5. $V(F_9)=\{u_4,u_{10},u_{11}\}$.} This case is illustrated in \Cref{7-distance}(p). We now consider the graph $G'= G - \{u_1,u_2,u_3,u_5,u_6,u_7,u_8,u_9\}$ and let $G'$ have order~$n'$ with $k'$ vertices of degree~$2$. We note that $G'$ is a mop of order $n' = n - 8 \ge 5$ and $k' \le k$. The edge $u_4u_{10}$ is an outer edge of $G'$. If $5 \le n' \le 7$, then by \cref{obs3}, there exists a 2DD-set $D'$ of $G'$ such that $u_4\in D'$ and $|D'|=2$. Therefore, $D' \cup \{u_2\}$ is a 2DD-set of $G$, and so $\gamma_2^d(G)\le 3\le  \floor*{\frac{2}{9}(n+k)}$, a contradiction. Hence, $n' \ge 8$. Let $G_1$ be a graph of order $n_1$ obtained from $G'$ by contracting the edge $u_4u_{10}$ to form a new vertex $x$ in $G_1$. By Lemma~\ref{key}, $G_1$ is a mop. Since $n' \ge 8$, we note that $n_1 = n' - 1 \ge 7$. Moreover, $n_1 = n - 9$. By the minimality of the mop $G$, we have $\gamma_2^d(G_1) \le \floor*{\frac{2}{9}(n_1+k_1)} \le \floor*{\frac{2}{9}(n-9+k)} \le \floor*{\frac{2}{9}(n+k)}-2$. Let $D_1$ be a $\gamma_2^d$-set of $G_1$. If $x \in D_1$, then let $D = (D_1\setminus \{x\}) \cup \{u_2,u_4,u_{10}\}$. If $x \notin D_1$, then let $D = D_1 \cup \{u_2,u_4\}$. In both cases, the set $D$ is a 2DD-set of $G$, and so $\gamma_2^d(G) \le |D| \le |D_1| + 2 \le \floor*{\frac{2}{9}(n+k)}$, a contradiction.
	
	This completes the proof of \Cref{lem1}.
\end{proof}

\section{Proof of main result}
\label{Sect:main-proof}

In this section, we present a proof our main result, namely Theorem~\ref{thm:main}. Recall its statement.

\noindent \textbf{Theorem~\ref{thm:main}} \emph{
	If $G$ is a mop of order $n \ge 7$ with $k$ vertices of degree~$2$, then $\gamma_2^d(G) \le \floor*{\frac{2}{9}(n+k)}$.
}

\noindent
\begin{proof}
	If $7 \le n\le 12$, then by \cref{7to12mops}, $\gamma_2^d(G) \le  \floor*{\frac{2}{9}(n+k)}$. Hence we may assume that $G$ is a mop of order $n \ge 13$. Suppose, to the contrary, that there exists a counterexample to our theorem. With this supposition, let $G$ be a counterexample of minimum order $n \ge 13$ and let $G$ have $k$ vertices of degree~$2$. Since $G$ is a counterexample of minimum order, the mop $G$ satisfies $\gamma_2^d(G) > \floor*{\frac{2}{9}(n+k)}$. Furthermore, if $G'$ is a mop of order~$n'$ where $7 \le n' < n$ and with $k'$ vertices of degree~$2$, then $\gamma_2^d(G') \le \lfloor \frac{2}{9}(n'+k') \rfloor$.

	\begin{figure}[htbp]
		\begin{center}
			\includegraphics[scale=0.69]{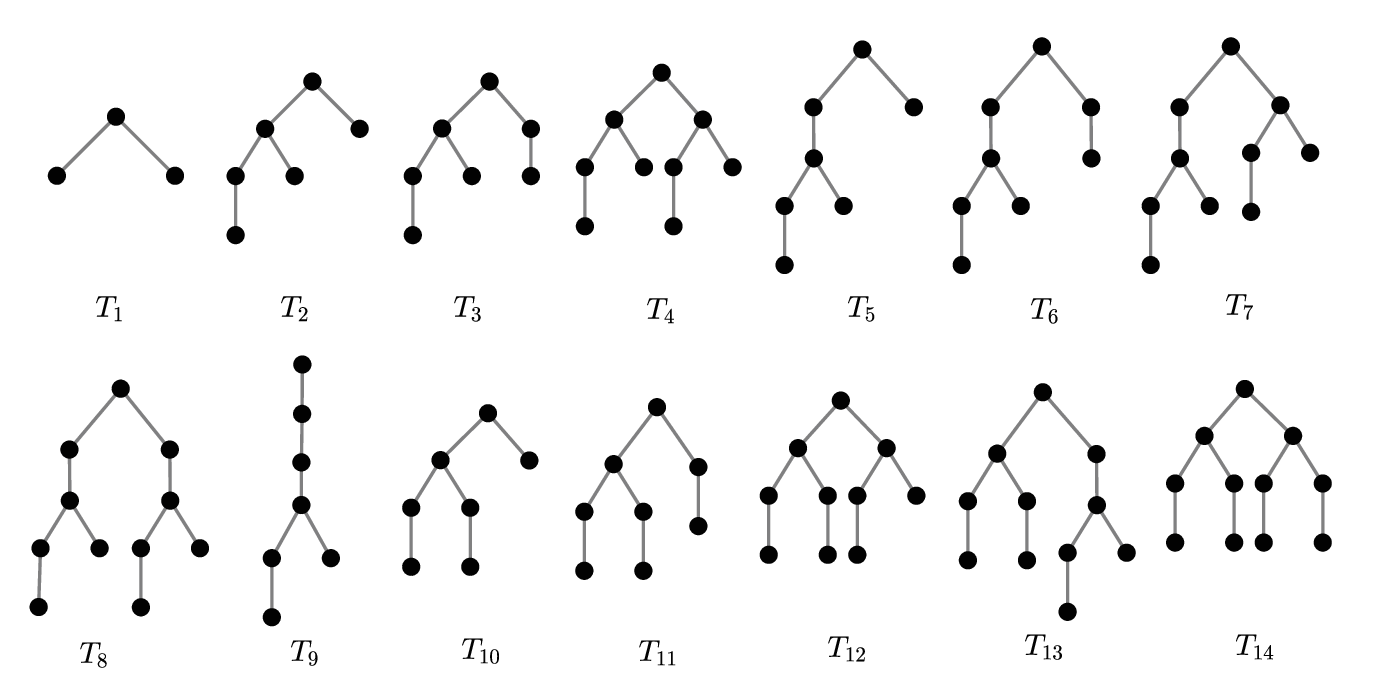}
			\includegraphics[scale=0.7]{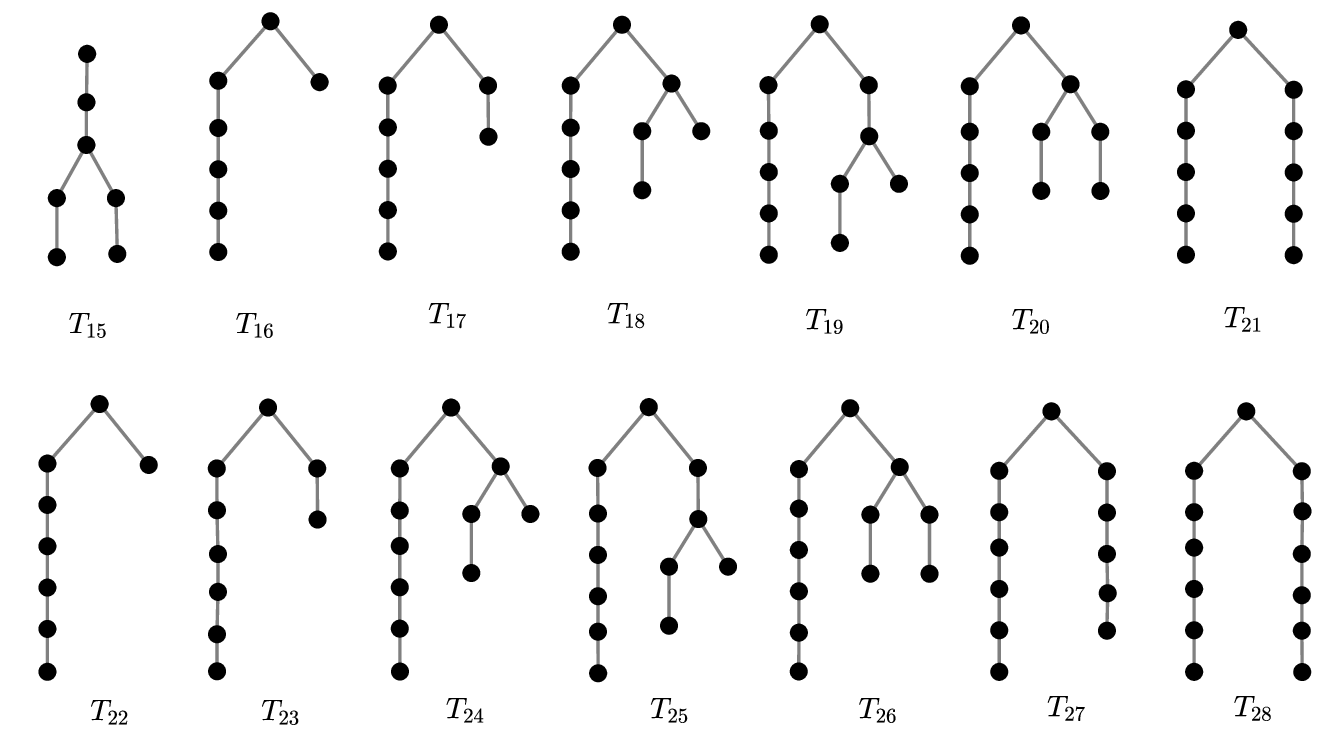}
			\caption{Trees.}\label{trees}
		\end{center}
	\end{figure}
	
	Let $T$ be the tree associated with the mop $G$. By \Cref{lem1}\ref{npath}, $T$ is not a path graph, implying that $T$ has at least one vertex of degree~$3$. We now root the tree $T$ at a leaf~$r$ that belongs to a longest path $P$ in $T$. Recall that such a vertex~$r$ is called a diametrical vertex of $T$. Further recall that if $v$ is a vertex in the rooted tree $T$, then we denote by $T_v$ the maximal subtree rooted at $v$, that is, $T_v$ is the subtree of $T$ induced by $D[v]$. Let $x$ be an arbitrary leaf of $T$ and let $y$ be a nearest vertex of degree~$3$ to $x$ in $T$. Then by \cref{lem1}\ref{pdistance}, $d_T(x,y)=i$, where $i\in\{1,2,5,6\}$. We therefore infer the following structural property of the rooted tree $T$.
	
	\begin{claim}
		\label{cl-2}
		The rooted tree $T$ contains at least one tree $T_i$ shown in Figure~\ref{trees} as a maximal subtree $T_v$ for some vertex $v$ of $T$ and some $i \in [28]$.
	\end{claim}

	We proceed as follows. We systematically show that the rooted tree $T$ cannot contain a tree $T_i$ shown in Figure~\ref{trees} as a maximal subtree for any $i \in [28]$. To do this, we analyze the specific region of the mop $G$ based on a given subtree $T_i$. In our arguments, if $T_i$ is the maximal subtree $T_v$ for some vertex $v$ of $T$, then in our illustrations of the associated subgraph of $G$ the shaded triangle corresponds to the root $v$ of the maximal subtree $T_v = T_i$. The other regions of $G$ are then triangulates according to the structure of the tree $T_i$ and according to \cref{lem1}\ref{5distance}-\ref{6distance}. Throughout our proof, we adopt the notation that if $T_i$ is a maximal subtree of $T$ for some $i \in [28]$ and $T_i$ is the maximal subtree $T_v$ for some vertex $v$ of $T$, then $R_v$ denotes the triangle in $G$ corresponding to the root vertex $v$ of $T_v$. We adopt the following notation. If $V(G) = \{v_1,v_2,\ldots,v_n\}$, and $i$ and $j$ are integers such that $1 \le j < i \le n$, then we let
	\[
	V_j^i = \{v_j,v_{j+1},\ldots,v_i\}.
	\]

	\begin{claim}\label{tree1}
		The tree $T_1$ is not a  maximal subtree of $T$.
	\end{claim}
	\begin{proof}[Proof of \cref{tree1}]
		Suppose, to the contrary, that $T_1$ is a maximal subtree of $T$, and so $T_1 = T_v$. Let $R_v$ be the triangle in $G$ corresponding to the vertex $v$. Let $V(R_v) = \{v_1,v_2,v_3\}$. Let $s_1$ and $t_1$ be the two children of $v$, and let $R_1$ and $Q_1$ be the triangles in $G$ corresponding to the vertices $s_1$ and $t_1$, respectively. Further, let $V(R_1) = \{v_1,v_3,v_4\}$ and $V(Q_1) = \{v_2,v_3,v_5\}$. Thus, $G$ contains the subgraph illustrated in Figure~\ref{tree_1}, where the shaded triangle corresponds to the vertex $v$ in~$T_v$. Since $s_1$ and $t_1$ are leaves in $T$, we note that $\deg_G(v_4) = \deg_G(v_5) = 2$ and $\deg_G(v_3) = 4$. Recall that $n \ge 13$.
		
		\begin{figure}[htbp]
			\begin{center}
				\includegraphics[scale=0.55]{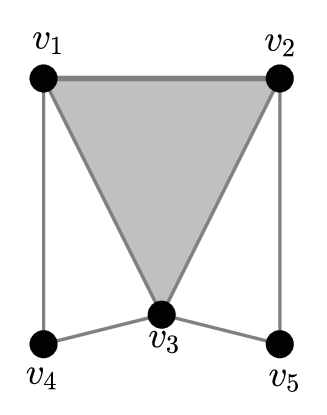}
				\caption{The region of $G$ corresponding to trees $T_1$.}\label{tree_1}
			\end{center}
		\end{figure}
		
		Let $G'$ be a graph of order $n'$ obtained from $G$ by deleting the vertices in  $V_3^5 = \{v_3,v_4,v_5\}$, and so $G' = G - \{v_3,v_4,v_5\}$. Since $n \ge 13$, we have $n' = n - 3 \ge 10$. We note that $G'$ is a mop with $k-1$ vertices of degree~$2$ and $v_1v_2$ is an outer edge of $G'$. Let $G_1$ be a graph of order $n_1$ obtained from $G'$ by contracting the edge $v_1v_2$ to form a new vertex $x$ in $G_1$, and let $G_1$ have $k_1$ vertices of degree~$2$.  Note that $n_1 = n-4$ and $k_1 \le k-1$. By Lemma~\ref{key}, $G_1$ is a mop. Since $n' \ge 10$, we have $n_1 = n' - 1 \ge 9$. Since $G$ is a counterexample of minimum order, we have $\gamma_2^d(G_1) \le \floor*{\frac{2}{9}(n_1+k_1)} \le \floor*{\frac{2}{9}(n-4+k-1)} \le \floor*{\frac{2}{9}(n+k)}-1$. Let $D_1$ be a $\gamma_2^d$-set of $G_1$. If $x \in D_1$, then let $D = (D_1 \setminus \{x\}) \cup \{v_1,v_2\}$. If $x \notin D_1$, then let $D = D_1 \cup \{v_3\}$. In both cases, the set $D$ is a 2DD-set of $G$, and so $\gamma_2^d(G) \le |D| \le |D_1| + 1 \le \floor*{\frac{2}{9}(n+k)}$, a contradiction.
	\end{proof}
	
	Let $T'$ be a path $P_4$ rooted at a vertex at distance~$1$ from a leaf of $T'$, as illustrated in Figure~\ref{tree2to9}(a). We note that the tree $T'$ is a subtree of $T_i$ for all $i$, where $ i\in ([9]\setminus\{1\})\cup \{12,13,18,19,24,25\}$. Let $v'$ be the root of $T'$. In our illustrations of the subgraph of $G$ associated with the rooted tree $T'$, let the shaded triangle with vertex set $\{v_1,v_2,v_{3}\}$ corresponds to the root $v'$ of $T'$, and let the subgraph of $G$ be obtained from region $v_2v_3v_5$ and from the region $v_1v_3v_4v_6$ by triangulating by adding the edge $v_3v_6$ or $v_1v_4$ as illustrated in Figure~\ref{tree2to9}(b)-(c), depending on the two possible cases that these regions can be triangulated.
	
	\begin{figure}[htbp]
		\begin{center}
			\includegraphics[scale=0.47]{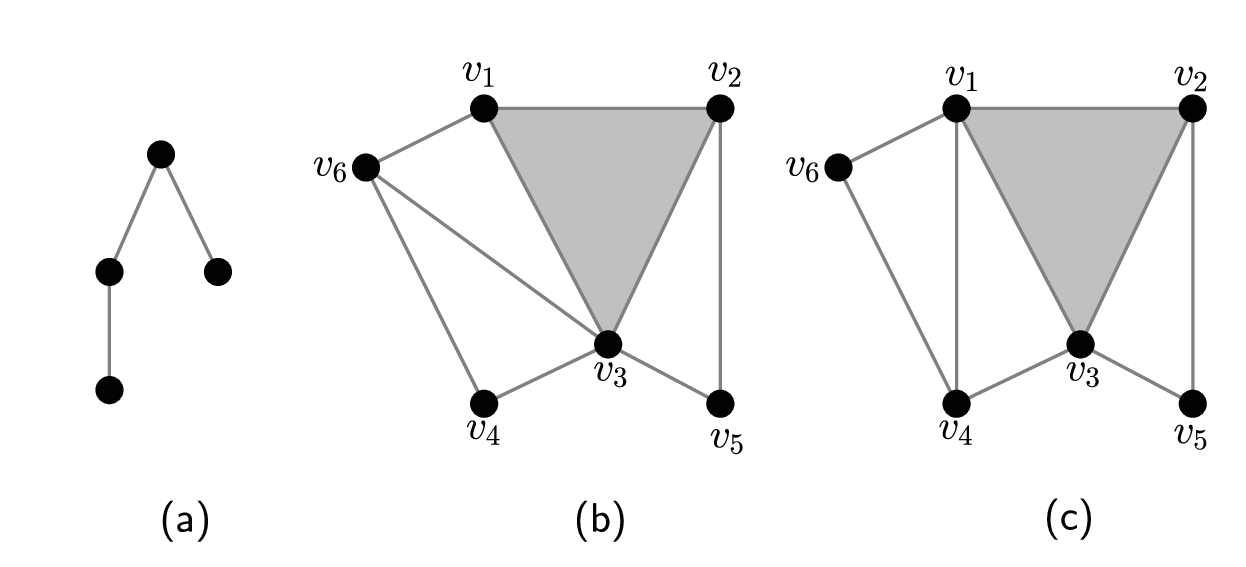}
			\caption{(a)~$T'$, (b)~$H_9$, and (c)~$H_{10}$. Tree $T'$ and possible regions of $G$ corresponding to tree $T'$.}\label{tree2to9}
		\end{center}
	\end{figure}

	\begin{claim}\label{cl3to10}
		The subgraph of $G$ associated with the tree $T'$ in $\Cref{tree2to9}{\rm(a)}$ corresponds to the region $H_{10}$ illustrated in $\Cref{tree2to9}{\rm(c)}$.
	\end{claim}
	\begin{proof}[Proof of Claim~\ref{cl3to10}]
		Suppose, to the contrary, that the subgraph of $G$ associated with the tree $T'$ in Figure~\ref{tree2to9}(a) corresponds to the region $H_9$ illustrated in Figure~\ref{tree2to9}(b). Let $G'$ be the mop of order~$n'$ obtained from $G$ by deleting the vertices in  $V_3^6=\{v_3,v_4,v_5,v_6\}$, and let $G'$ have $k'$ vertices of degree~$2$. We note that $n'=n-4$ and $k' \le k - 1$. Since $n \ge 13$, we have $n' = n - 4 \ge 9$. Since $G$ is a counterexample of minimum order, we have $\gamma_2^d(G') \le \floor*{\frac{2}{9}(n'+k')} \le \floor*{\frac{2}{9}(n-4+k-1)} \le \floor*{\frac{2}{9}(n+k)}-1$. Let $D'$ be a $\gamma_t^d$-set of $G'$ and let $D=D'\cup\{v_3\}$. The set $D$ is a 2DD-set of $G$, and so $\gamma_2^d(G) \le |D| \le |D'| + 1 \le \floor*{\frac{2}{9}(n+k)}$, a contradiction.
	\end{proof}

	In what follows, by Claim~\ref{cl3to10}, the subgraph of $G$ associated with the tree $T'$ in Figure~\ref{tree2to9}(a) corresponds to the region $H_{10}$ illustrated in Figure~\ref{tree2to9}(c).

	\begin{claim}\label{tree2}
		The tree $T_{2}$ is not a maximal subtree of $T$.
	\end{claim}
	\begin{proof}[Proof of Claim~\ref{tree2}]
		Suppose, to the contrary, that $T_{2}$ is a maximal subtree of $T$, and so $T_{2} = T_v$. We therefore infer that the subgraph of $G$ associated with $T_{2}$ is obtained from the region $H_{10}$ in two possible ways, as illustrated in Figure~\ref{tree_2&3}(a)-(b) where for notational convenience, we have interchanged the names of the vertices $v_1$ and $v_2$ in region $H_{10}$ illustrated in Figure~\ref{tree_2&3}(b). In the following, we present arguments that work in both cases.
		
		Let $G'$ be the mop of order $n'$ obtained from $G$ by deleting the vertices in $V_4^8  \setminus \{v_7\}$, and let $G'$ have $k'$ vertices of degree~$2$. We note that $n'=n-4$ and $k' \le k - 1$. Since $n \ge 13$, we have $n' = n - 4 \ge 9$. Since $G$ is a counterexample of minimum order, we have $\gamma_2^d(G') \le \floor*{\frac{2}{9}(n'+k')} \le \floor*{\frac{2}{9}(n-4+k-1)} \le \floor*{\frac{2}{9}(n+k)}-1$. Let $D'$ be a $\gamma_t^d$-set of $G'$. If $D'\cap\{v_1,v_2,v_7\}\ne \emptyset$, then let $ D=D'\cup \{v_3\}$. If $v_3\in D'$, then let $D=D'\cup \{v_2\}$. In both cases, the set $D$ is a 2DD-set of $G$, and so $\gamma_2^d(G)\le |D|\le  |D'|+1\le  \floor*{\frac{2}{9}(n+k)}$, a contradiction. Hence we may assume that $D' \cap \{v_1,v_2,v_3,v_7\} = \emptyset$. Since $D'$ is a 2DD-set of $G'$, we therefore infer that there exists a vertex $u\in D'$ such that $u\in N_{G'}(v_1)$. We now let $D^*=D'\cup \{v_2\}$. The resulting set $D^*$ is a 2DD-set of $G$, and so $\gamma_2^d(G)\le |D^*|\le |D'|+1\le  \floor*{\frac{2}{9}(n+k)}$, a contradiction.
	\end{proof}

	\begin{figure}[htbp]
		\begin{center}
			\includegraphics[scale=0.44]{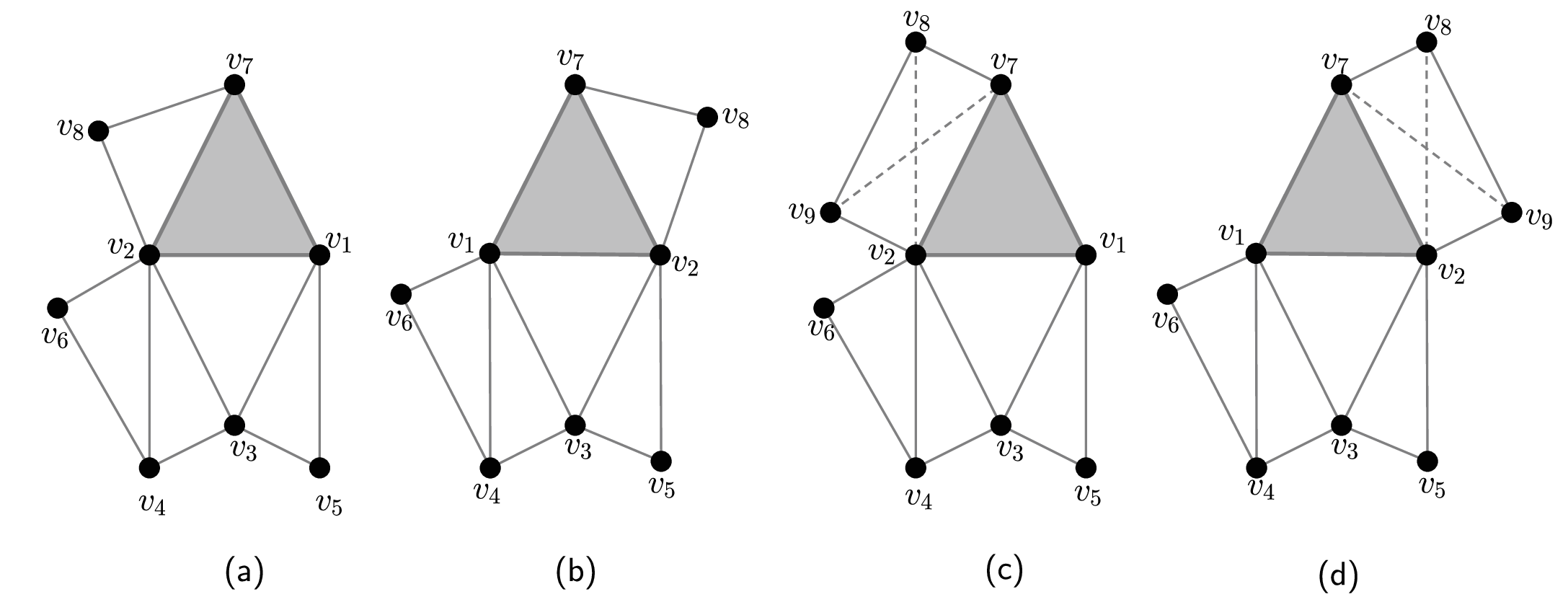}
			\caption{The regions of $G$ corresponding to trees $T_2$ and $T_3$.}\label{tree_2&3}
		\end{center}
	\end{figure}

	\begin{claim}\label{tree3}
		The tree $T_{3}$ is not a maximal subtree of $T$.
	\end{claim}
	\begin{proof}[Proof of Claim~\ref{tree3}]
		Suppose, to the contrary, that $T_{3}$ is a maximal subtree of $T$, and so $T_{3} = T_v$. We therefore infer that the subgraph of $G$ associated with $T_{3}$ is obtained from the region $H_{10}$ in two possible ways, as illustrated in Figure~\ref{tree_2&3}(c)-(d), where for notational convenience, we have interchanged the names of the vertices $v_1$ and $v_2$ in region $H_{10}$ illustrated in Figure~\ref{tree_2&3}(d). The region $v_2v_7v_8v_{10}$ can be triangulated by adding either the edge $v_2v_8$ or $v_7v_9$, as indicated by the dotted lines in Figure~\ref{tree_2&3}(c)-(d). In the following, we present arguments that work in both cases.

		Let $G'$ be a graph of order $n'$ obtained from $G$ by deleting the vertices in $V_2^9\setminus \{v_7\}$, and let $G'$ have $k'$ vertices of degree~$2$. We note that $n' = n-7$ and $k' = k - 2$, and $v_1v_7$ is an outer edge of $G'$. Since $n \ge 13$, we have $n' \ge 6$. If $6 \le n' \le 7$, then by \Cref{obs3}, there exists a 2DD-set $D'$ of $G'$ such that $v_1\in D'$ and $|D'|=2$. Therefore, $D'\cup \{v_2\}$ is a 2DD-set of $G$, and so $\gamma_2^d(G)\le 3 \le  \floor*{\frac{2}{9}(n+k)}$, a contradiction. Hence, $n' \ge 8$. Since $G$ is a counterexample of minimum order, we have $\gamma_2^d(G') \le \floor*{\frac{2}{9}(n'+k')} \le \floor*{\frac{2}{9}(n-7+k-2)} = \floor*{\frac{2}{9}(n+k)}-2$. Let $D'$ be a $\gamma_2^d$-set of $G'$ and let $D=D'\cup\{v_1,v_2\}$. The set $D$ is a 2DD-set of $G$, and so $\gamma_2^d(G) \le |D| \le |D'| + 2 \le \floor*{\frac{2}{9}(n+k)}$, a contradiction.
	\end{proof}

	\begin{claim}\label{tree4}
		The tree $T_{4}$ is not a maximal subtree of $T$.
	\end{claim}
	\begin{proof}[Proof of Claim~\ref{tree4}]
		Suppose, to the contrary, that $T_{4}$ is a maximal subtree of $T$, and so $T_{4} = T_v$ where $v$ denotes the root of the subtree $T_v$. We infer that the subgraph of $G$ associated with $T_{4}$ is obtained from region $H_{10}$ and  by triangulating the region $v_2v_{7}v_{8}v_{9}v_{10}v_{11}$ according to \Cref{cl3to10} as illustrated in Figure~\ref{tree_4}(a)-(d)), where we let $V(T_v) = \{v_1,v_2,v_{7}\}$ be the (shaded) triangle in $G$ associated with the vertex~$v$. In the following, we present arguments that work in each case.
		
		\begin{figure}[htbp]
			\begin{center}
				\includegraphics[scale=0.4]{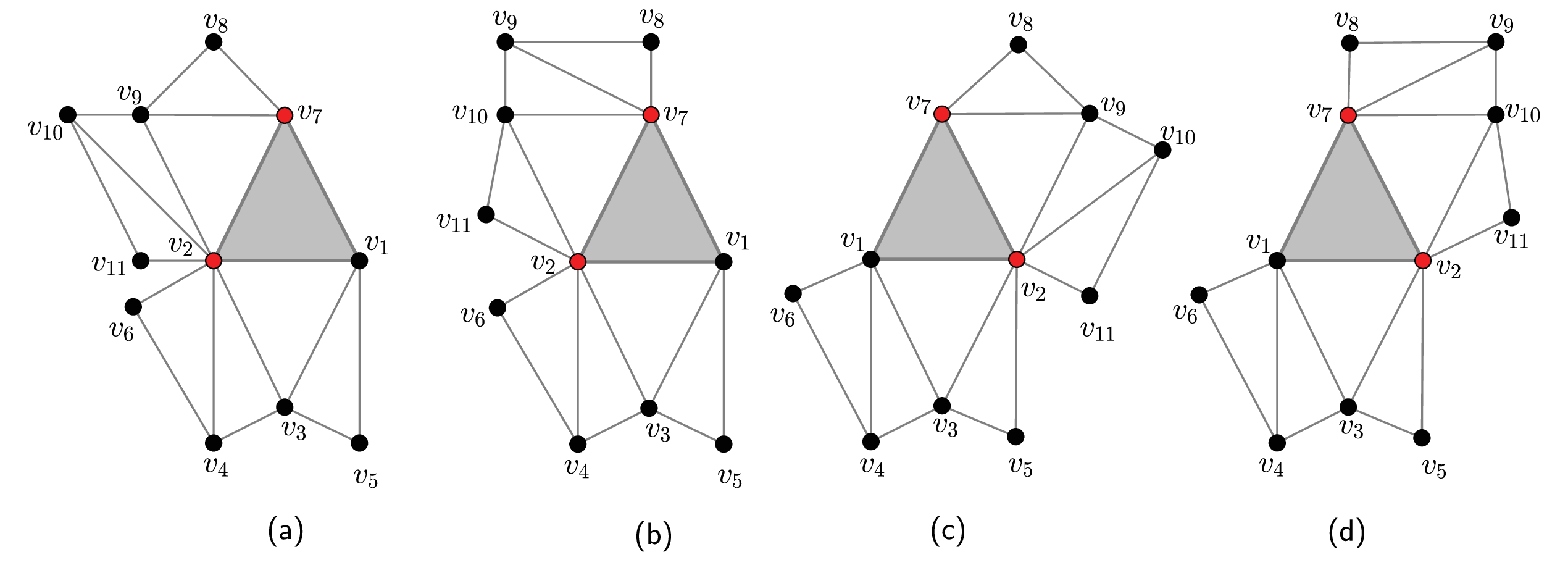}
				\caption{The regions of $G$ corresponding to tree $T_4$. The red vertices show a 2DD-set of $G[V(G)\setminus V(G')]$.}\label{tree_4}
			\end{center}
		\end{figure}

		Let $G'$ be the mop of order $n'$ obtained from $G$ by deleting the vertices in $V_3^{11}\setminus \{v_7\}$, and let $G'$ have $k'$ vertices of degree~$2$. We note that $n' = n-8$ and $k' = k - 3$. Since $n \ge 13$, we have $n' \ge 5$. If $5 \le n' \le 7$, then by~\Cref{obs3}, there exists a 2DD-set $D'$ of $G'$ such that $v_7\in D'$ and $|D'|=2$. Therefore, $D'\cup \{v_2\}$ is a 2DD-set of $G$, and so $\gamma_2^d(G) \le 3 \le  \floor*{\frac{2}{9}(n+k)}$, a contradiction. Hence, $n' \ge 8$. Since $G$ is a counterexample of minimum order, we have $\gamma_2^d(G') \le \floor*{\frac{2}{9}(n'+k')} \le \floor*{\frac{2}{9}(n-8+k-3)} \le \floor*{\frac{2}{9}(n+k)}-2$. Let $D'$ be a $\gamma_2^d$-set of $G'$ and let $D=D'\cup\{v_2,v_7\}$. The set $D$ is a 2DD-set of $G$, and so $\gamma_2^d(G) \le |D| \le |D'| + 2 \le \floor*{\frac{2}{9}(n+k)}$, a contradiction.
	\end{proof}

	\begin{claim}\label{tree5}
		The tree $T_{5}$ is not a maximal subtree of $T$.
	\end{claim}
	\begin{proof}[Proof of Claim~\ref{tree5}]
		Suppose, to the contrary, that $T_{5}$ is a maximal subtree of $T$, and so $T_{5} = T_v$ where $v$ denotes the root of the subtree $T_v$. We infer that the subgraph of $G$ associated with $T_{5}$ is obtained from the region $H_{10}$ in four possible ways, as illustrated in Figure~\ref{tree_5}(a)-(d), where in each case we let $V(T_v) = \{v_2,v_{7},v_{9}\}$ be the (shaded) triangle in $G$ associated with the vertex~$v$. In the following, we present arguments that work in each case.
		
		\begin{figure}[htbp]
			\begin{center}
				\includegraphics[scale=0.4]{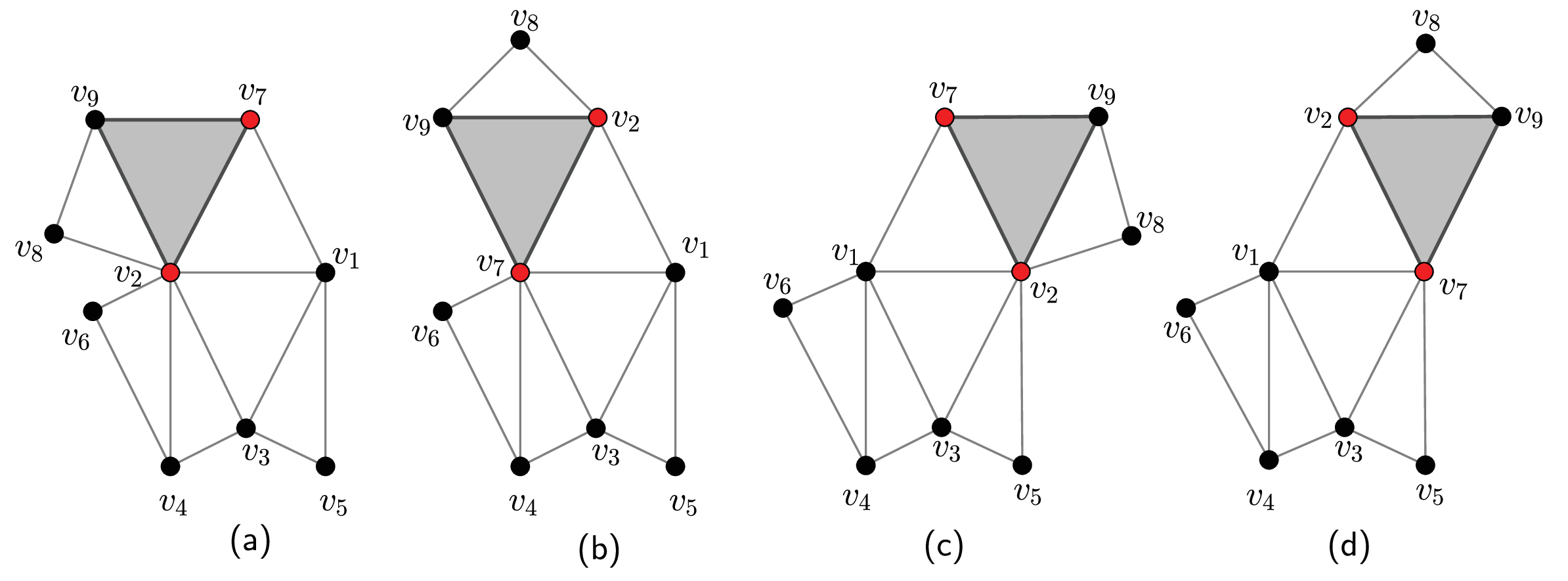}
				\caption{The regions of $G$ corresponding to tree $T_5$. The red vertices show a 2DD-set of $G[V(G)\setminus V(G')]$.}\label{tree_5}
			\end{center}
		\end{figure}	
		
		Let $G'$ be the mop of order $n'$ obtained from $G$ by deleting the vertices in $V_1^8 \setminus \{v_7\}$, and let $G'$ have $k'$ vertices of degree~$2$. We note that $n' = n-7$ and $k' = k - 2$. Since $n \ge 13$, we have $n' \ge 6$. If $6 \le n' \le 7$, then by~\cref{obs3}, there exists a 2DD-set $D'$ of $G'$ such that $v_7\in D'$ and $|D'|= 2$. Therefore, $D'\cup \{v_2\}$ is a 2DD-set of $G$, and so $\gamma_2^d(G)\le 3\le  \floor*{\frac{2}{9}(n+k)}$, a contradiction. Hence, $n' \ge 8$. Since $G$ is a counterexample of minimum order, we have $\gamma_2^d(G') \le \floor*{\frac{2}{9}(n'+k')} \le \floor*{\frac{2}{9}(n-7+k-2)} \le \floor*{\frac{2}{9}(n+k)}-2$. Let $D'$ be a $\gamma_2^d$-set of $G'$ and let $D=D'\cup\{v_2,v_7\}$. The set $D$ is a 2DD-set of $G$, and so $\gamma_2^d(G) \le |D| \le |D'| + 2 \le \floor*{\frac{2}{9}(n+k)}$, a contradiction.
	\end{proof}

	\begin{claim}\label{tree6}
		The tree $T_{6}$ is not a maximal subtree of $T$.
	\end{claim}
	\begin{proof}[Proof of Claim~\ref{tree6}]
		Suppose, to the contrary, that $T_{6}$ is a maximal subtree of $T$, and so $T_{6} = T_v$ where $v$ denotes the root of the subtree $T_v$. We infer that the subgraph of $G$ associated with $T_{6}$ is obtained from region $H_{10}$ in four possible ways, as illustrated in Figure~\ref{tree_6}(a)-(d), where in each case we let $V(T_v) = \{v_2,v_{7},v_{10}\}$ be the (shaded) triangle in $G$ associated with the vertex~$v$. The region $v_2v_8v_9v_{10}$ can be triangulated by adding either the edge $v_2v_9$ or $v_8v_{10}$, as indicated by the dotted lines in Figure~\ref{tree_6}(a)-(d). In the following, we present arguments that work in each case.
		
		\begin{figure}[htbp]
			\begin{center}
				\includegraphics[scale=0.4]{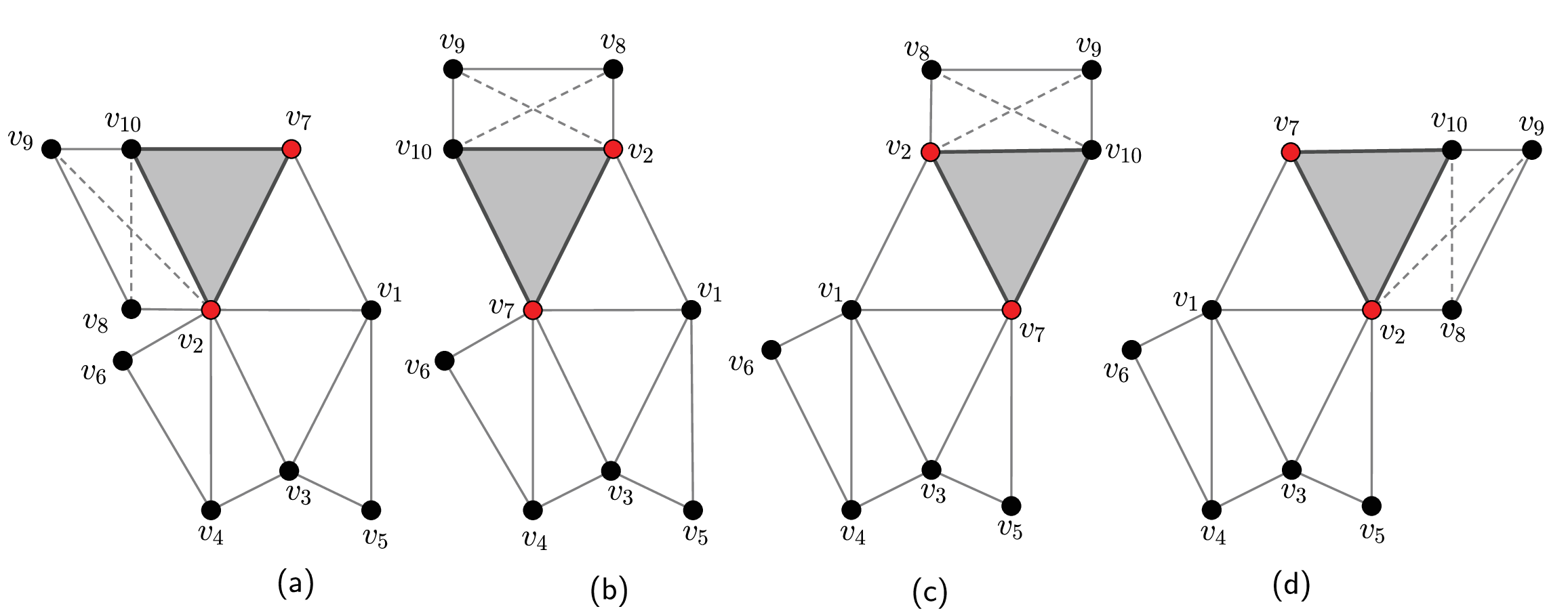}
				\caption{The regions of $G$ corresponding to tree $T_6$. The red vertices show a 2DD-set of $G[V(G)\setminus V(G')]$.}\label{tree_6}
			\end{center}
		\end{figure}

		Let $G'$ be the mop of order $n'$ obtained from $G$ by deleting the vertices in $V_1^{9}\setminus\{v_7\}$, and let $G'$ have $k'$ vertices of degree~$2$. We note that $n' = n-8$ and $k' = k - 2$. Since $n \ge 13$, we have $n' \ge 5$. If $5 \le n' \le 6$, then by~\Cref{obs3}, there exists a 2DD-set $D'$ of $G'$ such that $v_{7}\in D'$ and $|D'|= 2$. Therefore, $D'\cup \{v_2\}$ is a 2DD-set of $G$, and so $\gamma_2^d(G)\le 3\le  \floor*{\frac{2}{9}(n+k)}$, a contradiction. Hence, $n' \ge 7$.  Since $G$ is a counterexample of minimum order, we have $\gamma_2^d(G') \le \floor*{\frac{2}{9}(n'+k')} \le \floor*{\frac{2}{9}(n-8+k-2)} \le \floor*{\frac{2}{9}(n+k)}-2$. Let $D'$ be a $\gamma_2^d$-set of $G'$ and let $D=D'\cup\{v_2,v_7\}$. The set $D$ is a 2DD-set of $G$, and so $\gamma_2^d(G) \le |D| \le |D'| + 2 \le \floor*{\frac{2}{9}(n+k)}$, a contradiction.
	\end{proof}

	\begin{figure}[htbp]
		\begin{center}
			\includegraphics[scale=0.35]{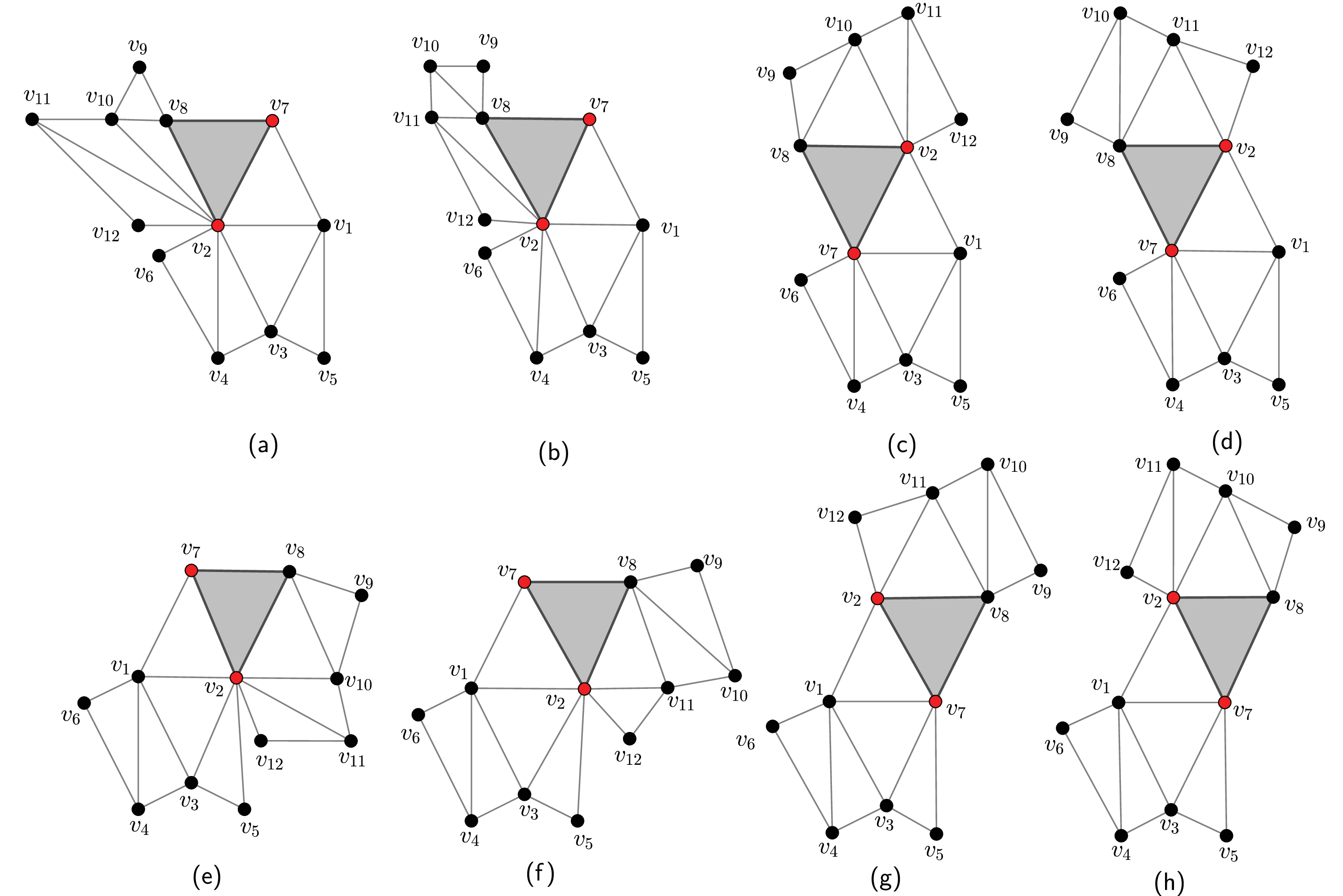}
			\caption{The regions of $G$ corresponding to tree $T_7$. The red vertices show a 2DD-set of $G[V(G)\setminus V(G')]$.}\label{tree_7}
		\end{center}
	\end{figure}
	
	\begin{claim}\label{tree7}
		The tree $T_{7}$ is not a maximal subtree of $T$.
	\end{claim}
	\begin{proof}[Proof of Claim~\ref{tree7}]
		Suppose, to the contrary, that $T_{7}$ is a maximal subtree of $T$, and so $T_{7} = T_v$ where $v$ denotes the root of the subtree $T_v$. We infer that the subgraph of $G$ associated with $T_{7}$ is obtained from region $H_{10}$ and  by triangulating the region $v_2v_{8}v_{9}v_{10}v_{11}v_{12}$ according to \Cref{cl3to10} as illustrated in Figure~\ref{tree_7}(a)-(h), where we let $V(T_v) = \{v_2,v_{7},v_{8}\}$ be the (shaded) triangle in $G$ associated with the vertex~$v$. In the following, we present arguments that work in each case.

		Let $G'$ be the mop of order $n'$ obtained from $G$ by deleting the vertices in $V_3^{12}\setminus\{v_7,v_8\}$, and let $G'$ have $k'$ vertices of degree~$2$. We note that $n' = n-8$ and $k' = k - 3$. Since $n \ge 13$, we have $n' \ge 5$. If $5 \le n' \le 7$, then by~\Cref{obs3}, there exists a 2DD-set $D'$ of $G'$ such that $v_{7}\in D'$ and $|D'|= 2$. Therefore, $D'\cup \{v_2\}$ is a 2DD-set of $G$, and so $\gamma_2^d(G)\le 3\le  \floor*{\frac{2}{9}(n+k)}$, a contradiction. Hence, $n' \ge 8$.  Since $G$ is a counterexample of minimum order, we have $\gamma_2^d(G') \le \floor*{\frac{2}{9}(n'+k')} \le \floor*{\frac{2}{9}(n-8+k-3)} \le \floor*{\frac{2}{9}(n+k)}-2$. Let $D'$ be a $\gamma_2^d$-set of $G'$ and let $D=D'\cup\{v_2,v_7\}$. The set $D$ is a 2DD-set of $G$, and so $\gamma_2^d(G) \le |D| \le |D'| + 2 \le \floor*{\frac{2}{9}(n+k)}$, a contradiction.
	\end{proof}

	\begin{figure}[htbp]
		\begin{center}
			\includegraphics[scale=0.3]{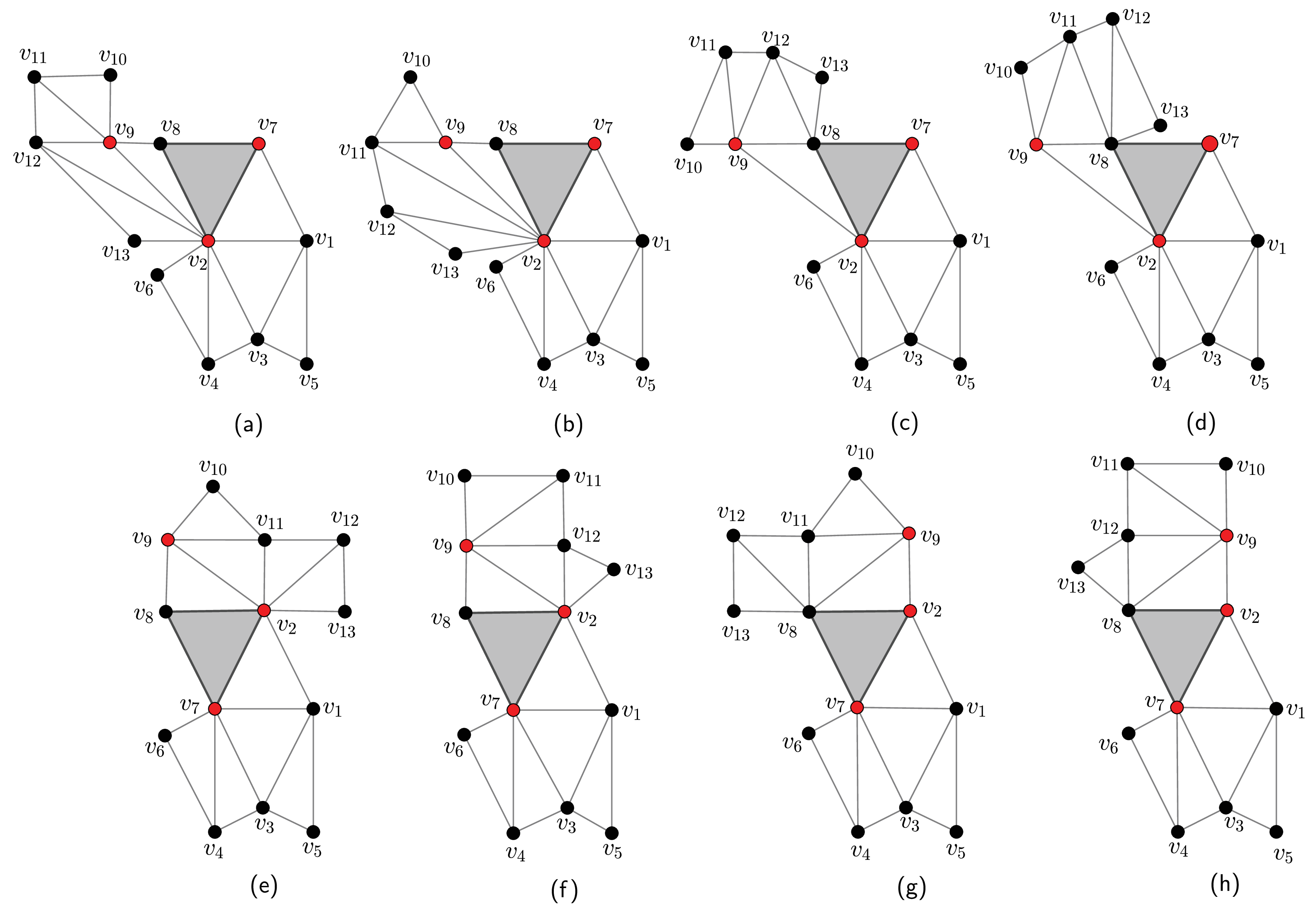}
			\includegraphics[scale=0.28]{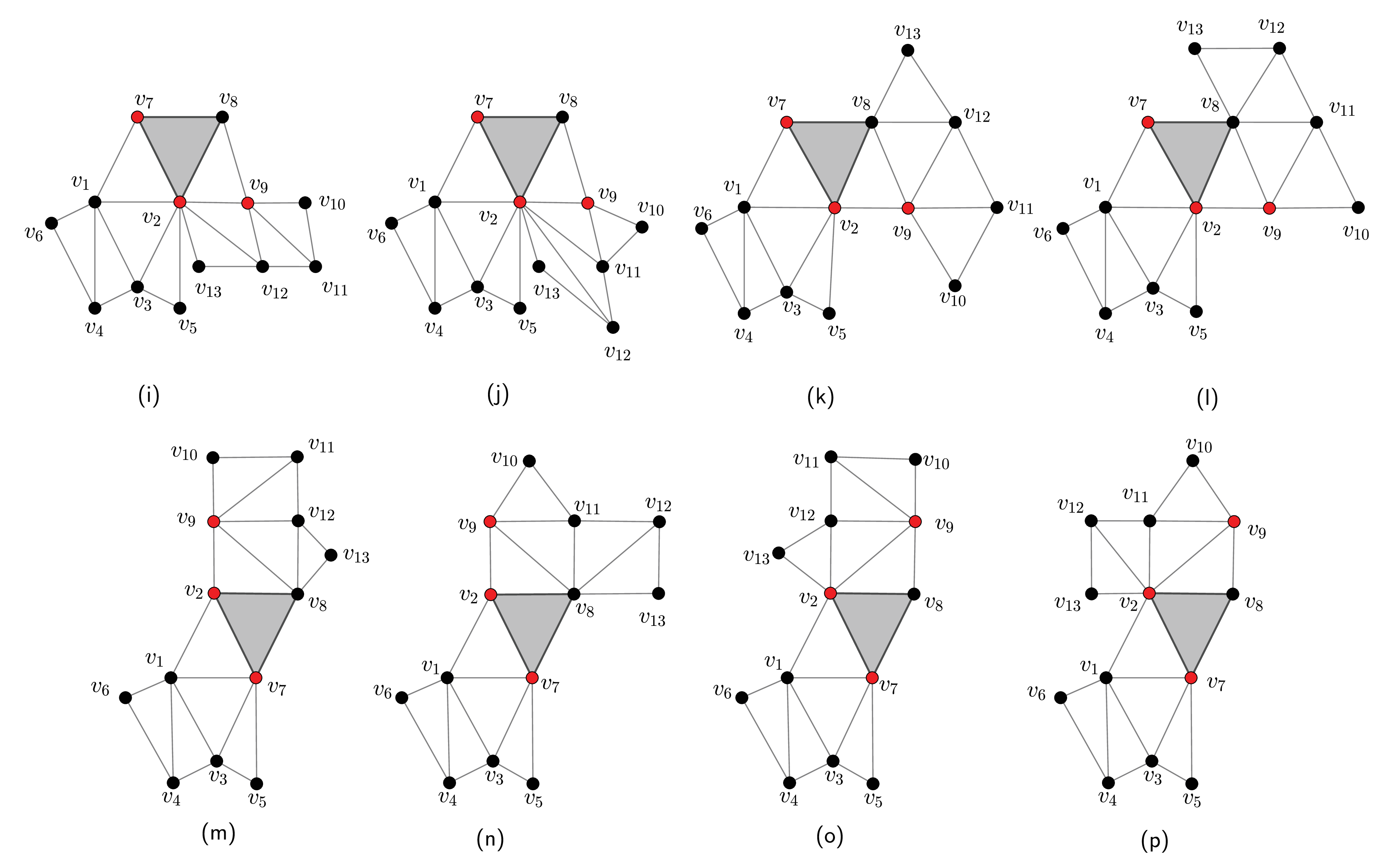}
			\vspace{-.8cm}
			\caption{The regions of $G$ corresponding to tree $T_8$. The red vertices show a 2DD-set of $G[V(G)\setminus V(G')]$.}\label{tree_8}
		\end{center}
	\end{figure}
	
	\begin{claim}\label{tree8}
		The tree $T_{8}$ is not a maximal subtree of $T$.
	\end{claim}
	\begin{proof}[Proof of Claim~\ref{tree8}]
		Suppose, to the contrary, that $T_{8}$ is a maximal subtree of $T$, and so $T_{8} = T_v$ where $v$ denotes the root of the subtree $T_v$. We infer that the subgraph of $G$ associated with $T_{8}$ is obtained from region $H_{10}$ in sixteen possible ways, as illustrated in Figure~\ref{tree_8}(a)-(p), where we let $V(T_v) = \{v_2,v_{7},v_{8}\}$ be the (shaded) triangle in $G$ associated with the vertex~$v$. In the following, we present arguments that work in each case.
		
		Let $G'$ be the mop of order $n'$ obtained from $G$ by deleting the vertices in $V_1^{13}\setminus\{v_7,v_8\}$, and let $G'$ have $k'$ vertices of degree~$2$. We note that $n' = n-11$ and $k' = k - 3$, and $v_7v_{8}$ is an outer edge of $G'$. Since $n \ge 13$, we have $n' \ge 2$. If $2 \le n' \le 4$, then $\{v_2,v_7,v_9\}$ is a 2DD-set of $G$, and hence $\gamma_2^d(G)\le 3 \le  \floor*{\frac{2}{9}(n+k)}$, a contradiction. If $5 \le n' \le 8$, then by~\Cref{obs3}, there exists a 2DD-set $D'$ of $G'$ such that $v_{7}\in D'$ and $|D'|= 2$. Therefore in this case,  $D'\cup \{v_2,v_9\}$ is a 2DD-set of $G$, and so $\gamma_2^d(G)\le 4 \le \floor*{\frac{2}{9}(n+k)}$, a contradiction. Hence, $n' \ge 8$. Since $G$ is a counterexample of minimum order, we have $\gamma_2^d(G') \le \floor*{\frac{2}{9}(n'+k')} \le \floor*{\frac{2}{9}(n-11+k-3)} \le \floor*{\frac{2}{9}(n+k)}-3$. Let $D'$ be a $\gamma_2^d$-set of $G'$ and let $D=D'\cup \{v_2,v_7,v_9\}$. The set $D$ is a 2DD-set of $G$, and so $\gamma_2^d(G) \le |D| \le |D'| + 3 \le \floor*{\frac{2}{9}(n+k)}$, a contradiction.
	\end{proof}

	\begin{claim}\label{cl-9to}
		The tree $T_9$ is not a maximal subtree of $T$.
	\end{claim}
	\begin{proof}[Proof of Claim~\ref{cl-9to}]
		Suppose, to the contrary, that $T_{9}$ is a maximal subtree of $T$, and so $T_{9} = T_v$ where $v$ denotes the root of the subtree $T_v$. We infer that the subgraph of $G$ associated with $T_{9}$ is obtained from region $H_{10}$ in four possible ways, as illustrated in Figure~\ref{tree_9}(a)-(d), where we let $V(T_v) = \{v_2,v_{8},v_{9}\}$ be the (shaded) triangle in $G$ associated with the vertex~$v$.
		
		\begin{figure}[htbp]
			\begin{center}
				\includegraphics[scale=0.35]{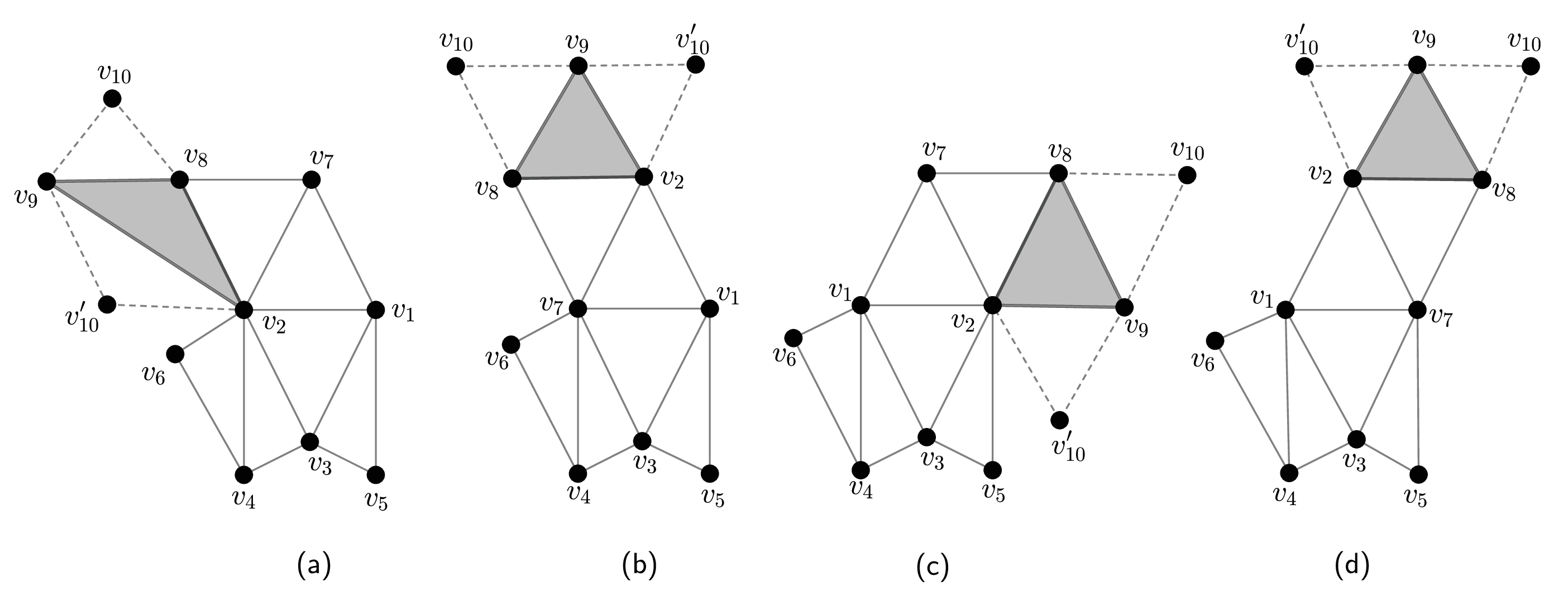}
				\caption{The regions of $G$ corresponding to tree $T_9$.} \label{tree_9}
			\end{center}
		\end{figure}
		
		Suppose firstly that $V(T_v) = \{v_2,v_{8},v_{9}\}$ is an internal triangle of $G$. Let $G'$ be a graph of order $n'$ obtained from $G$ by deleting the vertices in $V_1^{7}\setminus\{v_2\}$, and let $G'$ have $k'$ vertices of degree~$2$. We note that $n' = n-6$ and $k' = k - 2$, and $v_2v_{8}$ is an outer edge of $G'$. Since $n \ge 13$, we have $n' \ge 7$. If $n'=7$, then by ~\Cref{obs3}, there exists a 2DD-set $D'$ of $G'$ such that $v_{8}\in D'$ and $|D'|= 2$. Therefore, $D'\cup \{v_1\}$ is a 2DD-set of $G$, and so $\gamma_2^d(G)\le 3 \le  \floor*{\frac{2}{9}(n+k)}$, a contradiction. Hence, $n' \ge 8$. Let $G_1$ be a graph of order $n_1$ obtained from $G'$ by contracting the edge $v_2v_8$ to form a new vertex $x$ in $G_1$, and let $G_1$ have $k_1$ vertices of degree~$2$. By Lemma~\ref{key}, $G_1$ is a mop. Since $n' \ge 8$, we note that $n_1 = n' - 1 \ge 7$. Further, we note that $n_1 = n - 7$ and $k_1 \le k - 2$. By the minimality of the mop $G$, we have $\gamma_2^d(G_1) \le \floor*{\frac{2}{9}(n_1+k_1)} \le \floor*{\frac{2}{9}(n-7+k-2)} \le \floor*{\frac{2}{9}(n+k)}-2$. Let $D_1$ be a $\gamma_2^d$-set of $G_1$. If $x \in D_1$, then let $D = (D_1 \setminus \{x\}) \cup \{v_1,v_2,v_8\}$. If $x \notin D_1$, then let $D = D_1 \cup \{v_1,v_8\}$. In both cases $D$ is a 2DD-set of $G$, and so $\gamma_2^d(G) \le |D| \le |D_1| + 2 \le \floor*{\frac{2}{9}(n+k)}$, a contradiction.
		
		Hence, $V(T_v) = \{v_2,v_{8},v_{9}\}$ is not an internal triangle of $G$. Since $n \ge 13$, there exists a triangle $F$ adjacent to face $v_2v_8v_{9}$. There are two possible triangles that can be formed: either $V(F)=\{v_8,v_{9},v_{10}\}$ or $V(F)=\{v_2,v_{9},v'_{10}\}$. These are illustrated with dotted lines in \Cref{tree_9}(a)-(d).
		
		Suppose firstly that $V(F)=\{v_8,v_{9},v_{10}\}$. In this case, let $G'$ be a graph of order $n'$ obtained from $G$ by deleting the vertices in $V_1^{7}$, and let $G'$ have $k'$ vertices of degree~$2$. We note that $n' = n-7$ and $k' = k - 1$, and $v_8v_{9}$ is an outer edge of $G'$. Since $n \ge 13$, we have $n' \ge 6$. If $6 \le n' \le 7$, then by \Cref{obs3}, there exists a 2DD-set $D'$ of $G'$ such that $v_8\in D'$ and $|D'|=2$. Therefore, $D'\cup \{v_1\}$ is a 2DD-set of $G$, and so $\gamma_2^d(G)\le 3\le  \floor*{\frac{2}{9}(n+k)}$, a contradiction. Hence, $n' \ge 8$. Let $G_1$ be a graph of order $n_1$ obtained from $G'$ by contracting the edge $v_8v_{9}$ to form a new vertex $x$ in $G_1$, and let $G_1$ have $k_1$ vertices of degree~$2$. By Lemma~\ref{key}, $G_1$ is a mop. Since $n' \ge 8$, we note that $n_1 = n' - 1 \ge 7$. Further, we note that $n_1 = n - 8$ and $k_1 \le k - 1$. By the minimality of the mop $G$, we have $\gamma_2^d(G_1) \le \floor*{\frac{2}{9}(n_1+k_1)} \le \floor*{\frac{2}{9}(n-8+k-1)} \le \floor*{\frac{2}{9}(n+k)}-2$. Let $D_1$ be a $\gamma_2^d$-set of $G_1$. If $x \in D_1$, then let $D = (D_1 \setminus \{x\}) \cup \{v_1,v_8,v_9\}$. If $x \notin D_1$, then let $D = D_1 \cup \{v_1,v_8\}$. In both cases $D$ is a 2DD-set of $G$, and so $\gamma_2^d(G) \le |D| \le |D_1| + 2 \le \floor*{\frac{2}{9}(n+k)}$, a contradiction.
		
		Hence, $V(F)=\{v_2,v_{9},v'_{10}\}$. We now let $G'$ be a graph of order $n'$ obtained from $G$ by deleting the vertices in $V_1^{8}\setminus\{v_2\}$, and let $G'$ have $k'$ vertices of degree~$2$. We note that $n' = n-7$ and $k' = k - 1$, and $v_2v_{9}$ is an outer edge of $G'$. Since $n \ge 13$, we have $n' \ge 6$. If $6 \le n' \le 7$, then by \Cref{obs3}, there exists a 2DD-set $D'$ of $G'$ such that $v_2\in D'$ and $|D'|=2$. Therefore, $D'\cup \{v_7\}$ is a 2DD-set of $G$, and so $\gamma_2^d(G)\le 3\le  \floor*{\frac{2}{9}(n+k)}$, a contradiction. Hence, $n' \ge 8$. Let $G_1$ be a graph of order $n_1$ obtained from $G'$ by contracting the edge $v_2v_{9}$ to form a new vertex $x$ in $G_1$, and let $G_1$ have $k_1$ vertices of degree~$2$. By Lemma~\ref{key}, $G_1$ is a mop. Since $n' \ge 8$, we note that $n_1 = n' - 1 \ge 7$. Further, we note that $n_1 = n - 8$ and $k_1 \le k - 1$. By the minimality of the mop $G$, we have $\gamma_2^d(G_1) \le \floor*{\frac{2}{9}(n_1+k_1)} \le \floor*{\frac{2}{9}(n-8+k-1)} \le \floor*{\frac{2}{9}(n+k)}-2$. Let $D_1$ be a $\gamma_2^d$-set of $G_1$. If $x \in D_1$, then let $D = (D_1 \setminus \{x\}) \cup \{v_2,v_7,v_9\}$. If $x \notin D_1$, then let $D = D_1 \cup \{v_2,v_7\}$. In both cases $D$ is a 2DD-set of $G$, and so $\gamma_2^d(G) \le |D| \le |D_1| + 2 \le \floor*{\frac{2}{9}(n+k)}$, a contradiction.
	\end{proof}
	
	Let $T''$ be a path $P_5$ rooted at a vertex at distance~$2$ from a leaf of $T''$, as illustrated in Figure~\ref{tree10to21}(a). We note that the tree $T''$ is a subtree of $T_i$ for all $i$, where  $i\in([15]\setminus[9])\cup \{20,26\}$. Let $v'$ be the root of $T''$. In our illustrations of the subgraph of $G$ associated with the rooted tree $T'$, let the shaded triangle with vertex set $\{v_1,v_2,v_{3}\}$ corresponds to the root $v'$ of $T'$, and let the subgraph of $G$ be obtained from the region $v_2v_3v_7v_5$ by triangulating by adding the edge $v_3v_7$ or $v_2v_5$ and from the region $v_1v_3v_4v_6$ by triangulating by adding the edge $v_3v_6$ or $v_1v_4$ as illustrated in Figure~\ref{tree10to21}(b)-(d), depending on the three possible cases that these regions can be triangulated.	
	
	\begin{figure}[htbp]
		\begin{center}
			\includegraphics[scale=0.4]{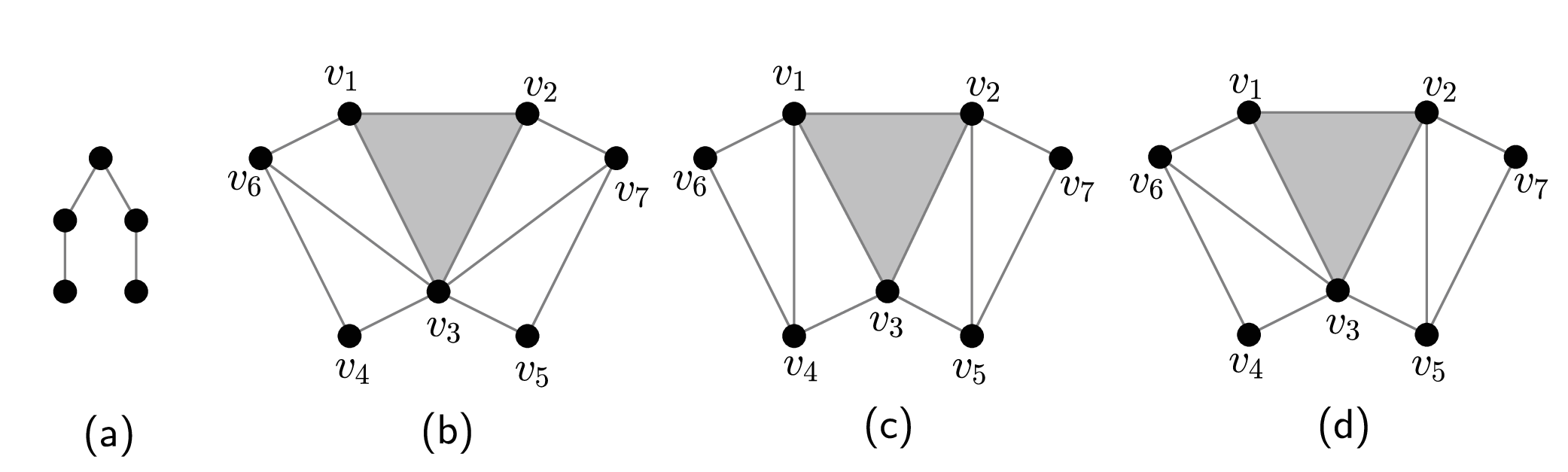}
			\caption{(a)~$T''$, (b)~$H_{11}$, (c)~$H_{12}$, and (d)~$H_{13}$. Tree $T''$ and possible regions of $G$ corresponding to tree $T''$.} \label{tree10to21}
		\end{center}
	\end{figure}

	\begin{claim}\label{cl-10to21}
		The subgraph of $G$ associated with the tree $T''$ in Figure~$\ref{tree10to21}(a)$ corresponds to the region $H_{13}$ illustrated in Figure~$\ref{tree10to21}(d)$.
	\end{claim}
	\begin{proof}[Proof of Claim~\ref{cl-10to21}]
		Suppose, to the contrary, that the subgraph of $G$ associated with the tree $T''$ in Figure~\ref{tree10to21}(a) corresponds to one of the regions $H_{11}$ and $H_{12}$ illustrated in Figure~\ref{tree10to21}(b) and~\ref{tree10to21}(c). Let $G'$ be the mop of order $n'$ obtained from $G$ by deleting the vertices in $V_4^{7}$, and let $G'$ have $k'$ vertices of degree~$2$. We note that $n' = n-4$ and $k' = k - 1$. Since $n \ge 13$, we have $n' \ge 9$. By the minimality of the mop $G$, we have $\gamma_2^d(G_1) \le \floor*{\frac{2}{9}(n'+k')} = \floor*{\frac{2}{9}(n-4+k-1)}= \floor*{\frac{2}{9}(n+k)}-1$. Let $D'$ be a $\gamma_2^d$-set of $G'$, and so $|D'| \le  \floor*{\frac{2}{9}(n+k)}-1$. If $T''$ corresponds to region $H_{11}$ shown in Figure~\ref{tree10to21}(b), then let $D = D' \cup \{v_3\}$.  The set $D$ is a 2DD-set of $G$, and so $\gamma_2^d(G) \le |D| \le |D'| + 1 \le \floor*{\frac{2}{9}(n+k)}$, a contradiction. Hence, $T''$ corresponds to region $H_{12}$ shown in Figure~\ref{tree2to9}(c).

		If $v_3 \in D'$, then let $D=(D'\setminus\{v_3\})\cup \{v_1,v_2\}$. If $D'\cap \{v_1,v_2\}\ne \emptyset$, then  let $D=D'\cup \{v_3\}$. In both cases, the resulting set $D$ is a 2DD-set of $G$, and so $\gamma_2^d(G)\le |D|\le |D'|+1 \le \floor*{\frac{2}{9}(n+k)}$, a contradiction. Hence, $D'\cap \{v_1,v_2,v_3\} = \emptyset$. Since $D'$ is a 2DD-set of $G'$, there exists a vertex $u\in D'$ such that $u\in N_{G'}(v_1)$ or $u\in N_{G'}(v_2)$. Without loss of generality, we may assume that  $u\in N_{G'}(v_1)$. Let $D^*=D'\cup \{v_2\}$. The set $D^*$ is a 2DD-set of $G$, and so  $\gamma_2^d(G)\le |D^*|\le |D'|+1\le \floor*{\frac{2}{9}(n+k)}$, a contradiction. Thus, $T''$ does not correspond to regions $H_3$ and $H_4$ shown in Figure~\ref{tree10to21}(b)-(c). We therefore infer that  $T''$ corresponds to the region $H_{13}$ illustrated in Figure~\ref{tree10to21}(d).
	\end{proof}
	
	In what follows, by Claim~\ref{cl-10to21}, the subgraph of $G$ associated with the tree $T''$ in Figure~\ref{tree10to21}(a) corresponds to the region $H_{13}$ illustrated in Figure~\ref{tree10to21}(d).
	
	\begin{claim}\label{tree10}
		The tree $T_{10}$ is not a maximal subtree of $T$.
	\end{claim}
	\begin{proof}[Proof of Claim~\ref{tree10}]
		Suppose, to the contrary, that $T_{10}$ is a maximal subtree of $T$, and so $T_{10} = T_v$. We therefore infer that the subgraph of $G$ associated with $T_{10}$ is obtained from the region $H_{13}$ in two possible ways, as illustrated in Figure~\ref{tree_1011}(a)-(b) where for notational convenience, we have interchanged the names of the vertices $v_1$ and $v_2$ in region $H_{13}$ in Figure~\ref{tree_1011}(b), where $v$ denotes the root of the subtree $T_v$ and where in this case we let $V(T_v) = \{v_1,v_2,v_{8}\}$ be the (shaded) triangle in $G$ associated with the vertex~$v$ as illustrated in Figure~\ref{tree_1011}(a)-(b).  In the following, we present arguments that work in both cases.

		Let $G'$ be the mop of order $n'$ obtained from $G$ by deleting the vertices in $V_2^9 \setminus \{v_8\}$, and let $G'$ have $k'$ vertices of degree~$2$. We note that $n' = n-7$ and $k' = k - 2$. Since $n \ge 13$, we have $n' \ge 6$. If $6 \le n' \le 7$, then by \cref{obs3}, there exists a 2DD-set $D'$ of $G'$ such that $v_{1}\in D'$ and  $|D'|=2$. Therefore, $D'\cup \{v_2\}$ is a 2DD-set of $G$, and so $\gamma_2^d(G)\le 3 \le  \floor*{\frac{2}{9}(n+k)}$, a contradiction. Hence, $n' \ge 8$. By the minimality of the mop $G$, we have $\gamma_2^d(G_1) \le \floor*{\frac{2}{9}(n'+k')} = \floor*{\frac{2}{9}(n-7+k-2)} = \floor*{\frac{2}{9}(n+k)}-2$. Let $D'$ be a $\gamma_2^d$-set of $G'$ and let $D=D'\cup\{v_1,v_2\}$. The set $D$ is a 2DD-set of $G$, and so $\gamma_2^d(G) \le |D| \le |D'| + 2 \le \floor*{\frac{2}{9}(n+k)}$, a contradiction.
	\end{proof}
	
	\begin{figure}[htbp]
		\begin{center}
			\includegraphics[scale=0.4]{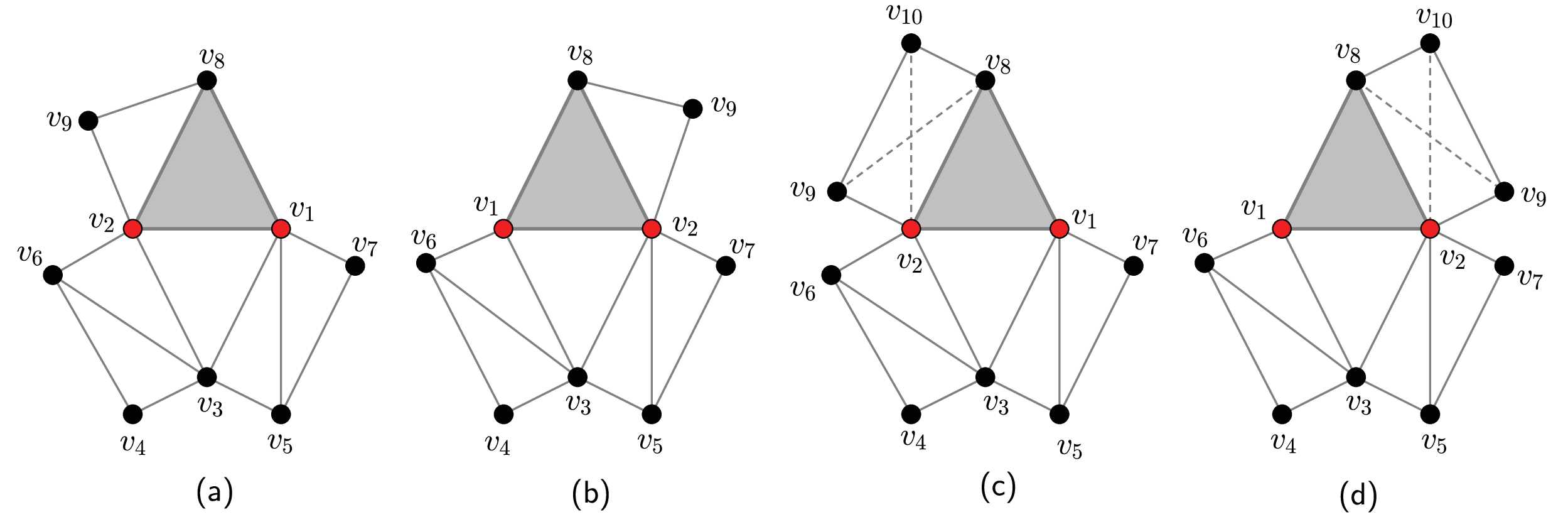}
			\caption{The regions of $G$ corresponding to trees $T_{10}$ and $T_{11}$. The red vertices show a 2DD-set of $G[V(G)\setminus V(G')]$.}\label{tree_1011}
		\end{center}
	\end{figure}

	\begin{claim}\label{tree11}
		The tree $T_{11}$ is not a maximal subtree of $T$.
	\end{claim}
	\begin{proof}[Proof of Claim~\ref{tree11}]
		Suppose, to the contrary, that $T_{11}$ is a maximal subtree of $T$, and so $T_{11} = T_v$. We therefore infer that the subgraph of $G$ associated with $T_{11}$ is obtained from the region $H_{13}$ in two possible ways, as illustrated in Figure~\ref{tree_1011}(c)-(d) where for notational convenience, we have interchanged the names of the vertices $v_1$ and $v_2$ in region $H_{13}$ illustrated in Figure~\ref{tree_1011}(d), where $v$ denotes the root of the subtree $T_v$ and where in this case we let $V(T_v) = \{v_1,v_2,v_{8}\}$ be the (shaded) triangle in $G$ associated with the vertex~$v$ as illustrated in Figure~\ref{tree_1011}(c)-(d). The region $v_2v_8v_{10}v_9$ can be triangulated by adding either the edge $v_2v_{10}$ or $v_8v_{9}$, as indicated by the dotted lines in Figure~\ref{tree_1011}(c)-(d).  In the following, we present arguments that work in both cases.
		
		Let $G'$ be the mop of order $n'$ obtained from $G$ by deleting the vertices in $V_3^{10}\setminus\{v_8\}$, and let $G'$ have $k'$ vertices of degree~$2$. We note that $n' = n-7$ and $k' = k - 2$. Since $n \ge 13$, we have $n' \ge 6$. If $6 \le n' \le 7$, then by \cref{obs3}, there exists a 2DD-set $D'$ of $G'$ such that $|D'|= 2$ and $v_{1}\in D'$. Therefore, $D'\cup \{v_2\}$ is a 2DD-set of $G$, and so $\gamma_2^d(G)\le 3 \le  \floor*{\frac{2}{9}(n+k)}$, a contradiction. Hence, $n' \ge 8$. By the minimality of the mop $G$, we have $\gamma_2^d(G_1) \le \floor*{\frac{2}{9}(n'+k')} = \floor*{\frac{2}{9}(n-7+k-2)} \le \floor*{\frac{2}{9}(n+k)}-2$. Let $D'$ be a $\gamma_2^d$-set of $G'$ and let $D=D'\cup\{v_1,v_2\}$. The set $D$ is a 2DD-set of $G$, and so $\gamma_2^d(G) \le |D| \le |D'| + 2 \le \floor*{\frac{2}{9}(n+k)}$, a contradiction.
	\end{proof}

	\begin{claim}\label{tree12}
		The tree $T_{12}$ is not a maximal subtree of $T$.
	\end{claim}
	\begin{proof}[Proof of Claim~\ref{tree12}]
		Suppose, to the contrary, that $T_{12}$ is a maximal subtree of $T$, and so $T_{12} = T_v$. We infer that the subgraph of $G$ associated with $T_{12}$ is obtained from region $H_{13}$ and by triangulating the region $v_2v_{8}v_{9}v_{10}v_{11}v_{12}$ according to \Cref{cl3to10} as illustrated in Figure~\ref{tree_12}(a)-(d), where $v$ denotes the root of the subtree $T_v$ and where in this case we let $V(T_v) = \{v_1,v_2,v_{8}\}$ be the (shaded) triangle in $G$ associated with the vertex~$v$. In the following, we present arguments that work in each case.
		
		\begin{figure}[htbp]
			\begin{center}
				\includegraphics[scale=0.36]{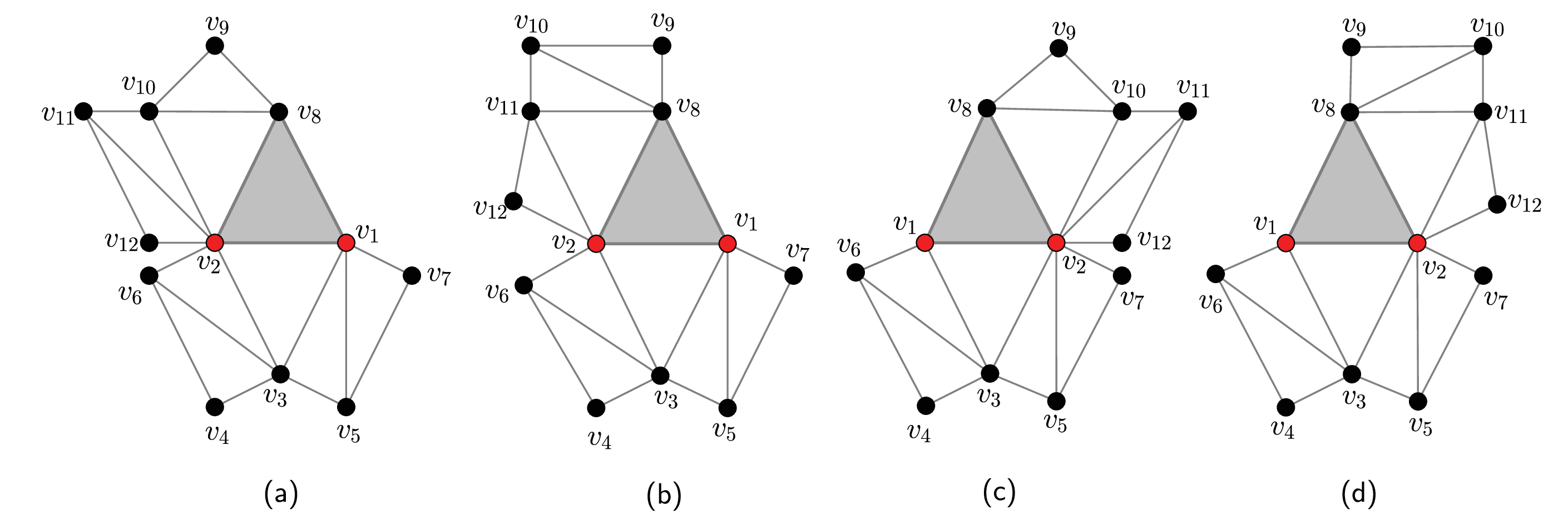}
				\caption{The regions of $G$ corresponding to tree $T_{12}$. The red vertices show a 2DD-set of $G[V(G)\setminus V(G')]$.}\label{tree_12}
			\end{center}
		\end{figure}
		
		Let $G'$ be a graph of order $n'$ obtained from $G$ by deleting the vertices in $V_3^{12}\setminus\{v_8\}$, and let $G'$ have $k'$ vertices of degree~$2$. We note that $n' = n-9$ and $k' = k - 3$. Since $n \ge 13$, we have $n' \ge 4$. If $n'=4$, then $\{v_1,v_2\}$ is a 2DD-set of $G$, and hence $\gamma_2^d(G)\le 2 \le \floor*{\frac{2}{9}(n+k)}$, a contradiction. If $5 \le n' \le 7$, then by \cref{obs3}, there exists a 2DD-set $D'$ of $G'$ such that $|D'|= 2$ and $v_{1}\in D'$. Therefore in this case, $D'\cup \{v_2\}$ is a 2DD-set of $G$, and so $\gamma_2^d(G)\le 3 \le  \floor*{\frac{2}{9}(n+k)}$, a contradiction. Hence, $n' \ge 8$. By the minimality of the mop $G$, we have $\gamma_2^d(G_1) \le \floor*{\frac{2}{9}(n'+k')} = \floor*{\frac{2}{9}(n-9+k-3)} \le \floor*{\frac{2}{9}(n+k)}-2$. Let $D'$ be a $\gamma_2^d$-set of $G'$ and let $D=D'\cup\{v_1,v_2\}$. The set $D$ is a 2DD-set of $G$, and so $\gamma_2^d(G) \le |D| \le |D'| + 2 \le \floor*{\frac{2}{9}(n+k)}$, a contradiction.
	\end{proof}

	\begin{claim}\label{tree13}
		The tree $T_{13}$ is not a maximal subtree of $T$.
	\end{claim}
	\begin{proof}[Proof of Claim~\ref{tree13}]
		Suppose, to the contrary, that $T_{13}$ is a maximal subtree of $T$, and so $T_{13} = T_v$. We infer that the subgraph of $G$ associated with $T_{13}$ is obtained from region $H_{13}$ and  (i) either by triangulating the region $v_2v_{9}v_{10}v_{11}v_{12}v_{13}$ according to \Cref{cl3to10} as illustrated in Figure~\ref{tree_4}(a), (b), (e), and (f) or (ii) by triangulating the region $v_8v_{9}v_{10}v_{11}v_{12}v_{13}$ according to \Cref{cl3to10} as illustrated in Figure~\ref{tree_4}(c), (d), (g), and (h), where we let $V(T_v) = \{v_1,v_{2},v_{8}\}$ be the (shaded) triangle in $G$ associated with the vertex~$v$. In the following, we present arguments that work in each case.
		
		\begin{figure}[htbp]
			\begin{center}
				\includegraphics[scale=0.35]{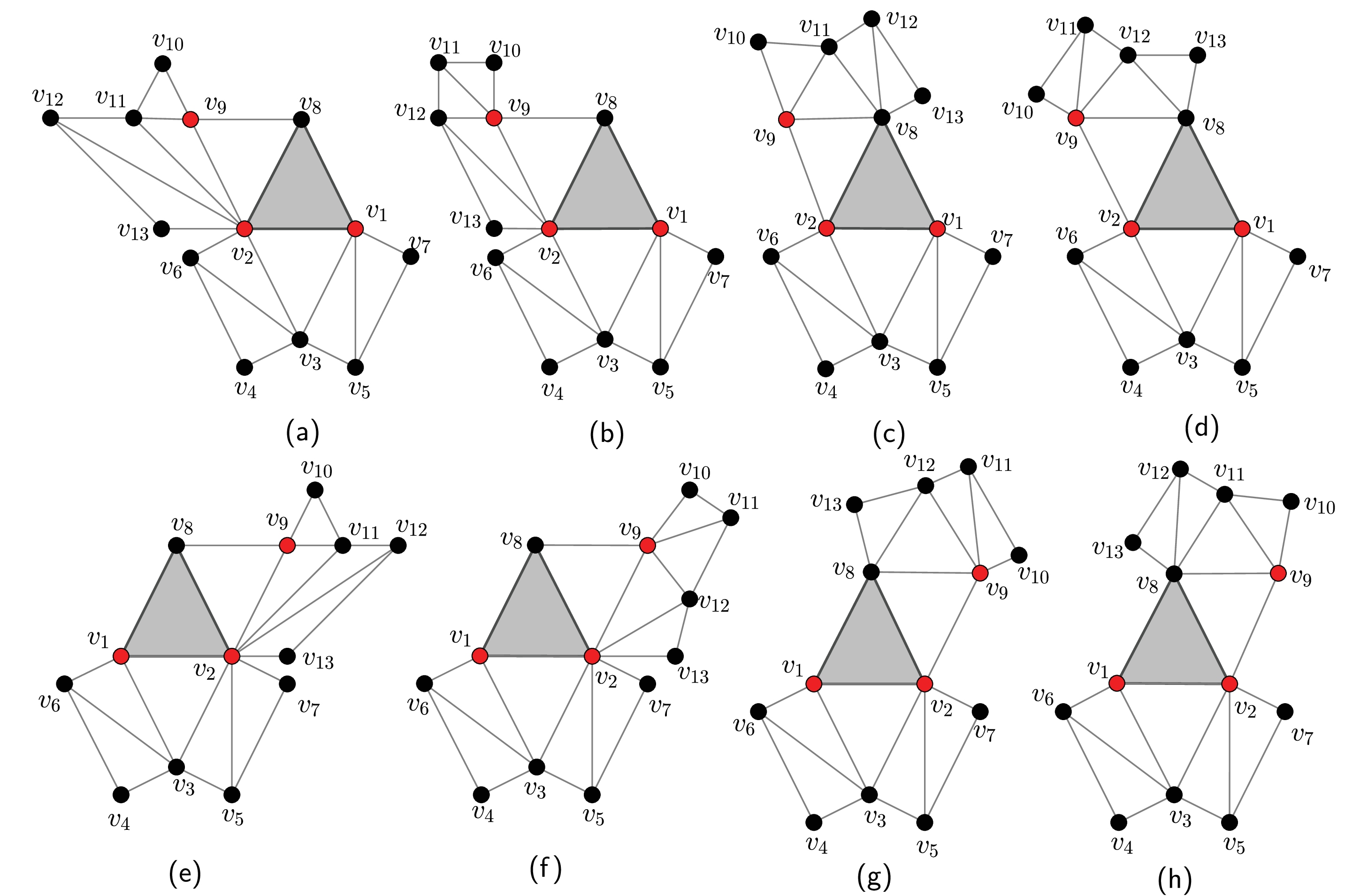}
				\caption{The regions of $G$ corresponding to tree $T_{13}$. The red vertices show a 2DD-set of $G[V(G)\setminus V(G')]$.}\label{tree_13}
			\end{center}
		\end{figure}
		
		Let $G'$ be a graph of order $n'$ obtained from $G$ by deleting the vertices $V_2^{13}\setminus\{v_8\}$, and let $G'$ have $k'$ vertices of degree~$2$. We note that $n' = n-11$ and $k' = k - 3$. Since $n \ge 13$, we have $n' \ge 2$. If $2 \le n' \le 4$, then $\{v_1,v_2,v_9\}$ is a 2DD-set of $G$, and hence $\gamma_2^d(G)\le 3 \le \floor*{\frac{2}{9}(n+k)}$, a contradiction. If $5 \le n' \le 7$, then by \cref{obs3}, there exists a 2DD-set $D'$ of $G'$ such that $|D'|= 2$ and $v_{1}\in D'$. Therefore, $D'\cup \{v_2,v_9\}$ is a 2DD-set of $G$, and so $\gamma_2^d(G)\le 4 \le  \floor*{\frac{2}{9}(n+k)}$, a contradiction. Hence, $n' \ge 8$. By the minimality of the mop $G$, we have $\gamma_2^d(G_1) \le \floor*{\frac{2}{9}(n'+k')} = \floor*{\frac{2}{9}(n-11+k-3)} \le \floor*{\frac{2}{9}(n+k)}-3$. Let $D'$ be a $\gamma_2^d$-set of $G'$ and let $D=D'\cup\{v_1,v_2,v_9\}$. The set $D$ is a 2DD-set of $G$, and so $\gamma_2^d(G) \le |D| \le |D'| + 3 \le \floor*{\frac{2}{9}(n+k)}$, a contradiction.
	\end{proof}

	\begin{claim}\label{tree14}
		The tree $T_{14}$ is not a maximal subtree of $T$.
	\end{claim}
	\begin{proof}[Proof of Claim~\ref{tree14}]
		Suppose, to the contrary, that $T_{14}$ is a maximal subtree of $T$, and so $T_{14} = T_v$. We infer that the subgraph of $G$ associated with $T_{14}$ is obtained from region $H_{13}$ and by triangulating the region $v_2v_{8}v_{13}v_9v_{10}v_{11}v_{12}$ according to \Cref{cl-10to21} as illustrated in Figure~\ref{tree_14}(a)-(d), where $v$ denotes the root of the subtree $T_v$ and where in this case we let $V(T_v) = \{v_1,v_2,v_{8}\}$ be the (shaded) triangle in $G$ associated with the vertex~$v$. In the following, we present arguments that work in each case.
		
		\begin{figure}[htbp]
			\begin{center}
				\includegraphics[scale=0.37]{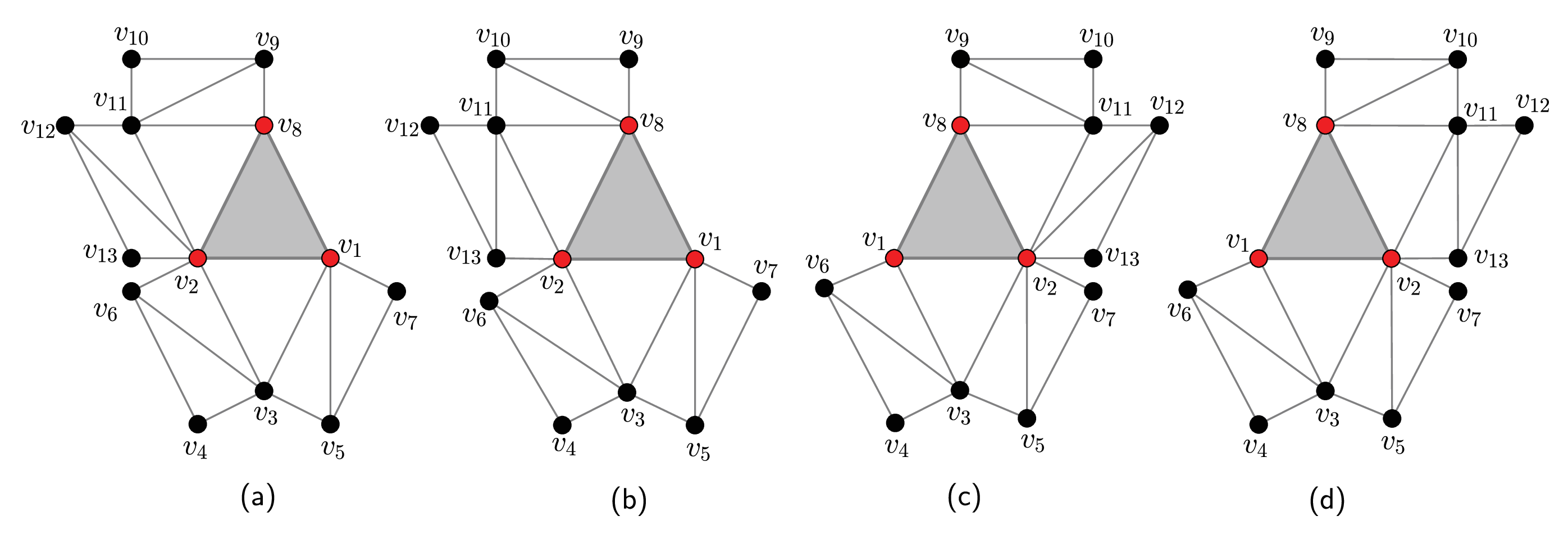}
				\caption{The regions of $G$ corresponding to tree $T_{14}$. The red vertices show a 2DD-set of $G[V(G)\setminus V(G')]$.}\label{tree_14}
			\end{center}
		\end{figure}
		
		Let $G'$ be a graph of order $n'$ obtained from $G$ by deleting the vertices in $V_2^{13}\setminus\{v_8\}$, and let $G'$ have $k'$ vertices of degree~$2$. We note that $n' = n-11$ and $k' = k - 3$. Since $n \ge 13$, we have $n' \ge 2$. If $2 \le n' \le 4$, then $\{v_1,v_2,v_8\}$ is a 2DD-set of $G$, and hence $\gamma_2^d(G)\le 2 \le  \floor*{\frac{2}{9}(n+k)}$, a contradiction. If $5 \le n' \le 7$, then by \cref{obs3}, there exists a 2DD-set $D'$ of $G'$ such that $|D'|= 2$ and $v_{1}\in D'$. Therefore, $D'\cup \{v_2,v_8\}$ is a 2DD-set of $G$, and so $\gamma_2^d(G)\le 4 \le  \floor*{\frac{2}{9}(n+k)}$, a contradiction. Hence, $n' \ge 8$. By the minimality of the mop $G$, we have $\gamma_2^d(G_1) \le \floor*{\frac{2}{9}(n'+k')} = \floor*{\frac{2}{9}(n-11+k-3)} \le \floor*{\frac{2}{9}(n+k)}-3$. Let $D'$ be a $\gamma_2^d$-set of $G'$ and let $D=D'\cup\{v_1,v_2,v_8\}$. The set $D$ is a 2DD-set of $G$, and so $\gamma_2^d(G) \le |D| \le |D'| + 3 \le \floor*{\frac{2}{9}(n+k)}$, a contradiction.
	\end{proof}

	\begin{claim}\label{tree15}
		The tree $T_{15}$ is not a maximal subtree of $T$.
	\end{claim}
	\begin{proof}[Proof of Claim~\ref{tree15}]
		Suppose, to the contrary, that $T_{15}$ is a maximal subtree of $T$, and so $T_{15} = T_v$ where $v$ denotes the root of the subtree $T_v$. We infer that the subgraph of $G$ associated with $T_{15}$ is obtained from region $H_{13}$ in four possible ways, as illustrated in Figure~\ref{tree_15}(a)-(b), where we let $V(T_v) = \{v_2,v_{8},v_{9}\}$ be the (shaded) triangle in $G$ associated with the vertex~$v$.
		
		\begin{figure}[htbp]
			\begin{center}
				\includegraphics[scale=0.43]{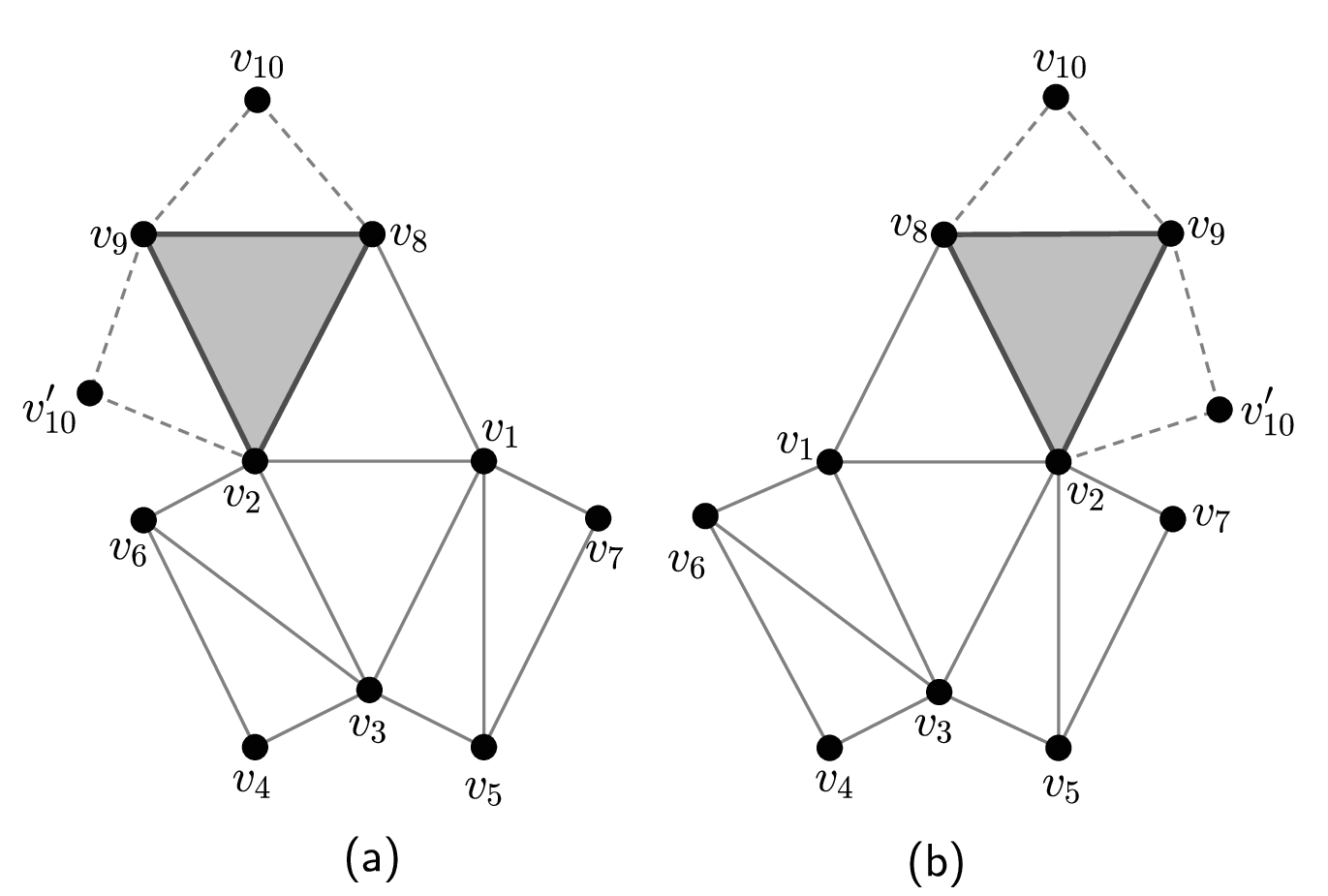}
				\caption{The regions of $G$ corresponding to tree $T_{15}$.}\label{tree_15}
			\end{center}
		\end{figure}
		
		Suppose that $V(T_v) = \{v_2,v_{8},v_{9}\}$ is an internal triangle of $G$. Let $G'$ be a graph of order $n'$ obtained from $G$ by deleting the vertices in $V_1^{7}\setminus\{v_2\}$, and let $G'$ have $k'$ vertices of degree~$2$. We note that $n' = n-6$ and $k' = k - 2$, and $v_2v_{8}$ is an outer edge of $G'$. Since $n \ge 13$, we have $n' \ge 7$. If $n'=7$, then by \cref{obs3}, there exists a 2DD-set $D'$ of $G'$ such that $|D'|= 2$ and $v_{2}\in D'$, and so $D'\cup \{v_1\}$ is a 2DD-set of $G$, and hence $\gamma_2^d(G)\le 3 \le  \floor*{\frac{2}{9}(n+k)}$, a contradiction. Hence, $n' \ge 8$. Let $G_1$ be a graph of order $n_1$ obtained from $G'$ by contracting the edge $v_2v_8$ to form a new vertex $x$ in $G_1$, and let $G_1$ have $k_1$ vertices of degree~$2$. By Lemma~\ref{key}, $G_1$ is a mop. Since $n' \ge 8$, we note that $n_1 = n' - 1 \ge 7$. By the minimality of the mop $G$, we have $\gamma_2^d(G_1) \le \floor*{\frac{2}{9}(n_1+k_1)} \le \floor*{\frac{2}{9}(n-7+k-2)} \le \floor*{\frac{2}{9}(n+k)}-2$. Let $D_1$ be a $\gamma_2^d$-set of $G_1$. If $x \in D_1$, then let $D = (D_1 \setminus \{x\}) \cup \{v_1,v_2,v_8\}$. If $x \notin D_1$, then let $D = D_1 \cup \{v_1,v_2\}$. In both cases $D$ is a 2DD-set of $G$, and so $\gamma_2^d(G) \le |D| \le |D_1| + 2 \le \floor*{\frac{2}{9}(n+k)}$, a contradiction.

		Hence, $V(T_v) = \{v_2,v_{8},v_{9}\}$ is not an internal triangle of $G$. Since $n \ge 13$, there exists a triangle $F$ adjacent to face $v_2v_8v_{9}$. There are two possible triangles that can be formed: either $V(F)=\{v_8,v_{9},v_{10}\}$ or $V(F)=\{v_2,v_{9},v'_{10}\}$. These are illustrated with dotted lines in \Cref{tree_15}(a)-(b).
		
		Suppose firstly that $V(F)=\{v_8,v_{9},v_{10}\}$. In this case, we let $G'$ be the mop of order $n'$ obtained from $G$ by deleting the vertices in $V_1^{7}$, and let $G'$ have $k'$ vertices of degree~$2$. We note that $n' = n-7$ and $k' \le k - 1$, and $v_8v_{9}$ is an outer edge of $G'$. Since $n \ge 13$, we have $n' \ge 6$. If $6 \le n' \le 7$, then by \cref{obs3}, there exists a 2DD-set $D'$ of $G'$ such that $v_8\in D'$ and $|D'|=2$. Therefore, $D'\cup \{v_3\}$ is a 2DD-set of $G$, and so $\gamma_2^d(G)\le 3\le  \floor*{\frac{2}{9}(n+k)}$, a contradiction. Hence, $n' \ge 8$. Let $G_1$ be a graph of order $n_1$ obtained from $G'$ by contracting the edge $v_8v_{9}$ to form a new vertex $x$ in $G_1$, and let $G_1$ have $k_1$ vertices of degree~$2$. By Lemma~\ref{key}, $G_1$ is a mop. Since $n' \ge 8$, we note that $n_1 = n' - 1 \ge 7$. Further we note that $n_1 = n - 8$ and $k_1 \le k-1$. By the minimality of the mop $G$, we have $\gamma_2^d(G_1) \le \floor*{\frac{2}{9}(n_1+k_1)} \le \floor*{\frac{2}{9}(n-8+k-1)} \le \floor*{\frac{2}{9}(n+k)}-2$. Let $D_1$ be a $\gamma_2^d$-set of $G_1$. If $x \in D_1$, then let $D = (D_1 \setminus \{x\}) \cup \{v_3,v_8,v_9\}$. If $x \notin D_1$, then let $D = D_1 \cup \{v_3,v_8\}$. In both cases $D$ is a 2DD-set of $G$, and so $\gamma_2^d(G) \le |D| \le |D_1| + 2 \le \floor*{\frac{2}{9}(n+k)}$, a contradiction.
		
		Hence, $V(F)=\{v_2,v_{9},v'_{10}\}$. Let $G'$ be the mop of order $n'$ obtained from $G$ by deleting the vertices in $V_1^{8} \setminus \{v_2\}$, and let $G'$ have $k'$ vertices of degree~$2$. We note that $n' = n-7$ and $k' \le k - 1$, and $v_2v_{9}$ is an outer edge of $G'$. Since $n \ge 13$, we have $n' \ge 6$. If $6 \le n' \le 7$, then by \cref{obs3}, there exists a 2DD-set $D'$ of $G'$ such that $v_9\in D'$ and $|D'|=2$. Therefore, $D'\cup \{v_3\}$ is a 2DD-set of $G$, and so $\gamma_2^d(G)\le 3 \le  \floor*{\frac{2}{9}(n+k)}$, a contradiction. Hence, $n' \ge 8$. Let $G_1$ be a graph of order $n_1$ obtained from $G'$ by contracting the edge $v_2v_{9}$ to form a new vertex $x$ in $G_1$, and let $G_1$ have $k_1$ vertices of degree~$2$. By Lemma~\ref{key}, $G_1$ is a mop. Since $n' \ge 8$, we note that $n_1 = n' - 1 \ge 7$. Further we note that $n_1 = n - 8$ and $k_1 \le k-1$. By the minimality of the mop $G$, we have $\gamma_2^d(G_1) \le \floor*{\frac{2}{9}(n_1+k_1)} \le \floor*{\frac{2}{9}(n-8+k-1)} \le \floor*{\frac{2}{9}(n+k)}-2$. Let $D_1$ be a $\gamma_2^d$-set of $G_1$. If $x \in D_1$, then let $D = (D_1 \setminus \{x\}) \cup \{v_3,v_8,v_9\}$. If $x \notin D_1$, then let $D = D_1 \cup \{v_3,v_9\}$. In both cases $D$ is a 2DD-set of $G$, and so $\gamma_2^d(G)\le |D| \le |D_1| + 2 \le \floor*{\frac{2}{9}(n+k)}$, a contradiction.
	\end{proof}
	
	\begin{figure}[htbp]
		\begin{center}
			\includegraphics[scale=0.185]{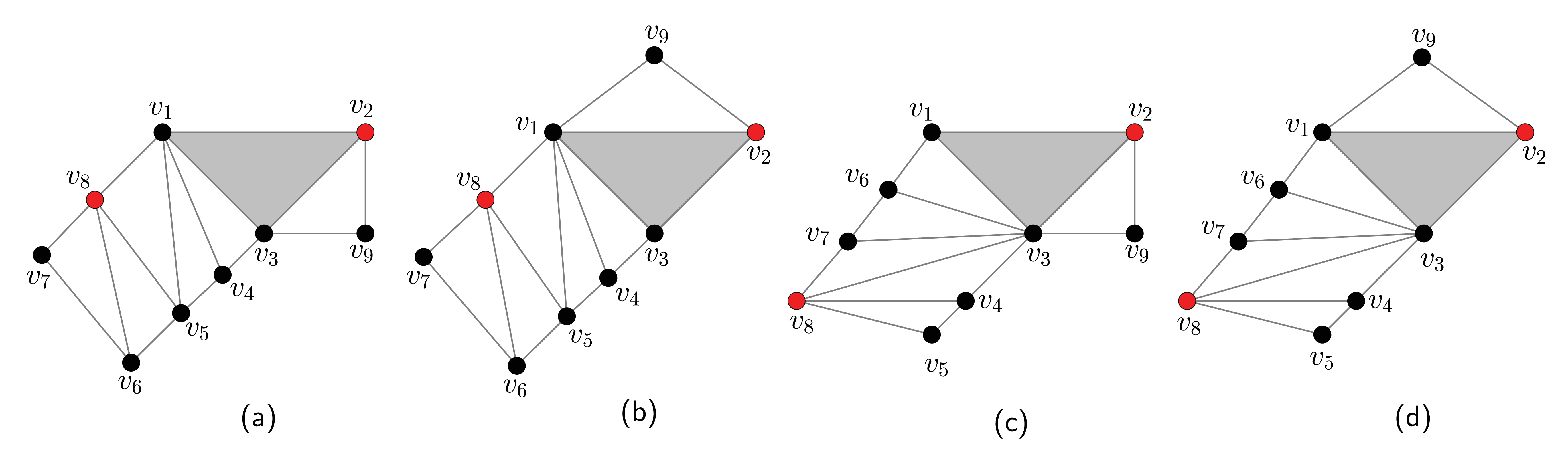}
			\caption{The regions of $G$ corresponding to tree $T_{16}$. The red vertices show a 2DD-set of $G[V(G)\setminus V(G')]$.}\label{tree_16}
		\end{center}
	\end{figure}

	\begin{claim}\label{tree16}
		The tree $T_{16}$ is not a maximal subtree of $T$.
	\end{claim}
	\begin{proof}[Proof of Claim~\ref{tree16}]
		Suppose, to the contrary, that $T_{16}$ is a maximal subtree of $T$, and so $T_{16} = T_v$ where $v$ denotes the root of the subtree $T_v$. We infer that the subgraph of $G$ associated with $T_{16}$ is obtained from either (i) the region $H_1$ by triangulating the region $v_1v_3v_4v_5v_6v_7v_8$  according to \Cref{lem1}\ref{5distance} as illustrated in Figure~\ref{tree_16}(a)-(b) or (ii) the region $H_2$ by triangulating the region $v_1v_3v_4v_5v_8v_7v_6$ according to \Cref{lem1}\ref{5distance} as illustrated in Figure~\ref{tree_16}(c)-(d), where we let $V(T_v) = \{v_1,v_2,v_{3}\}$ be the (shaded) triangle in $G$ associated with the vertex~$v$. In the following, we present arguments that work in each cases.
		
		Let $G'$ be the mop of order $n'$ obtained from $G$ by deleting the vertices in $V_3^{9}$, and let $G'$ have $k'$ vertices of degree~$2$. We note that $n' = n-7$ and $k' \le k - 1$, and $v_1v_2$ is an outer edge of $G'$. Since $n \ge 13$, we have $n' \ge 6$. If $6 \le n' \le 7$, then by \cref{obs3}, there exists a 2DD-set $D'$ of $G'$ such that $v_2\in D'$ and $|D'|=2$. Therefore, $D' \cup \{v_8\}$ is a 2DD-set of $G$, and so $\gamma_2^d(G)\le 3\le  \floor*{\frac{2}{9}(n+k)}$, a contradiction. Hence, $n' \ge 8$. Let $G_1$ be a graph of order $n_1$ obtained from $G'$ by contracting the edge $v_1v_2$ to form a new vertex $x$ in $G_1$, and let $G_1$ have $k_1$ vertices of degree~$2$. By Lemma~\ref{key}, $G_1$ is a mop. Since $n' \ge 8$, we note that $n_1 = n' - 1 \ge 7$. Further we note that $n_1 = n - 8$ and $k_1 \le k-1$. By the minimality of the mop $G$, we have $\gamma_2^d(G_1) \le \floor*{\frac{2}{9}(n_1+k_1)} \le \floor*{\frac{2}{9}(n-8+k-1)} \le \floor*{\frac{2}{9}(n+k)}-2$. Let $D_1$ be a $\gamma_2^d$-set of $G_1$. If $x \in D_1$, then let $D = (D_1 \setminus \{x\}) \cup \{v_1,v_2,v_8\}$. If $x \notin D_1$, then let $D = D_1 \cup \{v_2,v_8\}$. In both cases $D$ is a 2DD-set of $G$, and so $\gamma_2^d(G) \le |D| \le |D_1| + 2 \le \floor*{\frac{2}{9}(n+k)}$, a contradiction.
	\end{proof}

	\begin{claim}\label{tree17}
		The tree $T_{17}$ is not a maximal subtree of $T$.
	\end{claim}
	\begin{proof}[Proof of Claim~\ref{tree17}]
		Suppose, to the contrary, that $T_{17}$ is a maximal subtree of $T$, and so $T_{17} = T_v$ where $v$ denotes the root of the subtree $T_v$. We infer that the subgraph of $G$ associated with $T_{17}$ is obtained from either (i) the region $H_1$ by triangulating the region $v_1v_3v_4v_5v_6v_7v_8$  according to \Cref{lem1}\ref{5distance} as illustrated in Figure~\ref{tree_17}(a)-(b) or (ii) the region $H_2$ by triangulating the region $v_1v_3v_4v_5v_8v_7v_6$ according to \Cref{lem1}\ref{5distance} as illustrated in Figure~\ref{tree_17}(c)-(d), where we let $V(T_v) = \{v_1,v_2,v_{3}\}$ be the (shaded) triangle in $G$ associated with the vertex~$v$. The region $v_2v_9v_{10}v_3$ can be triangulated by adding either the edge $v_2v_{10}$ or $v_3v_{9}$, as indicated by the dotted lines in Figure~\ref{tree_17}(a)-(d). In the following, we present arguments that work in each cases.
		
		Let $G'$ be the mop of order $n'$ obtained from $G$ by deleting the vertices in $V_3^{10}$, and let $G'$ have $k'$ vertices of degree~$2$. We note that $n' = n-8$ and $k' \le k - 1$, and $v_1v_2$ is an outer edge of $G'$. Since $n \ge 13$, we have $n' \ge 5$. If $5 \le n' \le 7$, then by \cref{obs3}, there exists a 2DD-set $D'$ of $G'$ such that $v_3\in D' $ or $v_2\in D'$ and $|D'|=2$. Therefore, $D'\cup \{v_8\}$ is a 2DD-set of $G$, and so $\gamma_2^d(G)\le 3\le  \floor*{\frac{2}{9}(n+k)}$, a contradiction. Hence, $n' \ge 8$. By the minimality of the mop $G$, we have $\gamma_2^d(G_1) \le \floor*{\frac{2}{9}(n'+k')} = \floor*{\frac{2}{9}(n-8+k-1)} \le \floor*{\frac{2}{9}(n+k)}-2$. Let $D'$ be a $\gamma_2^d$-set of $G'$ and let $D=D'\cup\{v_2,v_7\}$ or $D=D'\cup\{v_3,v_8\}$. The set $D$ is a 2DD-set of $G$, and so $\gamma_2^d(G) \le |D| \le |D'| + 2 \le \floor*{\frac{2}{9}(n+k)}$, a contradiction.
	\end{proof}
	
	\begin{figure}[htbp]
		\begin{center}
			\includegraphics[scale=0.185]{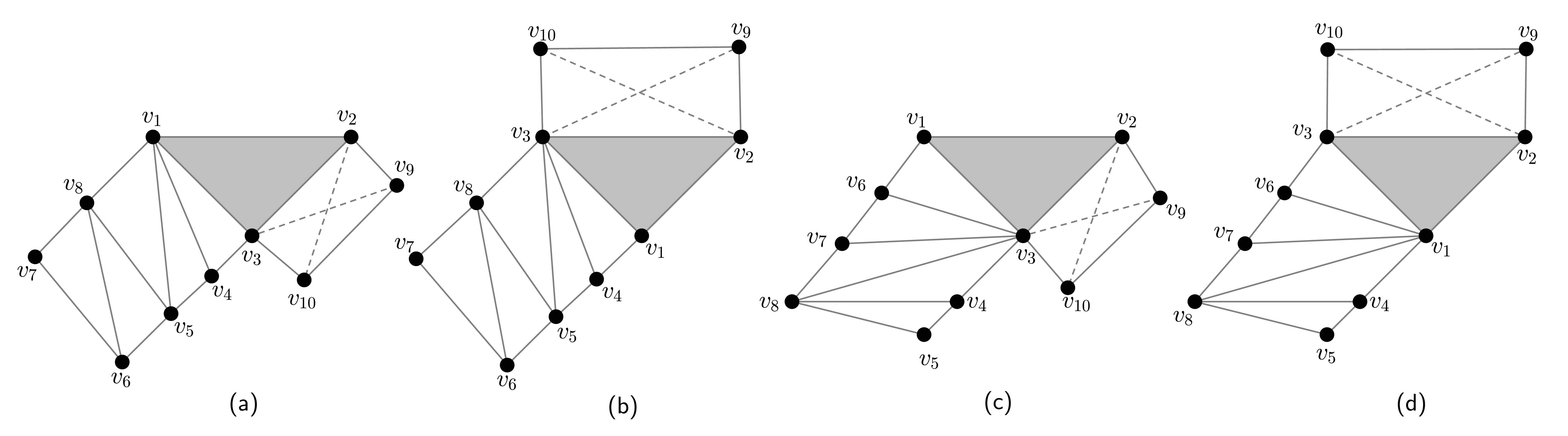}
			\caption{The regions of $G$ corresponding to tree $T_{17}$.}\label{tree_17}
		\end{center}
	\end{figure}
	
	\begin{figure}[htbp]
		\begin{center}
			\includegraphics[scale=0.16]{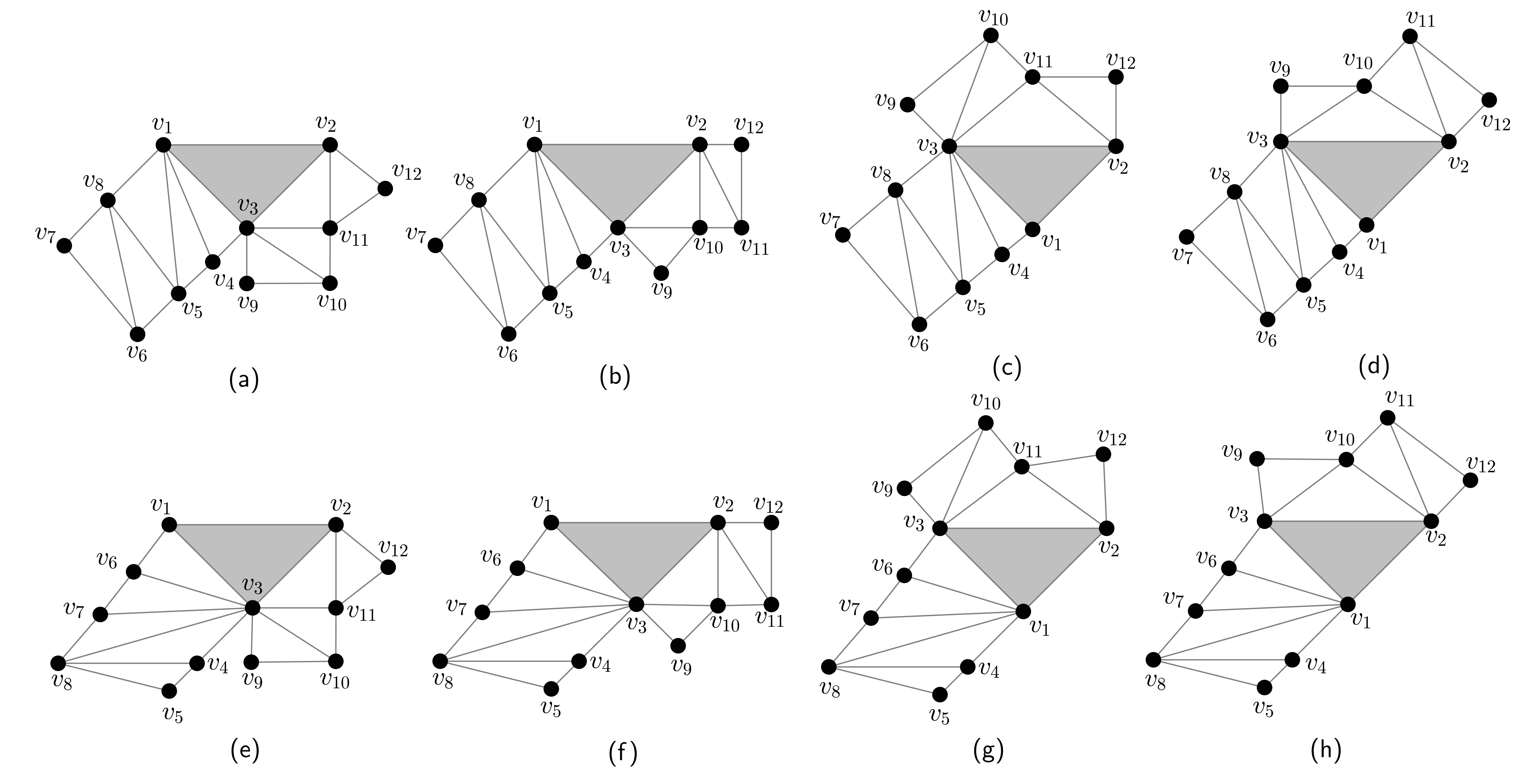}
			\caption{The regions of $G$ corresponding to tree $T_{18}$. The red vertices show a 2DD-set of $G[V(G)\setminus V(G')]$.}\label{tree_18}
		\end{center}
	\end{figure}

	\begin{claim}\label{tree18}
		The tree $T_{18}$ is not a maximal subtree of $T$.
	\end{claim}
	\begin{proof}[Proof of Claim~\ref{tree18}]
		Suppose, to the contrary, that $T_{18}$ is a maximal subtree of $T$, and so $T_{18} = T_v$ where $v$ denotes the root of the subtree $T_v$. We infer that the subgraph of $G$ associated with $T_{18}$ is obtained from either (i) the region $H_1$ by triangulating the region $v_1v_3v_4v_5v_6v_7v_8$  according to \Cref{lem1}\ref{5distance} as illustrated in Figure~\ref{tree_18}(a)-(d) or (ii) the region $H_2$ by triangulating the region $v_1v_3v_4v_5v_8v_7v_6$ according to \Cref{lem1}\ref{5distance} as illustrated in Figure~\ref{tree_18}(e)-(h), where we let $V(T_v) = \{v_1,v_2,v_{3}\}$ be the (shaded) triangle in $G$ associated with the vertex~$v$. The region $v_2v_3v_{9}v_{10}v_{11}v_{12}$ is triangulated as region $H_{10}$ according to \Cref{cl3to10} as illustrated in Figure~\ref{tree_18}(a)-(h). In the following, we present arguments that work in each cases.
		
		Let $G'$ be the mop of order $n'$ obtained from $G$ by deleting the vertices in $V_4^{8}$, and let $G'$ have $k'$ vertices of degree~$2$. We note that $n' = n-5$ and $k' = k - 1$. Since $n \ge 13$, we have $n' \ge 8$.  By the minimality of the mop $G$, we have $\gamma_2^d(G_1) \le \floor*{\frac{2}{9}(n'+k')} = \floor*{\frac{2}{9}(n-5+k-1)} \le \floor*{\frac{2}{9}(n+k)}-1$. Let $D'$ be a $\gamma_2^d$-set of $G'$. Since $v_9$ is disjunctive dominated by some vertex of $D'$, we have $D' \cap \{v_1,v_2,v_3,v_9,v_{10}\} \ne \emptyset$. We now consider the set $D=D' \cup \{v_8\}$. The resulting set $D$ is a 2DD-set of $G$, and so $\gamma_2^d(G) \le |D| \le |D'| + 1\le \floor*{\frac{2}{9}(n+k)}$, a contradiction.
	\end{proof}
	
	\begin{figure}
		\begin{center}
			\includegraphics[scale=0.16]{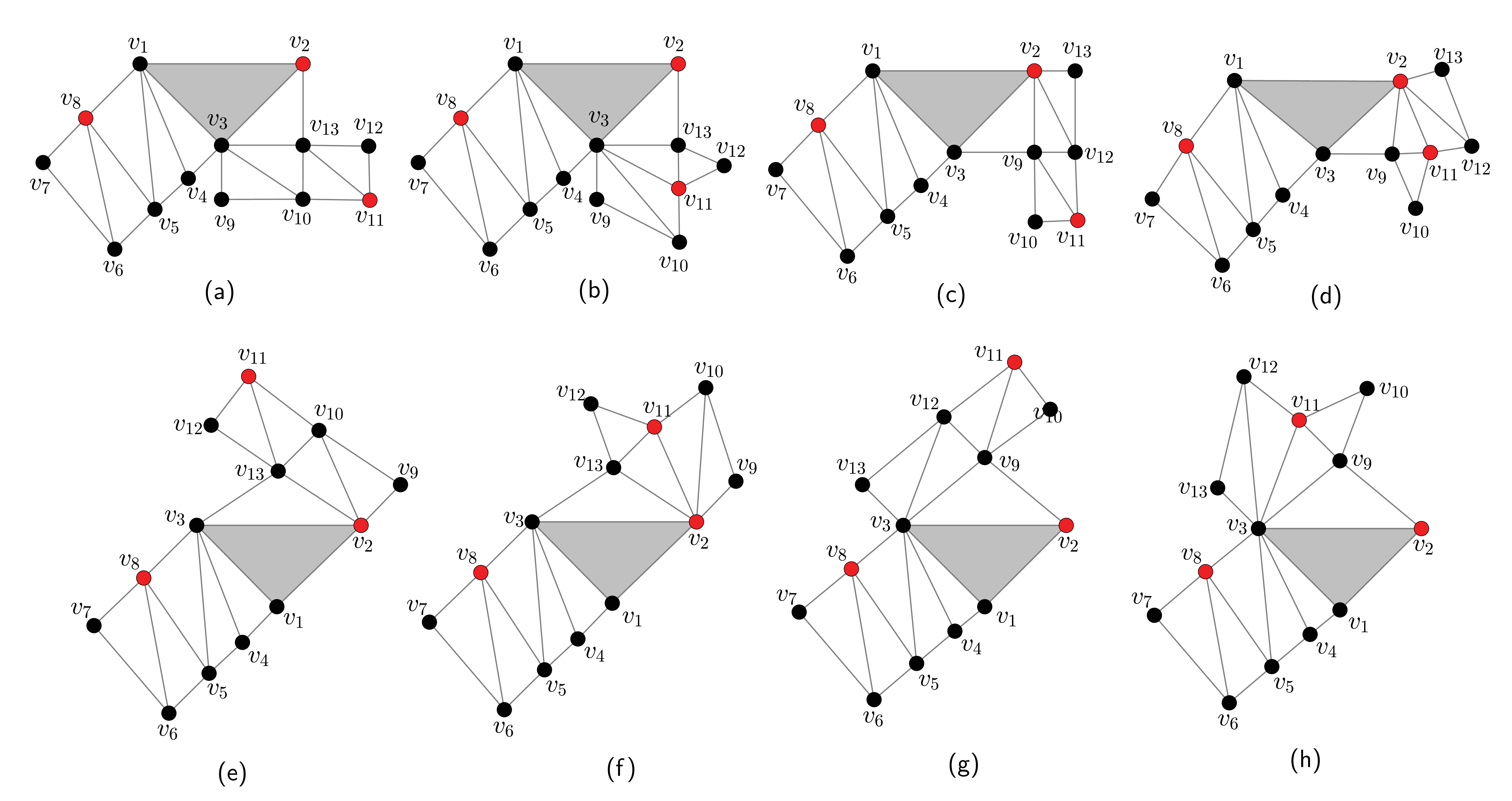}
			\includegraphics[scale=0.16]{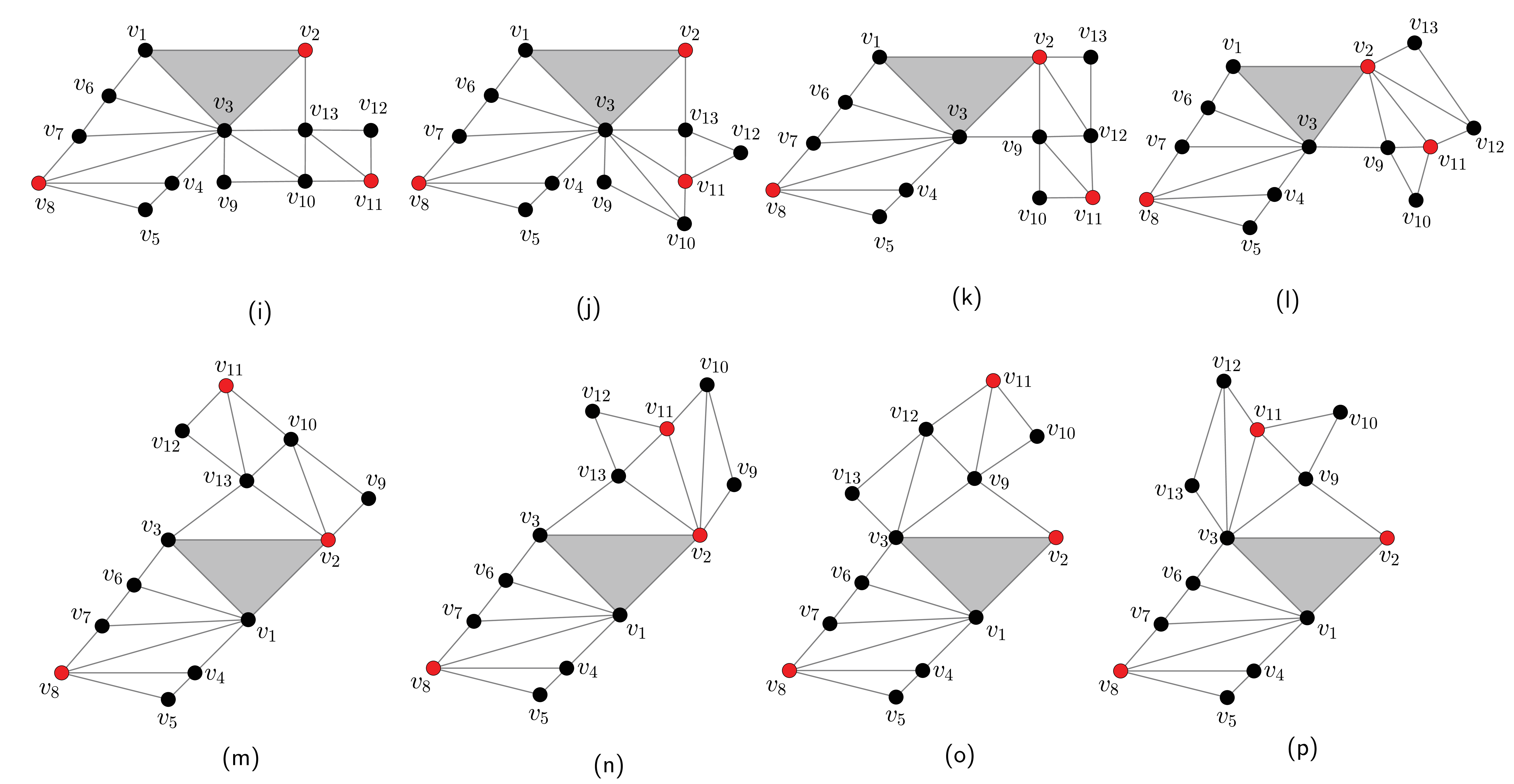}
			\caption{The regions of $G$ corresponding to tree $T_{19}$. The red vertices show a 2DD-set of $G[V(G)\setminus V(G')]$.}\label{tree_19}
		\end{center}
	\end{figure}

	\begin{claim}\label{tree19}
		The tree $T_{19}$ is not a maximal subtree of $T$.
	\end{claim}
	\begin{proof}[Proof of Claim~\ref{tree19}]
		Suppose, to the contrary, that $T_{19}$ is a maximal subtree of $T$, and so $T_{19} = T_v$ where $v$ denotes the root of the subtree $T_v$. We infer that the subgraph of $G$ associated with $T_{19}$ is obtained from either (i) the region $H_1$ by triangulating the region $v_1v_3v_4v_5v_6v_7v_8$  according to \Cref{lem1}\ref{5distance} as illustrated in Figure~\ref{tree_19}(a)-(h) or (ii) the region $H_2$ by triangulating the region $v_1v_3v_4v_5v_8v_7v_6$ according to \Cref{lem1}\ref{5distance} as illustrated in Figure~\ref{tree_19}(i)-(p), where we let $V(T_v) = \{v_1,v_2,v_{3}\}$ be the (shaded) triangle in $G$ associated with the vertex~$v$. In the following, we present arguments that work in each cases.

		Let $G'$ be the mop of order $n'$ obtained from $G$ by deleting the vertices in $V_3^{13}$, and let $G'$ have $k'$ vertices of degree~$2$. We note that $n' = n-11$ and $k' = k - 2$, and $v_1v_2$ is an outer edge of $G'$. Since $n \ge 13$, we have $n' \ge 2$. If $2 \le n' \le 4$, then $\{v_1,v_8,v_{11}\}$ or  $\{v_2,v_8,v_{11}\}$ is a 2DD-set of $G$, and hence $\gamma_2^d(G)\le 3 \le  \floor*{\frac{2}{9}(n+k)}$, a contradiction. If $5 \le n' \le 7$, then by \cref{obs3}, there exists a 2DD-set $D'$ of $G'$ such that $v_2\in D'$ and $|D'|=2$. Therefore, $D'\cup \{v_8,v_{11}\}$ is a 2DD-set of $G$, and so $\gamma_2^d(G)\le 4 \le \floor*{\frac{2}{9}(n+k)}$, a contradiction. Hence, $n' \ge 8$. Let $G_1$ be a graph of order $n_1$ obtained from $G'$ by contracting the edge $v_1v_2$ to form a new vertex $x$ in $G_1$, and let $G_1$ have $k_1$ vertices of degree~$2$. By Lemma~\ref{key}, $G_1$ is a mop. Since $n' \ge 8$, we note that $n_1 = n' - 1 \ge 7$. Further we note that $n_1 = n - 12$ and $k_1 \le k-2$. By the minimality of the mop $G$, we have $\gamma_2^d(G_1) \le \floor*{\frac{2}{9}(n_1+k_1)} \le \floor*{\frac{2}{9}(n-12+k-2)} \le \floor*{\frac{2}{9}(n+k)}-3$. Let $D_1$ be a $\gamma_2^d$-set of $G_1$. If $x \in D_1$, then let $D = (D_1 \setminus \{x\}) \cup \{v_1,v_2,v_8,v_{11}\}$. If $x \notin D_1$, then let $D = D_1 \cup \{v_2,v_8,v_{11}\}$. In both cases $D$ is a 2DD-set of $G$, and so $\gamma_2^d(G) \le |D| \le |D_1| + 3 \le \floor*{\frac{2}{9}(n+k)}$, a contradiction.
	\end{proof}

	\begin{figure}[htbp]
		\begin{center}
			\includegraphics[scale=0.164]{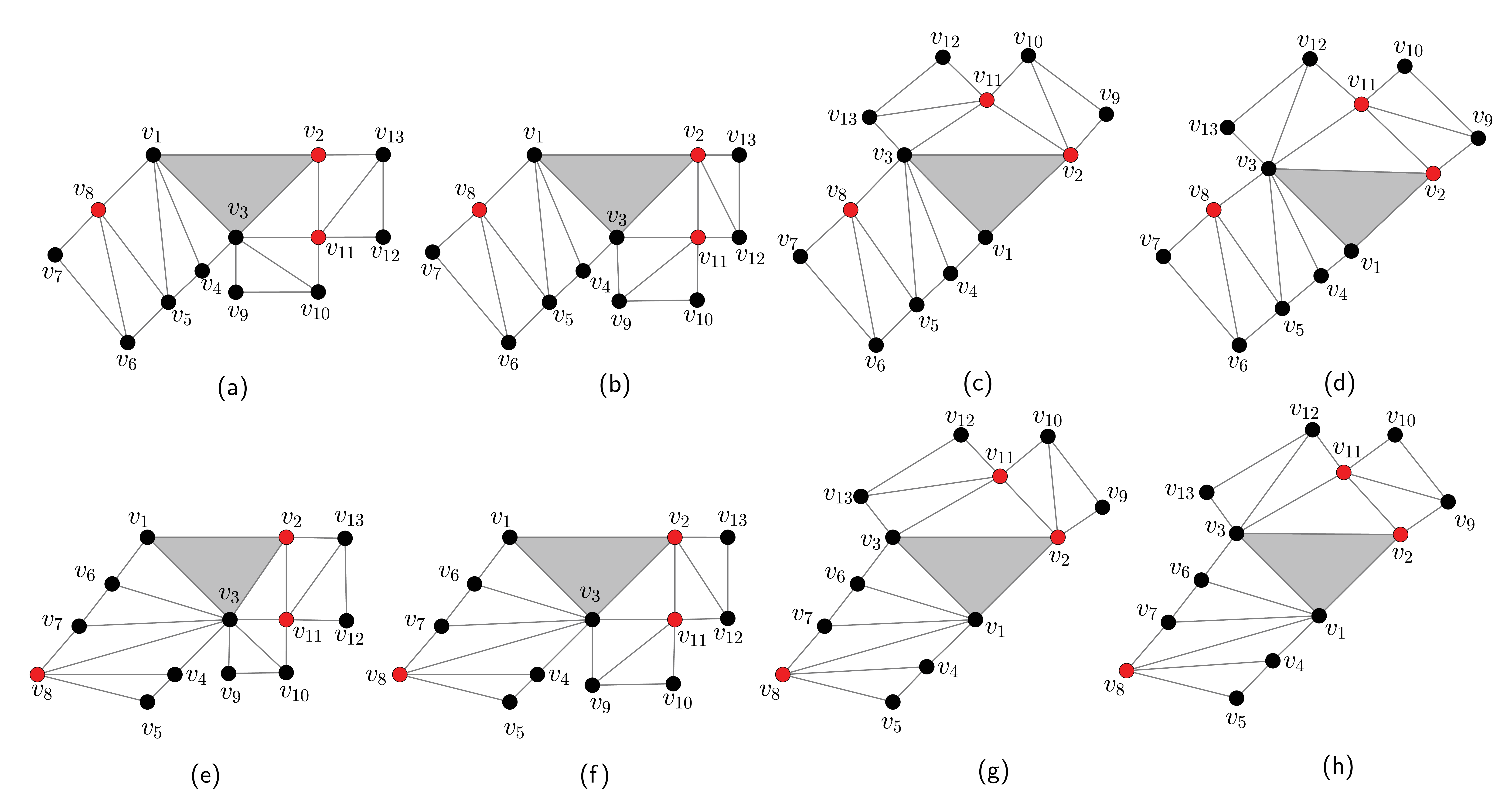}
			\caption{The regions of $G$ corresponding to tree $T_{20}$. The red vertices show a 2DD-set of $G[V(G)\setminus V(G')]$.}\label{tree_20}
		\end{center}
	\end{figure}

	\begin{claim}\label{tree20}
		The tree $T_{20}$ is not a maximal subtree of $T$.
	\end{claim}
	\begin{proof}[Proof of Claim~\ref{tree20}]
		Suppose, to the contrary, that $T_{20}$ is a maximal subtree of $T$, and so $T_{20} = T_v$ where $v$ denotes the root of the subtree $T_v$.  We infer that the subgraph of $G$ associated with $T_{20}$ is obtained from either (i) the region $H_1$ by triangulating the region $v_1v_3v_4v_5v_6v_7v_8$  according to \Cref{lem1}\ref{5distance} as illustrated in Figure~\ref{tree_20}(a)-(d) or (ii) the region $H_2$ by triangulating the region $v_1v_3v_4v_5v_8v_7v_6$ according to \Cref{lem1}\ref{5distance} as illustrated in Figure~\ref{tree_20}(e)-(h), where we let $V(T_v) = \{v_1,v_2,v_{3}\}$ be the (shaded) triangle in $G$ associated with the vertex~$v$. The region $v_2v_3v_{9}v_{10}v_{11}v_{12}v_{13}$ is triangulated as region $H_{13}$ according to \Cref{cl-10to21} as illustrated in Figure~\ref{tree_20}(a)-(h). In the following, we present arguments that work in each cases.

		Let $G'$ be the mop of order $n'$ obtained from $G$ by deleting the vertices in $V_{3}^{13}$, and let $G'$ have $k'$ vertices of degree~$2$. We note that $n' = n-11$ and $k' = k - 2$, and $v_1v_2$ is an outer edge of $G'$. Since $n \ge 13$, we have $n' \ge 2$. If $2 \le n' \le 4$, then $\{v_1,v_8,v_{11}\}$ or $\{v_2,v_8,v_{11}\}$ is a 2DD-set of $G$, and hence $\gamma_2^d(G)\le 3 \le  \floor*{\frac{2}{9}(n+k)}$, a contradiction. If $4\le n' \le 7$, then by \cref{obs3}, there exists a 2DD-set $D'$ of $G'$ such that $v_2\in D'$ and $|D'|=2$. Therefore, $D'\cup \{v_8,v_{11}\}$ is a 2DD-set of $G$, and so $\gamma_2^d(G)\le 4 \le  \floor*{\frac{2}{9}(n+k)}$, a contradiction. Hence, $n' \ge 8$. Let $G_1$ be a graph of order $n_1$ obtained from $G'$ by contracting the edge $v_1v_2$ to form a new vertex $x$ in $G_1$, and let $G_1$ have $k_1$ vertices of degree~$2$. By Lemma~\ref{key}, $G_1$ is a mop. Since $n' \ge 8$, we note that $n_1 = n' - 1 \ge 7$. Further we note that $n_1 = n - 12$ and $k_1 \le k-2$. By the minimality of the mop $G$, we have $\gamma_2^d(G_1) \le \floor*{\frac{2}{9}(n_1+k_1)} \le \floor*{\frac{2}{9}(n-12+k-2)} \le \floor*{\frac{2}{9}(n+k)}-3$. Let $D_1$ be a $\gamma_2^d$-set of $G_1$. If $x \in D_1$, then let $D = (D_1 \setminus \{x\}) \cup \{v_1,v_2,v_8,v_{11}\}$. If $x \notin D_1$, then let $D = D_1 \cup \{v_2,v_8,v_{11}\}$. In both cases $D$ is a 2DD-set of $G$, and so $\gamma_2^d(G) \le |D| \le |D_1| + 3 \le \floor*{\frac{2}{9}(n+k)}$, a contradiction.
	\end{proof}

	\begin{figure}[htbp]
		\begin{center}
			\includegraphics[scale=0.15]{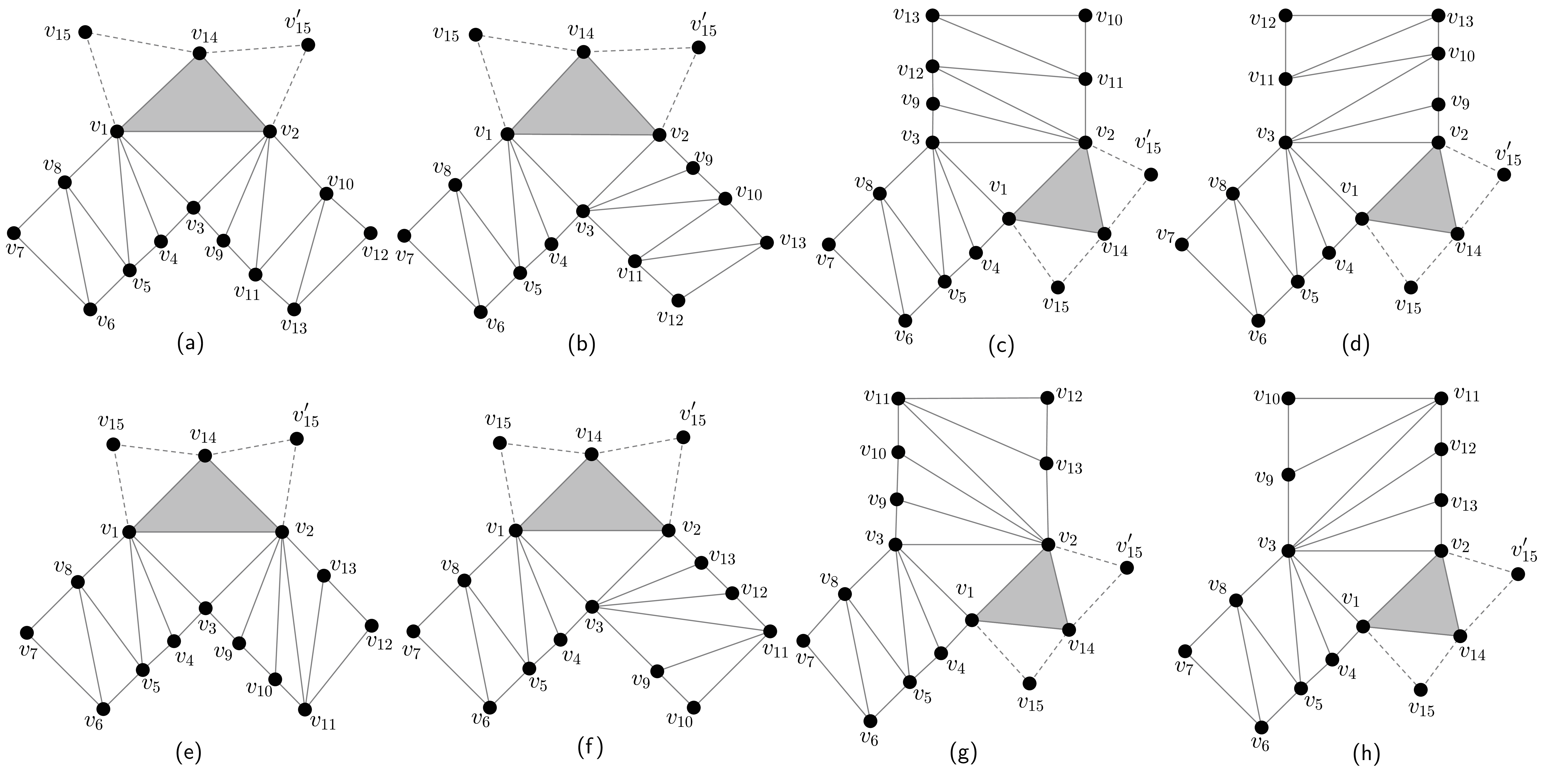}
			\includegraphics[scale=0.153]{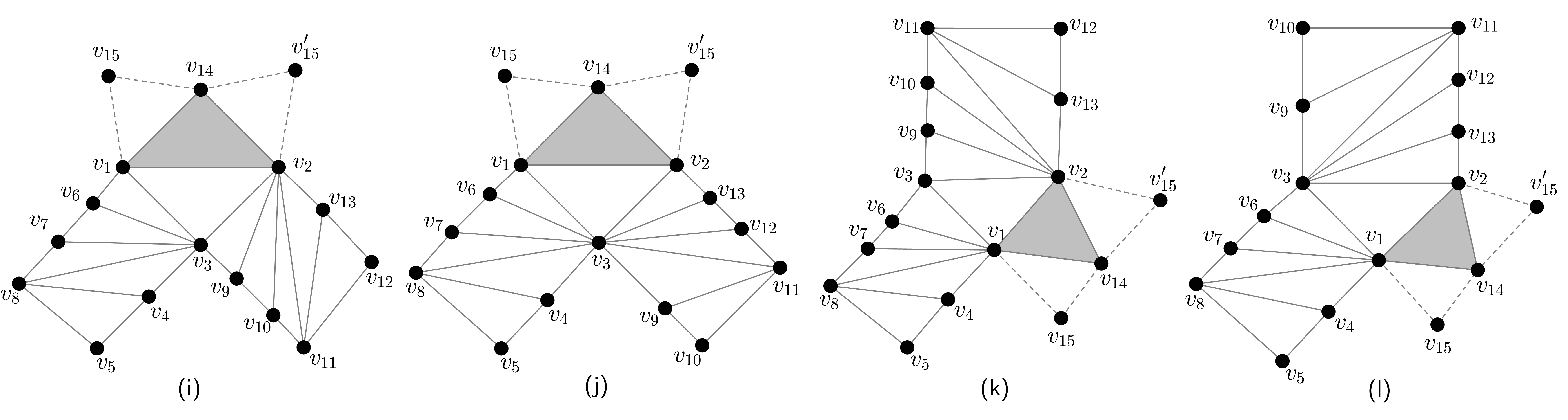}
			\caption{The regions of $G$ corresponding to tree $T_{21}$.}\label{tree_21}
		\end{center}
	\end{figure}

	\begin{claim}\label{tree21}
		The tree $T_{21}$ is not a maximal subtree of $T$.
	\end{claim}
	\begin{proof}[Proof of Claim~\ref{tree21}]
		Suppose, to the contrary, that $T_{21}$ is a maximal subtree of $T$, and so $T_{21} = T_v$ where $v$ denotes the root of the subtree $T_v$. We infer that the subgraph of $G$ associated with $T_{21}$ is obtained from either (i) the region $H_1$ by triangulating the region $v_1v_3v_4v_5v_6v_7v_8$  according to \Cref{lem1}\ref{5distance} as illustrated in Figure~\ref{tree_21}(a)-(h) or (ii) the region $H_2$ by triangulating the region $v_1v_3v_4v_5v_8v_7v_6$ according to \Cref{lem1}\ref{5distance} as illustrated in Figure~\ref{tree_21}(i)-(l), where we let $V(T_v) = \{v_1,v_2,v_{3}\}$ be the triangle in $G$ associated with the vertex~$v$. Recall that $n \ge 13$. If $13\le n\le 14$, then $\{v_2,v_8,v_{11}\}$ is a 2DD-set of $G$. So $n \ge 15$. Therefore there exists a triangle that is adjacent to $V(T_v) = \{v_1,v_2,v_{3}\}$. Let $V(F_1) = \{v_1,v_2,v_{14}\}$ be the (shaded) triangle adjacent to  $V(T_v) = \{v_1,v_2,v_{3}\}$ in $G$.

		Suppose that $V(F_1)=\{v_1,v_2,v_{14}\}$ is an internal triangle of $G$. Let $G'$ be the mop of order $n'$ obtained from $G$ by deleting the vertices in $V_3^{13}$, and let $G'$ have $k'$ vertices of degree~$2$. We note that $n' = n-11$ and $k' = k - 2$, and $v_1v_2$ is an outer edge of $G'$. Since $n \ge 13$, we have $n' \ge 2$. If $2 \le n' \le 4$, then $\{v_2,v_8,v_{11}\}$ or $\{v_1,v_8,v_{11}\}$ is a 2DD-set of $G$, and hence $\gamma_2^d(G)\le 3 \le  \floor*{\frac{2}{9}(n+k)}$, a contradiction. If $5 \le n' \le 7$, then by \cref{obs3}, there exists a 2DD-set $D'$ of $G'$ such that $v_2\in D'$ and $|D'|=2$. Therefore, $D'\cup \{v_8,v_{11}\}$ is a 2DD-set of $G$, and so $\gamma_2^d(G)\le 4 \le \floor*{\frac{2}{9}(n+k)}$, a contradiction. Hence, $n' \ge 8$. Let $G_1$ be a graph of order $n_1$ obtained from $G'$ by contracting the edge $v_1v_2$ to form a new vertex $x$ in $G_1$, and let $G_1$ have $k_1$ vertices of degree~$2$. By Lemma~\ref{key}, $G_1$ is a mop. Since $n' \ge 8$, we note that $n_1 = n' - 1 \ge 7$. Further we note that $n_1 = n - 12$ and $k_1 \le k-2$. By the minimality of the mop $G$, we have $\gamma_2^d(G_1) \le \floor*{\frac{2}{9}(n_1+k_1)} \le \floor*{\frac{2}{9}(n-12+k-2)} \le \floor*{\frac{2}{9}(n+k)}-3$. Let $D_1$ be a $\gamma_2^d$-set of $G_1$. If $x \in D_1$, then let $D = (D_1 \setminus \{x\}) \cup \{v_1,v_2,v_8,v_{11}\}$. If $x \notin D_1$, then let $D = D_1 \cup \{v_2,v_8,v_{11}\}$. In both cases $D$ is a 2DD-set of $G$, and so $\gamma_2^d(G) \le |D| \le |D_1| + 3 \le \floor*{\frac{2}{9}(n+k)}$, a contradiction.
		
		Hence, $V(F_1)=\{v_1,v_2,v_{14}\}$ is not an internal triangle of $G$, implying that $n \ge 15$. Therefore there exists a triangle $F$ adjacent to $V(F_1)=\{v_1,v_2,v_{14}\}$. There are two possible triangles that can be formed: either $V(F)=\{v_1,v_{14},v_{15}\}$ or $V(F)=\{v_1,v_{14},v'_{15}\}$. These are illustrated with dotted lines in \Cref{tree_21}(a)-(l).
		
		Suppose that $V(F)=\{v_1,v_{14},v_{15}\}$. In this case, let $G'$ be a graph of order $n'$ obtained from $G$ by deleting the vertices $V_2^{13}$, and let $G'$ have $k'$ vertices of degree~$2$. We note that $n' = n-12$ and $k' = k - 1$, and $v_1v_{14}$ is an outer edge of $G'$. If $2 \le n' \le 4$, then $\{v_1,v_8,v_{11}\}$ is a 2DD-set of $G$, and hence $\gamma_2^d(G)\le 3 \le  \floor*{\frac{2}{9}(n+k)}$, a contradiction. If $5 \le n' \le 7$, then by \cref{obs3}, there exists a 2DD-set $D'$ of $G'$ such that $v_1\in D'$ and $|D'|=2$. Therefore, $D'\cup \{v_8,v_{11}\}$ is a 2DD-set of $G$, and so $\gamma_2^d(G)\le 4 \le \floor*{\frac{2}{9}(n+k)}$, a contradiction. Hence, $n' \ge 8$. Let $G_1$ be a graph of order $n_1$ obtained from $G'$ by contracting the edge $v_1v_{14}$ to form a new vertex $x$ in $G_1$, and let $G_1$ have $k_1$ vertices of degree~$2$. By Lemma~\ref{key}, $G_1$ is a mop. Since $n' \ge 8$, we note that $n_1 = n' - 1 \ge 7$. Further we note that $n_1 = n - 13$ and $k_1 \le k-1$. By the minimality of the mop $G$, we have $\gamma_2^d(G_1) \le \floor*{\frac{2}{9}(n_1+k_1)} \le \floor*{\frac{2}{9}(n-13+k-1)} \le \floor*{\frac{2}{9}(n+k)}-3$. Let $D_1$ be a $\gamma_2^d$-set of $G_1$. If $x \in D_1$, then let $D = (D_1 \setminus \{x\}) \cup \{v_1,v_8,v_{11},v_{14}\}$. If $x \notin D_1$, then let $D = D_1 \cup \{v_1,v_8,v_{11}\}$. In both cases $D$ is a 2DD-set of $G$, and so $\gamma_2^d(G) \le |D| \le |D_1| + 3 \le \floor*{\frac{2}{9}(n+k)}$, a contradiction.
		
		Hence, $V(F)=\{v_1,v_{14},v'_{15}\}$. We now let $G'$ be the mop of order $n'$ obtained from $G$ by deleting the vertices $V_1^{13}\setminus \{v_2\}$, and let $G'$ have $k'$ vertices of degree~$2$. We note that $n' = n-12$ and $k' = k - 1$, and $v_2v_{14}$ is an outer edge of $G'$. If $2 \le n' \le 3$, then $\{v_2,v_8,v_{11}\}$ is a 2DD-set of $G$, and hence $\gamma_2^d(G)\le 3 \le  \floor*{\frac{2}{9}(n+k)}$, a contradiction. If $5 \le n' \le 7$, then by \cref{obs3}, there exists a 2DD-set $D'$ of $G'$ such that $v_2\in D'$ and $|D'|=2$. Therefore, $D'\cup \{v_8,v_{11}\}$ is a 2DD-set of $G$, and so $\gamma_2^d(G)\le 4 \le \floor*{\frac{2}{9}(n+k)}$, a contradiction. Hence, $n' \ge 8$. Let $G_1$ be a graph of order $n_1$ obtained from $G'$ by contracting the edge $v_2v_{14}$ to form a new vertex $x$ in $G_1$, and let $G_1$ have $k_1$ vertices of degree~$2$. By Lemma~\ref{key}, $G_1$ is a mop. Since $n' \ge 8$, we note that $n_1 = n' - 1 \ge 7$. Further we note that $n_1 = n - 13$ and $k_1 \le k-1$. By the minimality of the mop $G$, we have $\gamma_2^d(G_1) \le \floor*{\frac{2}{9}(n_1+k_1)} \le  \floor*{\frac{2}{9}(n-13+k-1)} \le \floor*{\frac{2}{9}(n+k)}-3$. Let $D_1$ be a $\gamma_2^d$-set of $G_1$. If $x \in D_1$, then let $D = (D_1 \setminus \{x\}) \cup \{v_2,v_8,v_{11},v_{14}\}$. If $x \notin D_1$, then let $D = D_1 \cup \{v_2,v_8,v_{11}\}$. In both cases $D$ is a 2DD-set of $G$, and so $\gamma_2^d(G) \le |D| \le |D_1| + 3 \le \floor*{\frac{2}{9}(n+k)}$, a contradiction.
	\end{proof}
	
	\begin{figure}[htbp]
		\begin{center}
			\includegraphics[scale=0.18]{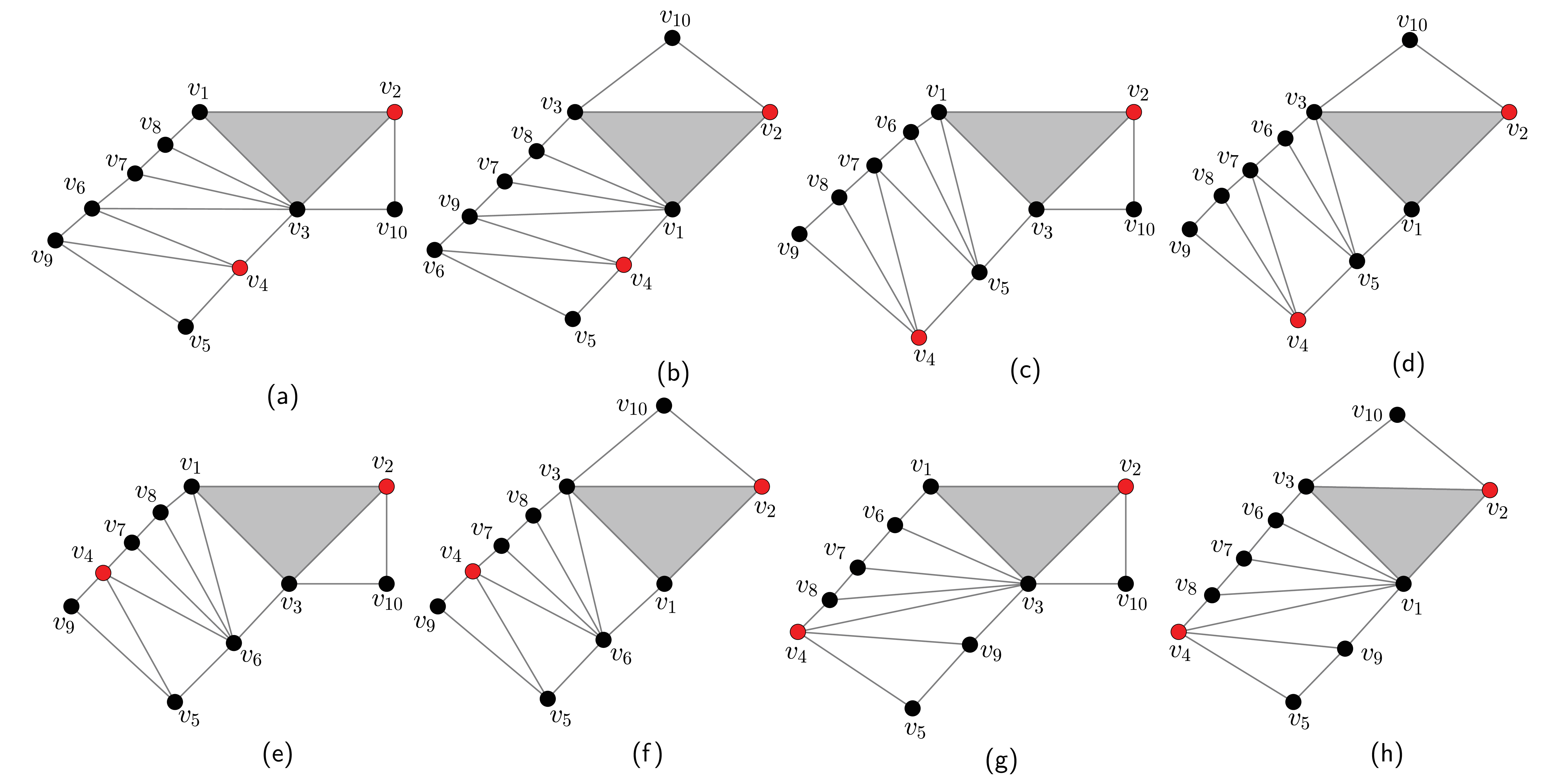}
			\caption{The regions of $G$ corresponding to tree $T_{22}$. The red vertices show a 2DD-set of $G[V(G)\setminus V(G')]$.}\label{tree_22}
		\end{center}
	\end{figure}

	\begin{claim}\label{tree22}
		The tree $T_{22}$ is not a maximal subtree of $T$.
	\end{claim}
	\begin{proof}[Proof of Claim~\ref{tree22}]
		Suppose, to the contrary, that $T_{22}$ is a maximal subtree of $T$, and so $T_{22} = T_v$ where $v$ denotes the root of the subtree $T_v$. We infer that the subgraph of $G$ associated with $T_{22}$ is obtained from either (i) the region $H_{5}$ by triangulating the region $v_1v_3v_4v_5v_9v_6v_7v_8$  according to \Cref{lem1}\ref{6distance} as illustrated in Figure~\ref{tree_22}(a)-(b) or (ii) the region $H_7$ by triangulating the region $v_1v_3v_5v_4v_9v_8v_7v_6$ according to \Cref{lem1}\ref{6distance} as illustrated in Figure~\ref{tree_22}(c)-(d) or (iii) the region $H_6$ by triangulating the region $v_1v_3v_6v_5v_9v_4v_7v_8$ according to \Cref{lem1}\ref{6distance} as illustrated in Figure~\ref{tree_22}(e)-(f) or (iv) the region $H_8$ by triangulating the region $v_1v_3v_9v_5v_4v_8v_7v_6$ according to \Cref{lem1}\ref{6distance} as illustrated in Figure~\ref{tree_22}(g)-(h), where we let $V(T_v) = \{v_1,v_2,v_{3}\}$ be the (shaded) triangle in $G$ associated with the vertex~$v$. In the following, we present arguments that work in each cases.
		
		Let $G'$ be the mop of order $n'$ obtained from $G$ by deleting the vertices in $V_3^{10}$, and let $G'$ have $k'$ vertices of degree~$2$. We note that $n' = n-8$ and $k' = k - 1$, and $v_1v_2$ is an outer edge of $G'$. Since $n \ge 13$, we have $n' \ge 5$. If $5 \le n' \le 7$, then by \cref{obs3}, there exists a 2DD-set $D'$ of $G'$ such that $v_2\in D'$ and $|D'|=2$. Therefore, $D'\cup \{v_4\}$ is a 2DD-set of $G$, and so $\gamma_2^d(G)\le 3\le  \floor*{\frac{2}{9}(n+k)}$, a contradiction. Hence, $n' \ge 8$. By the minimality of the mop $G$, we have $\gamma_2^d(G_1) \le \floor*{\frac{2}{9}(n'+k')} = \floor*{\frac{2}{9}(n-8+k-1)} \le \floor*{\frac{2}{9}(n+k)}-2$. Let $D'$ be a $\gamma_2^d$-set of $G_1$ and $D= D' \cup \{v_2,v_4\}$. The resulting set $D$ is a 2DD-set of $G$, and so $\gamma_2^d(G) \le |D| \le |D_1| + 2 \le \floor*{\frac{2}{9}(n+k)}$, a contradiction.
	\end{proof}
	
	\begin{figure}[htbp]
		\begin{center}
			\includegraphics[scale=0.19]{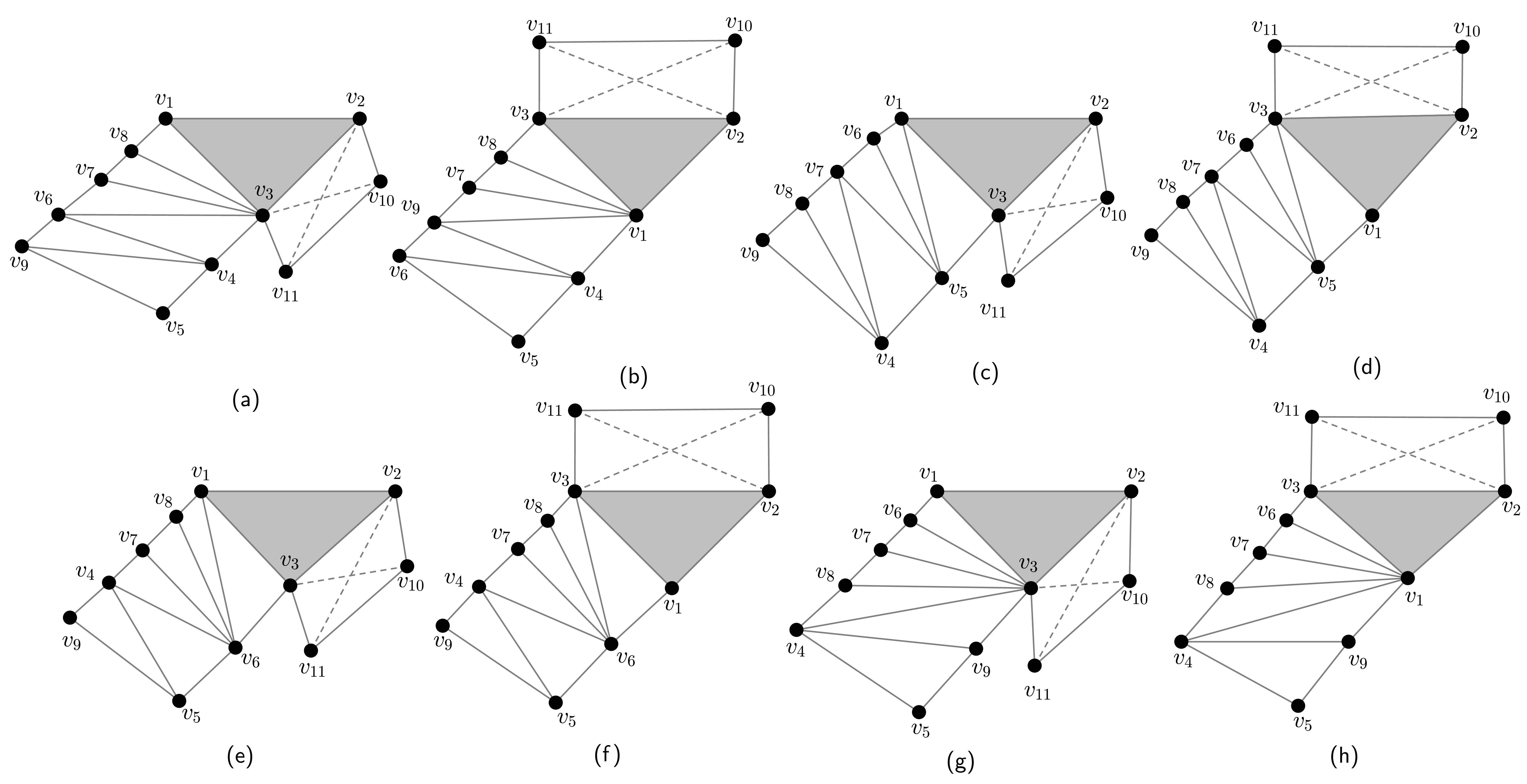}
			\caption{The regions of $G$ corresponding to tree $T_{23}$.}\label{tree_23}
		\end{center}
	\end{figure}

	\begin{claim}\label{tree23}
		The tree $T_{23}$ is not a maximal subtree of $T$.
	\end{claim}
	\begin{proof}[Proof of Claim~\ref{tree23}]
		Suppose, to the contrary, that $T_{23}$ is a maximal subtree of $T$, and so $T_{23} = T_v$ where $v$ denotes the root of the subtree $T_v$. We infer that the subgraph of $G$ associated with $T_{23}$ is obtained from either (i) the region $H_{5}$ by triangulating the region $v_1v_3v_4v_5v_9v_6v_7v_8$  according to \Cref{lem1}\ref{6distance} as illustrated in Figure~\ref{tree_23}(a)-(b) or (ii) the region $H_7$ by triangulating the region $v_1v_3v_5v_4v_9v_8v_7v_6$ according to \Cref{lem1}\ref{6distance} as illustrated in Figure~\ref{tree_23}(c)-(d) or (iii) the region $H_6$ by triangulating the region $v_1v_3v_6v_5v_9v_4v_7v_8$ according to \Cref{lem1}\ref{6distance} as illustrated in Figure~\ref{tree_23}(e)-(f) or (iv) the region $H_8$ by triangulating the region $v_1v_3v_9v_5v_4v_8v_7v_6$ according to \Cref{lem1}\ref{6distance} as illustrated in Figure~\ref{tree_23}(g)-(h), where we let $V(T_v) = \{v_1,v_2,v_{3}\}$ be the (shaded) triangle in $G$ associated with the vertex~$v$. The region $v_2v_{10}v_{11}v_3$ can be triangulated by adding either the edge $v_2v_{11}$ or $v_3v_{10}$, as indicated by the dotted lines in Figure~\ref{tree_23}(a)-(h). In the following, we present arguments that work in each cases.

		Let $G'$ be the mop of order $n'$ obtained from $G$ by deleting the vertices $V_4^{11}$, and let $G'$ have $k'$ vertices of degree~$2$. We note that $n' = n-8$ and $k' = k - 1$. Since $n \ge 13$, we have $n' \ge 5$. If $5 \le n' \le 7$, then by \cref{obs3}, there exists a 2DD-set $D'$ of $G'$ such that $v_2\in D'$ or $v_3\in D'$  and $|D'|=2$. Therefore, $D'\cup \{v_4\}$ is a 2DD-set of $G$, and so $\gamma_2^d(G)\le 3 \le  \floor*{\frac{2}{9}(n+k)}$, a contradiction. Hence, $n' \ge 8$. By the minimality of the mop $G$, we have $\gamma_2^d(G_1) \le \floor*{\frac{2}{9}(n'+k')} = \floor*{\frac{2}{9}(n-8+k-1)} \le \floor*{\frac{2}{9}(n+k)}-2$. Let $D'$ be a $\gamma_2^d$-set of $G_1$. If $v_2v_{11}\in E(G)$, then let $D= D' \cup \{v_2,v_4\}$. If $v_3v_{10}\in E(G)$, then let $D= D' \cup \{v_3,v_4\}$. In the both cases, $D$ is a 2DD-set of $G$, and so $\gamma_2^d(G) \le |D| \le |D_1| + 2 \le \floor*{\frac{2}{9}(n+k)}$, a contradiction.
	\end{proof}
	
	\begin{figure}[htbp]
		\begin{center}
			\includegraphics[scale=0.17]{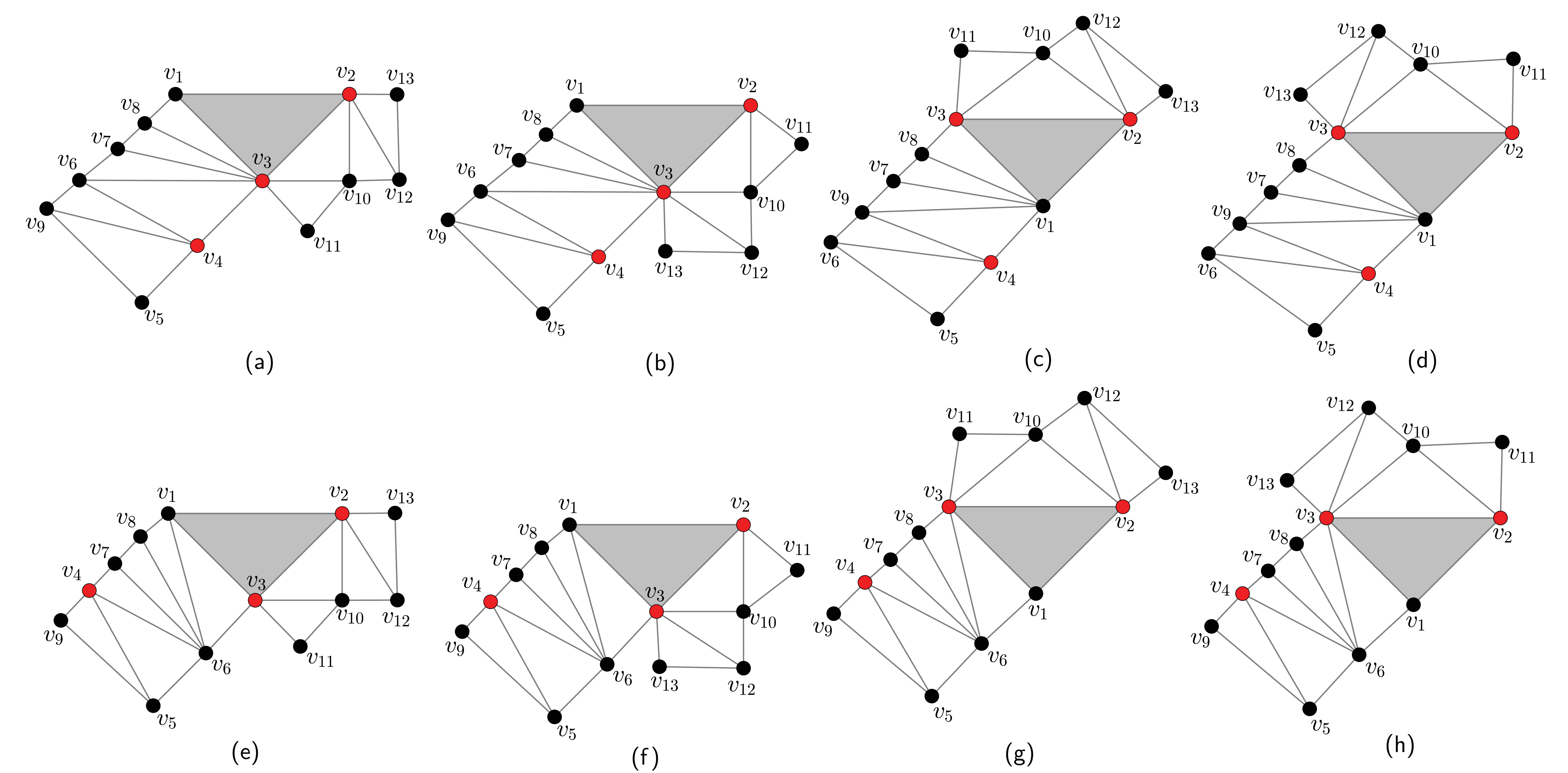}
			\includegraphics[scale=0.17]{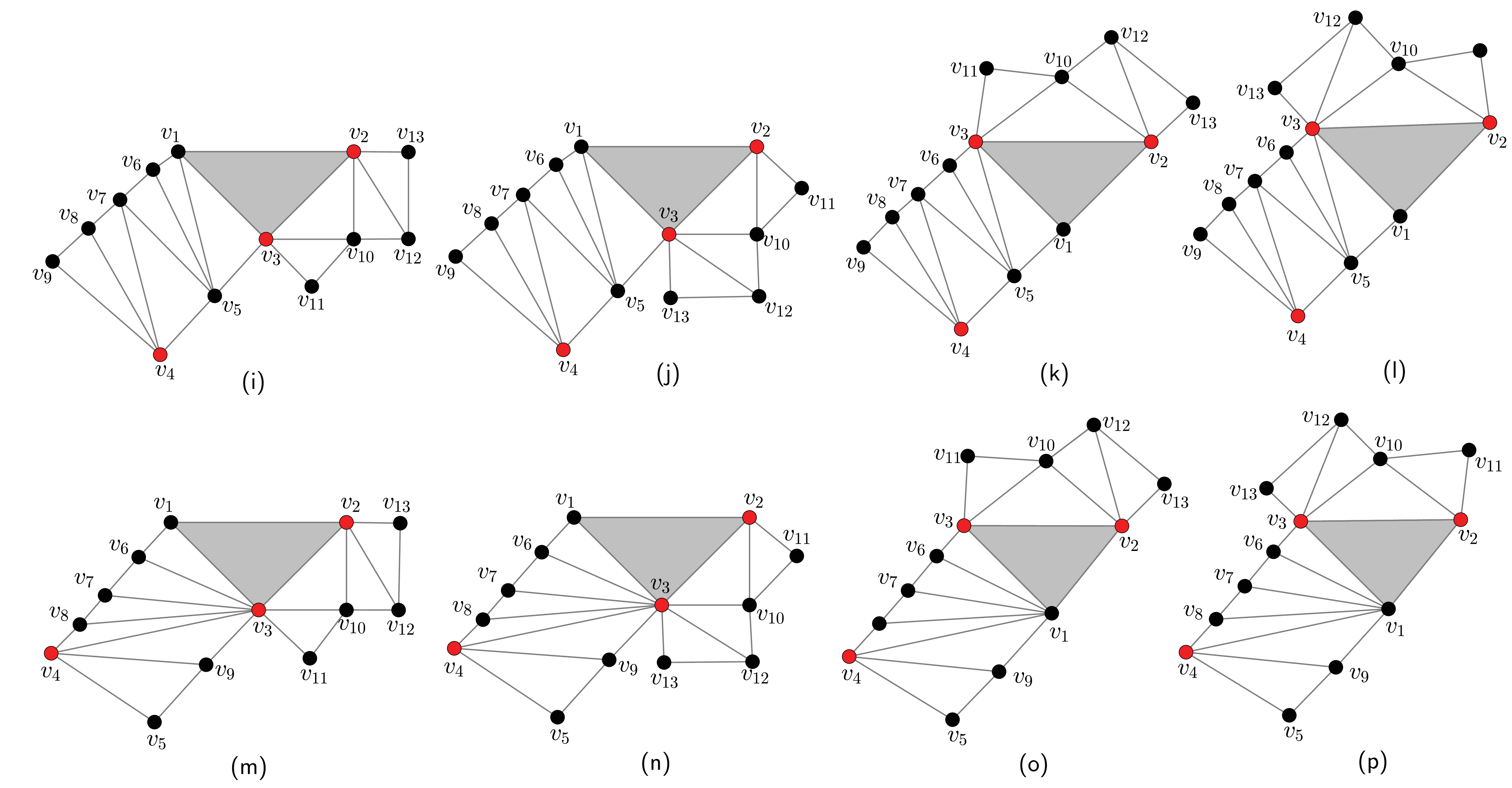}
			\caption{The regions of $G$ corresponding to tree $T_{24}$. The red vertices show a 2DD-set of $G[V(G)\setminus V(G')]$.}\label{tree_24}
		\end{center}
	\end{figure}

	\begin{claim}\label{tree24}
		The tree $T_{24}$ is not a maximal subtree of $T$.
	\end{claim}
	\begin{proof}[Proof of Claim~\ref{tree24}]
		Suppose, to the contrary, that $T_{24}$ is a maximal subtree of $T$, and so $T_{24} = T_v$ where $v$ denotes the root of the subtree $T_v$.  We infer that the subgraph of $G$ associated with $T_{24}$ is obtained from either (i) the region $H_{5}$ by triangulating the region $v_1v_3v_4v_5v_9v_6v_7v_8$  according to \Cref{lem1}\ref{6distance} as illustrated in Figure~\ref{tree_24}(a)-(d) or (ii) the region $H_6$ by triangulating the region $v_1v_3v_6v_5v_9v_4v_7v_8$  according to \Cref{lem1}\ref{6distance} as illustrated in Figure~\ref{tree_24}(e)-(h) or (iii) the region $H_7$ by triangulating the region  $v_1v_3v_5v_4v_9v_8v_7v_6$ according to \Cref{lem1}\ref{6distance} as illustrated in Figure~\ref{tree_24}(i)-(l) or (iv) the region $H_8$ by triangulating the region $v_1v_3v_9v_5v_4v_8v_7v_6$ according to \Cref{lem1}\ref{6distance} as illustrated in Figure~\ref{tree_24}(m)-(p), where we let $V(T_v) = \{v_1,v_2,v_{3}\}$ be the (shaded) triangle in $G$ associated with the vertex~$v$. In the following, we present arguments that work in each cases.
		
		Let $G'$ be the mop of order $n'$ obtained from $G$ by deleting the vertices $V_3^{13}$, and let $G'$ have $k'$ vertices of degree~$2$. We note that $n' = n-11$ and $k' = k - 2$, and $v_2v_{14}$ is an outer edge of $G'$. Since $n \ge 13$, we have $n' \ge 2$. If $2 \le n' \le 4$, then $\{v_2,v_3,v_4\}$ is a 2DD-set of $G$, and hence $\gamma_2^d(G)\le 3 \le  \floor*{\frac{2}{9}(n+k)}$, a contradiction. If $5 \le n' \le 7$, then by \cref{obs3}, there exists a 2DD-set $D'$ of $G'$ such that $v_2\in D'$ and $|D'|=2$. Therefore, $D'\cup \{v_3,v_4\}$ is a 2DD-set of $G$, and so $\gamma_2^d(G)\le 4 \le \floor*{\frac{2}{9}(n+k)}$, a contradiction. Hence, $n' \ge 8$. Let $G_1$ be a graph of order $n_1$ obtained from $G'$ by contracting the edge $v_2v_{14}$ to form a new vertex $x$ in $G_1$, and let $G_1$ have $k_1$ vertices of degree~$2$. By Lemma~\ref{key}, $G_1$ is a mop. Since $n' \ge 8$, we note that $n_1 = n' - 1 \ge 7$. Further we note that $n_1 = n - 12$ and $k_1 \le k-2$. By the minimality of the mop $G$, we have $\gamma_2^d(G_1) \le \floor*{\frac{2}{9}(n_1+k_1)} \le \floor*{\frac{2}{9}(n-12+k-2)} \le \floor*{\frac{2}{9}(n+k)}-3$. Let $D_1$ be a $\gamma_2^d$-set of $G_1$. If $x \in D_1$, then let $D = (D_1 \setminus \{x\}) \cup \{v_2,v_3,v_4,v_{14}\}$. If $x \notin D_1$, then let $D = D_1 \cup \{v_2,v_3,v_4\}$. In both cases $D$ is a 2DD-set of $G$, and so $\gamma_2^d(G) \le |D| \le |D_1| + 3 \le \floor*{\frac{2}{9}(n+k)}$, a contradiction.
	\end{proof}
	
	\begin{figure}[htbp]
		\begin{center}
			\includegraphics[scale=0.16]{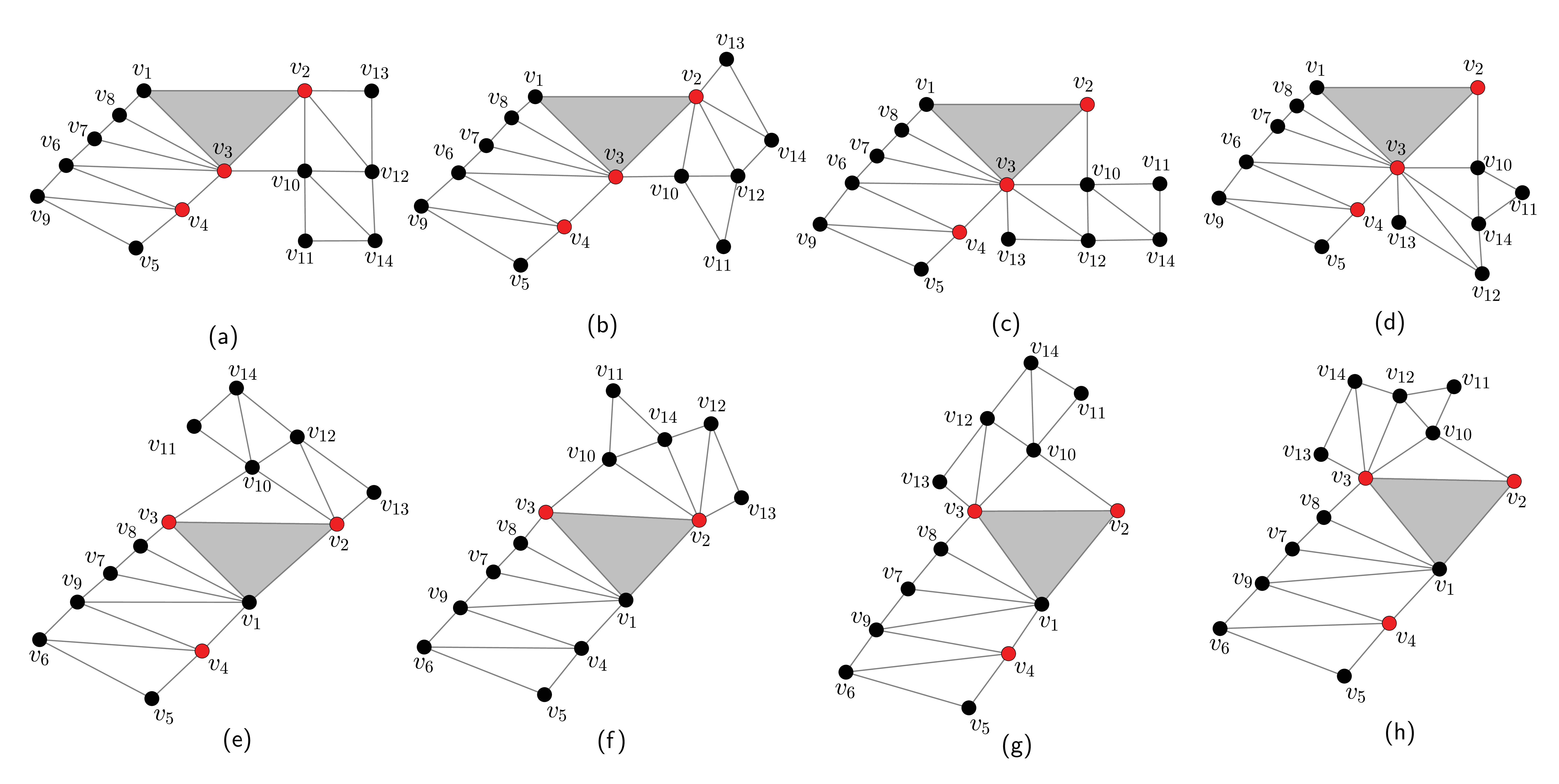}
			\includegraphics[scale=0.16]{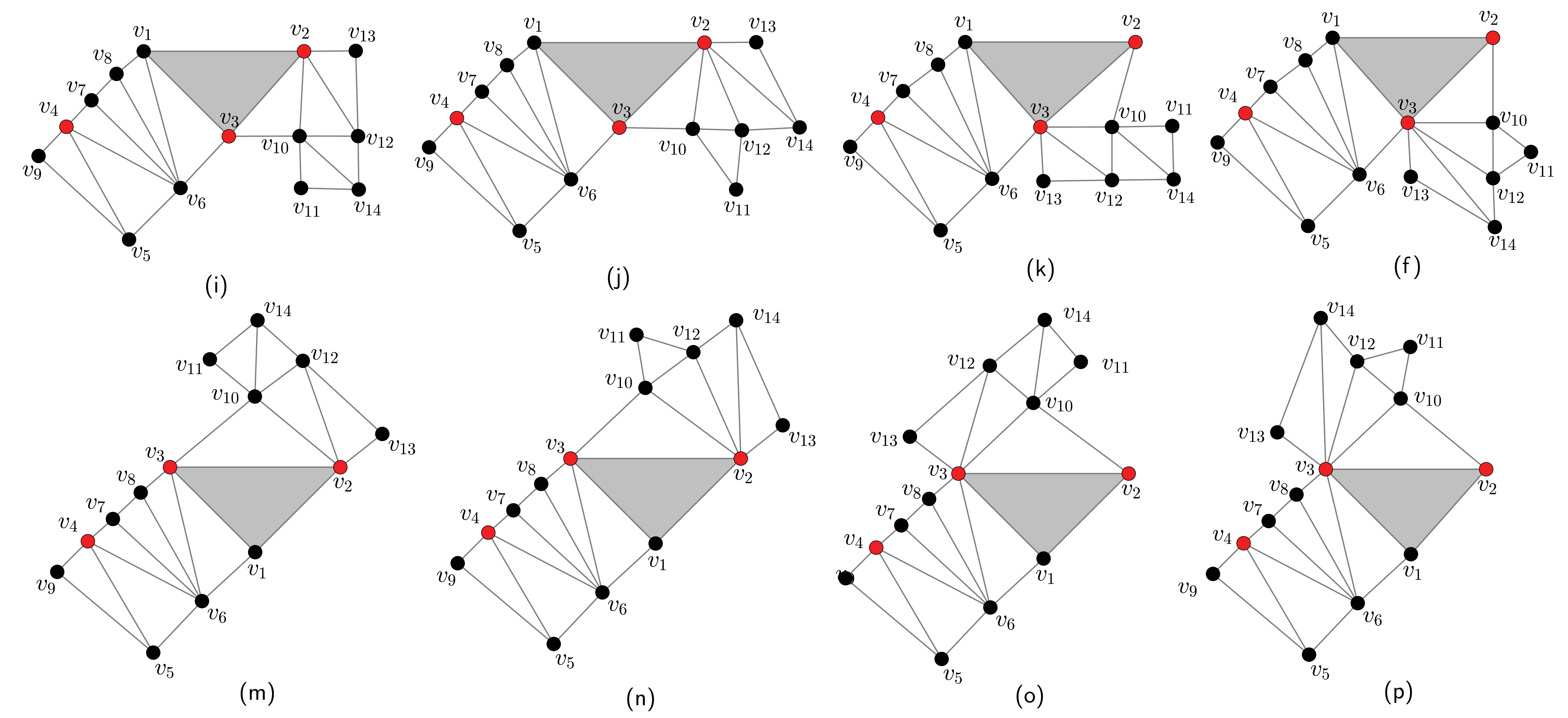}
			\caption{The regions of $G$ corresponding to tree $T_{25}$. The red vertices show a 2DD-set of $G[V(G)\setminus V(G')]$.}\label{tree_25}
		\end{center}
	\end{figure}

	\begin{figure}[htbp]
		\begin{center}
			\includegraphics[scale=0.16]{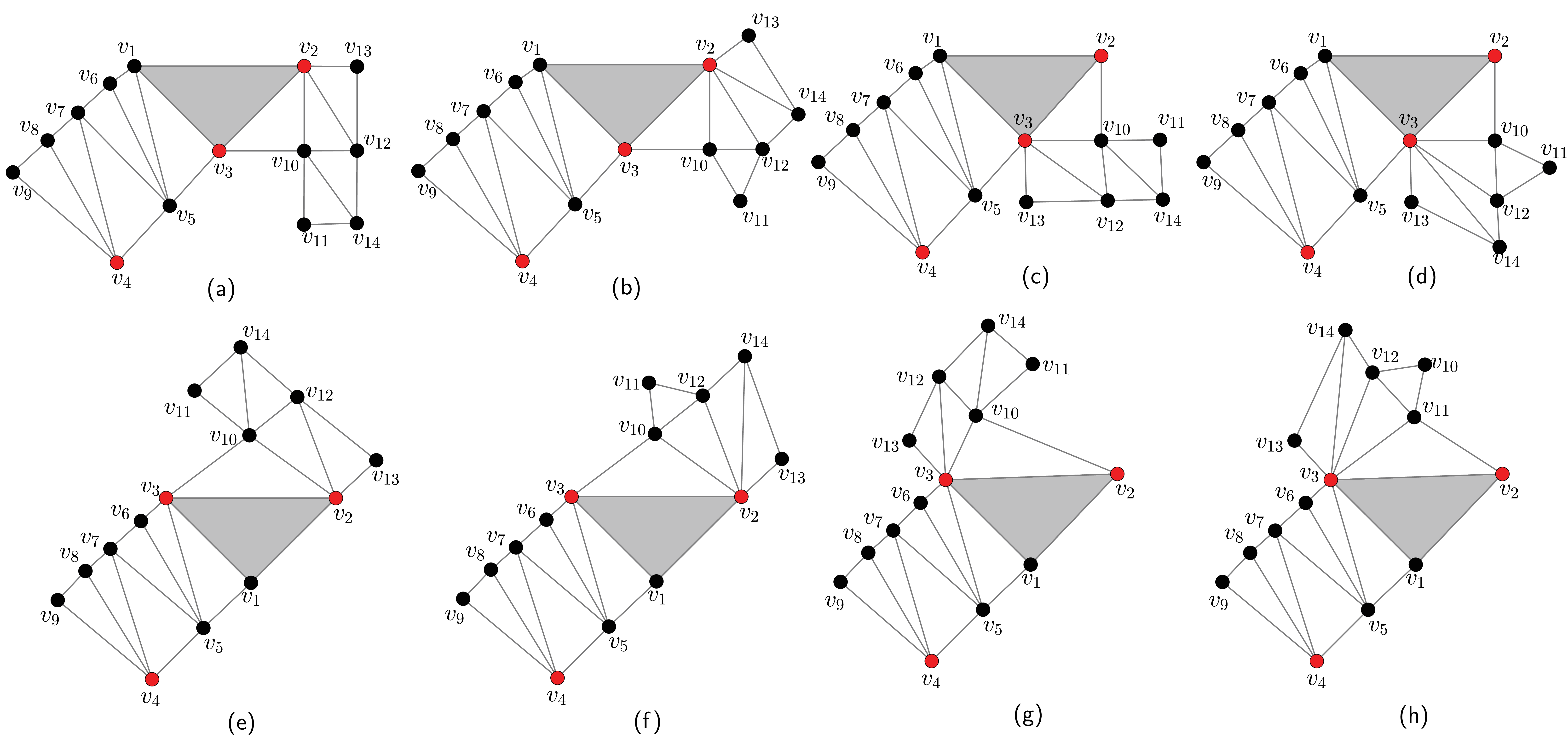}
			\includegraphics[scale=0.16]{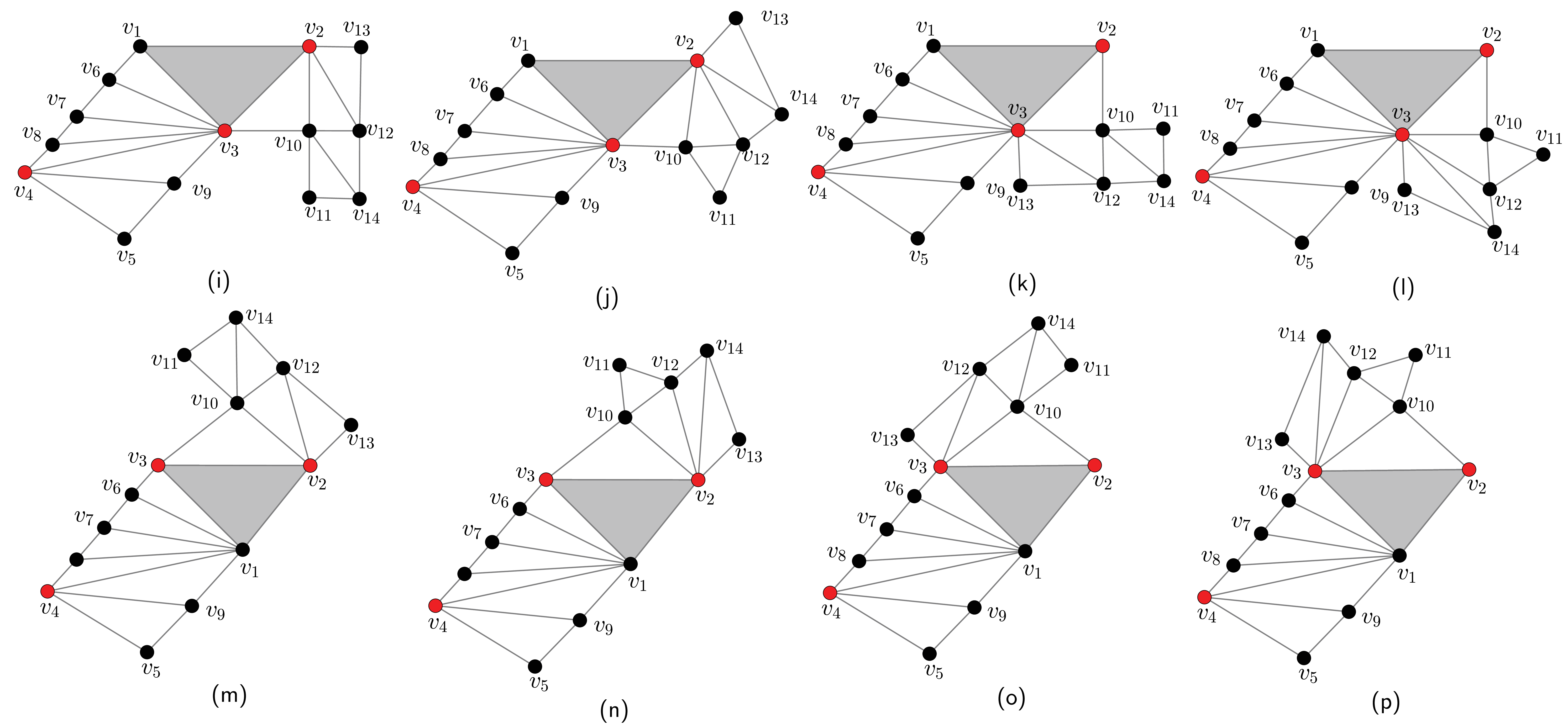}
			\caption{The regions of $G$ corresponding to tree $T_{25}$. The red vertices show a 2DD-set of $G[V(G)\setminus V(G')]$.}\label{tree_25_1}
		\end{center}
	\end{figure}

	\begin{claim}\label{tree25}
		The tree $T_{25}$ is not a maximal subtree of $T$.
	\end{claim}
	\begin{proof}[Proof of Claim~\ref{tree25}]
		Suppose, to the contrary, that $T_{25}$ is a maximal subtree of $T$, and so $T_{25} = T_v$ where $v$ denotes the root of the subtree $T_v$.  We infer that the subgraph of $G$ associated with $T_{25}$ is obtained from either (i) the region $H_{5}$ by triangulating the region $v_1v_3v_4v_5v_9v_6v_7v_8$  according to \Cref{lem1}\ref{6distance} as illustrated in Figure~\ref{tree_25}(a)-(h) or (ii) the region $H_6$ by triangulating the region $v_1v_3v_6v_5v_9v_4v_7v_8$  according to \Cref{lem1}\ref{6distance} as illustrated in Figure~\ref{tree_25}(i)-(p) or (iii) the region $H_7$ by triangulating the region  $v_1v_3v_5v_4v_9v_8v_7v_6$ according to \Cref{lem1}\ref{6distance} as illustrated in Figure~\ref{tree_25_1}(a)-(h) or (iv) the region $H_8$ by triangulating the region $v_1v_3v_9v_5v_4v_8v_7v_6$ according to \Cref{lem1}\ref{6distance} as illustrated in Figure~\ref{tree_25_1}(i)-(p), where we let $V(T_v) = \{v_1,v_2,v_{3}\}$ be the (shaded) triangle in $G$ associated with the vertex~$v$. In the following, we present arguments that work in each cases.
		
		Let $G'$ be the mop of order $n'$ obtained from $G$ by deleting the vertices in $V_3^{14}$, and let $G'$ have $k'$ vertices of degree~$2$. We note that $n' = n-12$ and $k' = k - 2$. If $2 \le n' \le 4$, then $\{v_2,v_3,v_4\}$ is a 2DD-set of $G$, and hence $\gamma_2^d(G)\le 3 \le  \floor*{\frac{2}{9}(n+k)}$, a contradiction. If $5 \le n' \le 7$, then by \cref{obs3}, there  exists a 2DD-set $D'$ of $G'$ such that $v_2\in D'$ and $|D'|=2$. Therefore, $D'\cup \{v_3,v_4\}$ is a 2DD-set of $G$, and so $\gamma_2^d(G)\le 4 \le \floor*{\frac{2}{9}(n+k)}$, a contradiction. Hence, $n' \ge 8$. By the minimality of the mop $G$, we have $\gamma_2^d(G_1) \le \floor*{\frac{2}{9}(n'+k')} \le  \floor*{\frac{2}{9}(n-12+k-2)} \le \floor*{\frac{2}{9}(n+k)}-3$. Let $D_1$ be a $\gamma_2^d$-set of $G_1$ and let $D = D_1 \cup \{v_2,v_3,v_4\}$. The set $D$ is a 2DD-set of $G$, and so $\gamma_2^d(G) \le |D|\le |D_1| + 3 \le \floor*{\frac{2}{9}(n+k)}$, a contradiction.
	\end{proof}
	
	\begin{figure}[htbp]
		\begin{center}
			\includegraphics[scale=0.17]{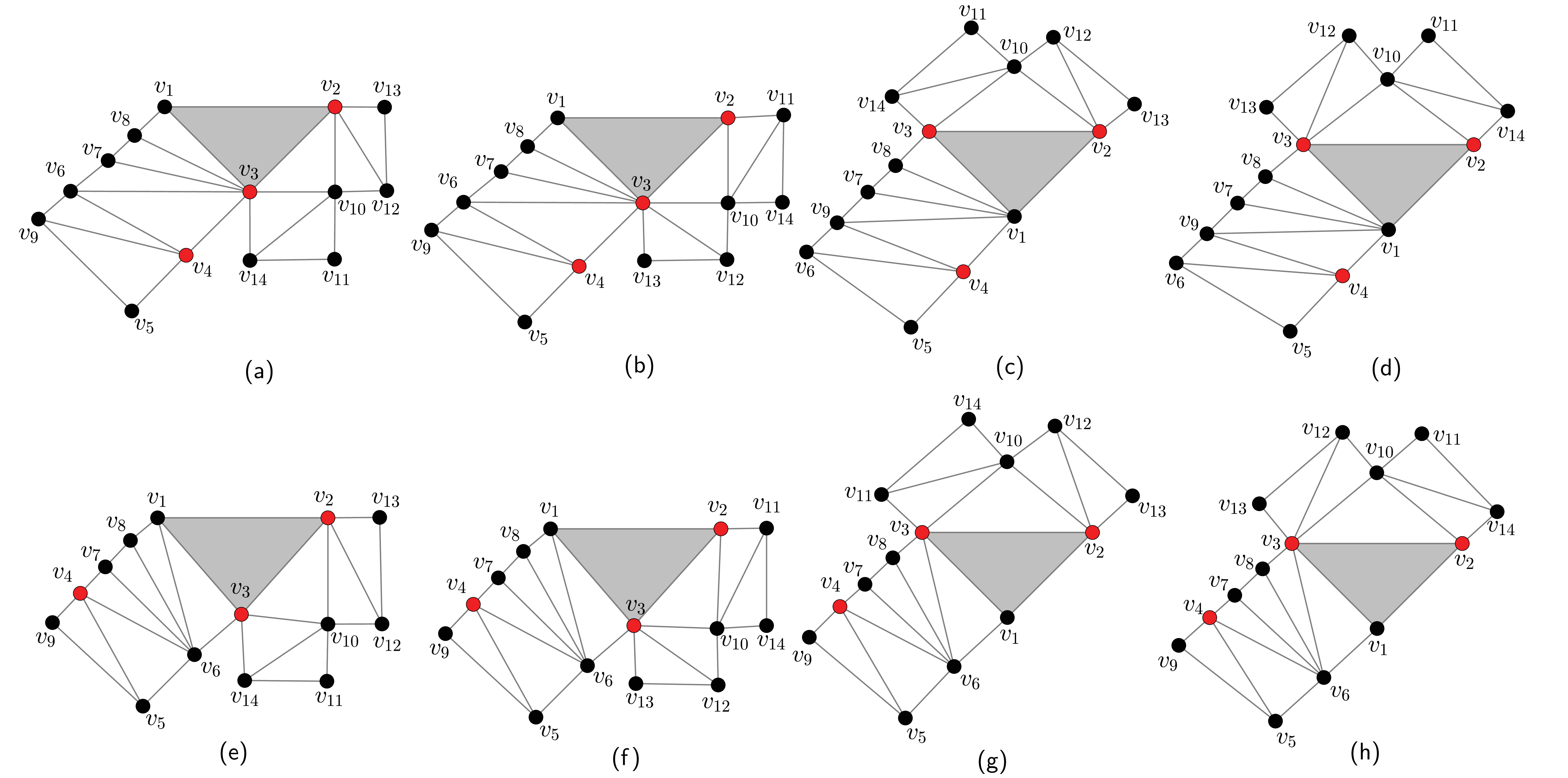}
			\includegraphics[scale=0.17]{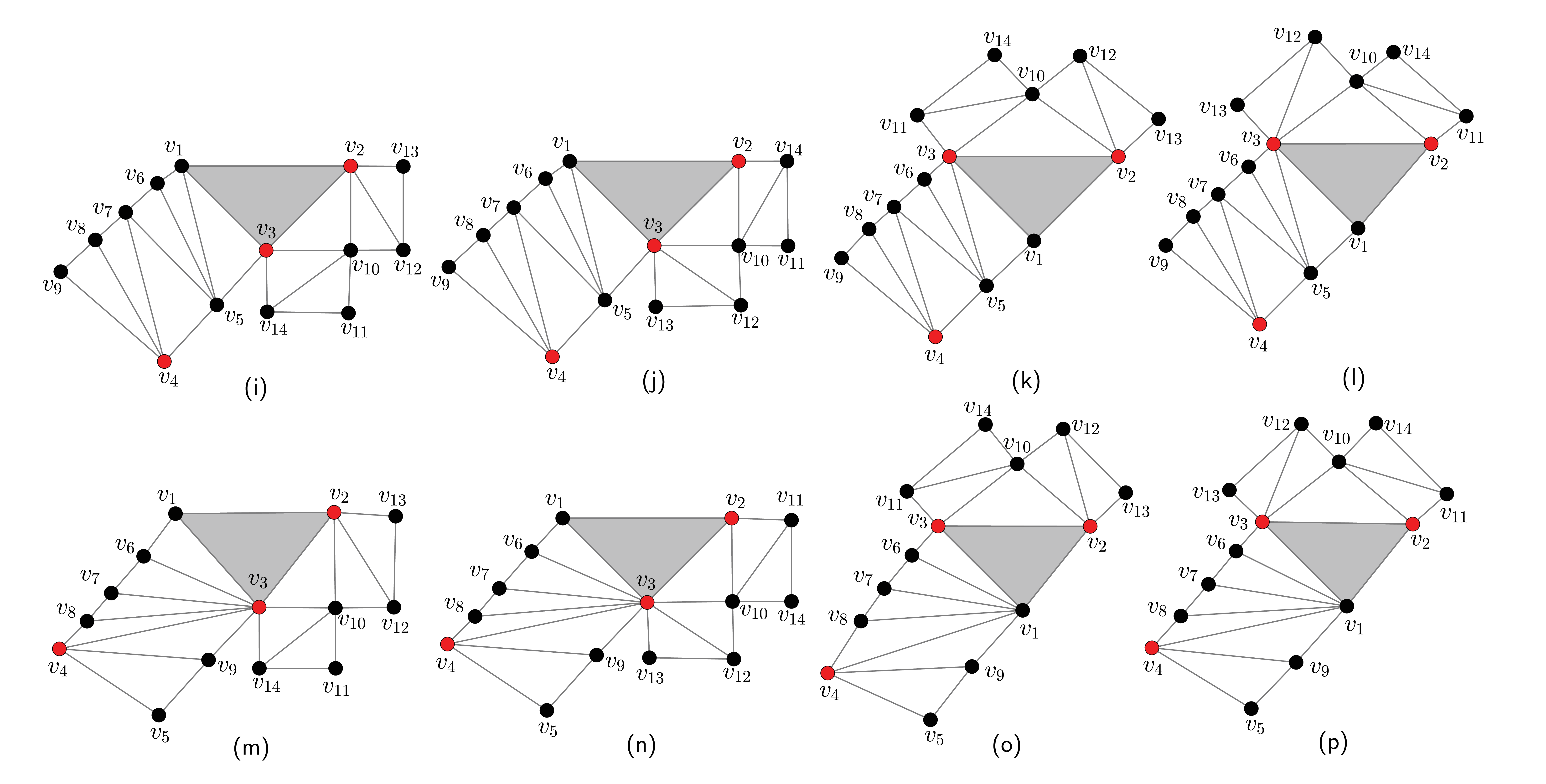}
			\caption{The regions of $G$ corresponding to tree $T_{26}$. The red vertices show a 2DD-set of $G[V(G)\setminus V(G')]$.}\label{tree_26}
		\end{center}
	\end{figure}

	\begin{claim}\label{tree26}
		The tree $T_{26}$ is not a maximal subtree of $T$.
	\end{claim}
	\begin{proof}[Proof of Claim~\ref{tree26}]
		Suppose, to the contrary, that $T_{26}$ is a maximal subtree of $T$, and so $T_{26} = T_v$ where $v$ denotes the root of the subtree $T_v$. We infer that the subgraph of $G$ associated with $T_{26}$ is obtained from either (i) the region $H_{5}$ by triangulating the region $v_1v_3v_4v_5v_9v_6v_7v_8$  according to \Cref{lem1}\ref{6distance} as illustrated in Figure~\ref{tree_26}(a)-(d) or (ii) the region $H_6$ by triangulating the region  $v_1v_3v_6v_5v_9v_4v_7v_8$ according to \Cref{lem1}\ref{6distance} as illustrated in Figure~\ref{tree_26}(e)-(h) or (iii) the region $H_7$ by triangulating the region $v_1v_3v_5v_4v_9v_8v_7v_6$  according to \Cref{lem1}\ref{6distance} as illustrated in Figure~\ref{tree_26}(i)-(l) or (iv) the region $H_8$ by triangulating the region $v_1v_3v_9v_5v_4v_8v_7v_6$ according to \Cref{lem1}\ref{6distance} as illustrated in Figure~\ref{tree_26}(m)-(p), where we let $V(T_v) = \{v_1,v_2,v_{3}\}$ be the (shaded) triangle in $G$ associated with the vertex~$v$. In the following, we present arguments that work in each cases.
		
		Let $G'$ be the mop of order $n'$ obtained from $G$ by deleting the vertices in $V_3^{14}$, and let $G'$ have $k'$ vertices of degree~$2$. We note that $n' = n-12$ and $k' = k - 2$. If $2 \le n' \le 4$, then $\{v_2,v_3,v_4\}$ is a 2DD-set of $G$, and hence $\gamma_2^d(G)\le 3 \le  \floor*{\frac{2}{9}(n+k)}$, a contradiction. If $5 \le n' \le 7$, then by \cref{obs3}, there exists a 2DD-set $D'$ of $G'$ such that $v_2\in D'$ and $|D'|=2$. Therefore, $D'\cup \{v_3,v_4\}$ is a 2DD-set of $G$, and so $\gamma_2^d(G)\le 4 \le \floor*{\frac{2}{9}(n+k)}$, a contradiction. Hence, $n' \ge 8$. By the minimality of the mop $G$, we have $\gamma_2^d(G_1) \le \floor*{\frac{2}{9}(n'+k')} \le \floor*{\frac{2}{9}(n-12+k-2)} \le \floor*{\frac{2}{9}(n+k)}-3$. Let $D_1$ be a $\gamma_2^d$-set of $G_1$ and let $D = D_1 \cup \{v_2,v_3,v_4\}$. The set $D$ is a 2DD-set of $G$, and so $\gamma_2^d(G) \le |D| \le |D_1| + 3 \le \floor*{\frac{2}{9}(n+k)}$, a contradiction.
	\end{proof}

	\begin{figure}[htbp]
		\begin{center}
			\includegraphics[scale=0.16]{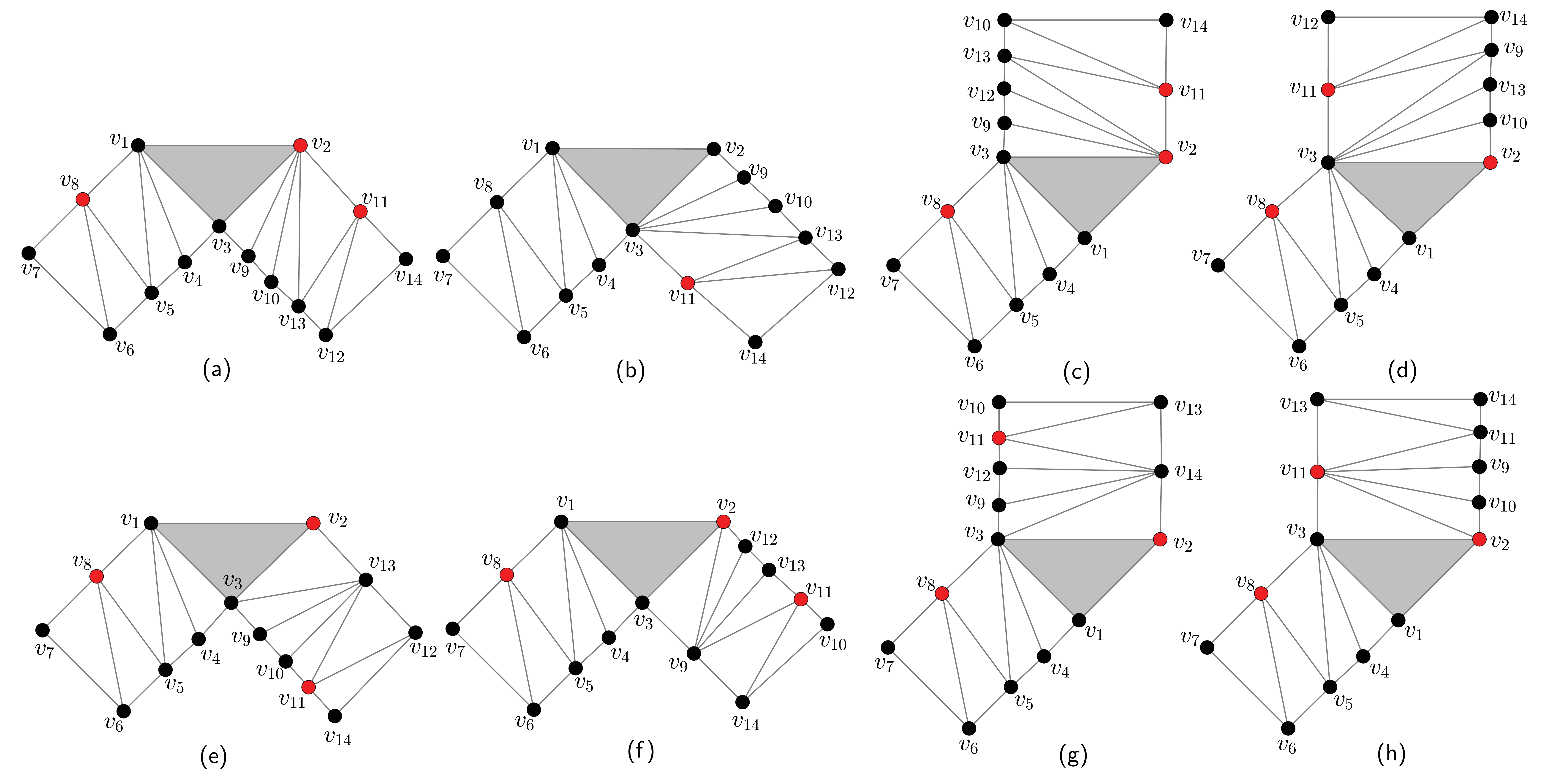}
			\includegraphics[scale=0.16]{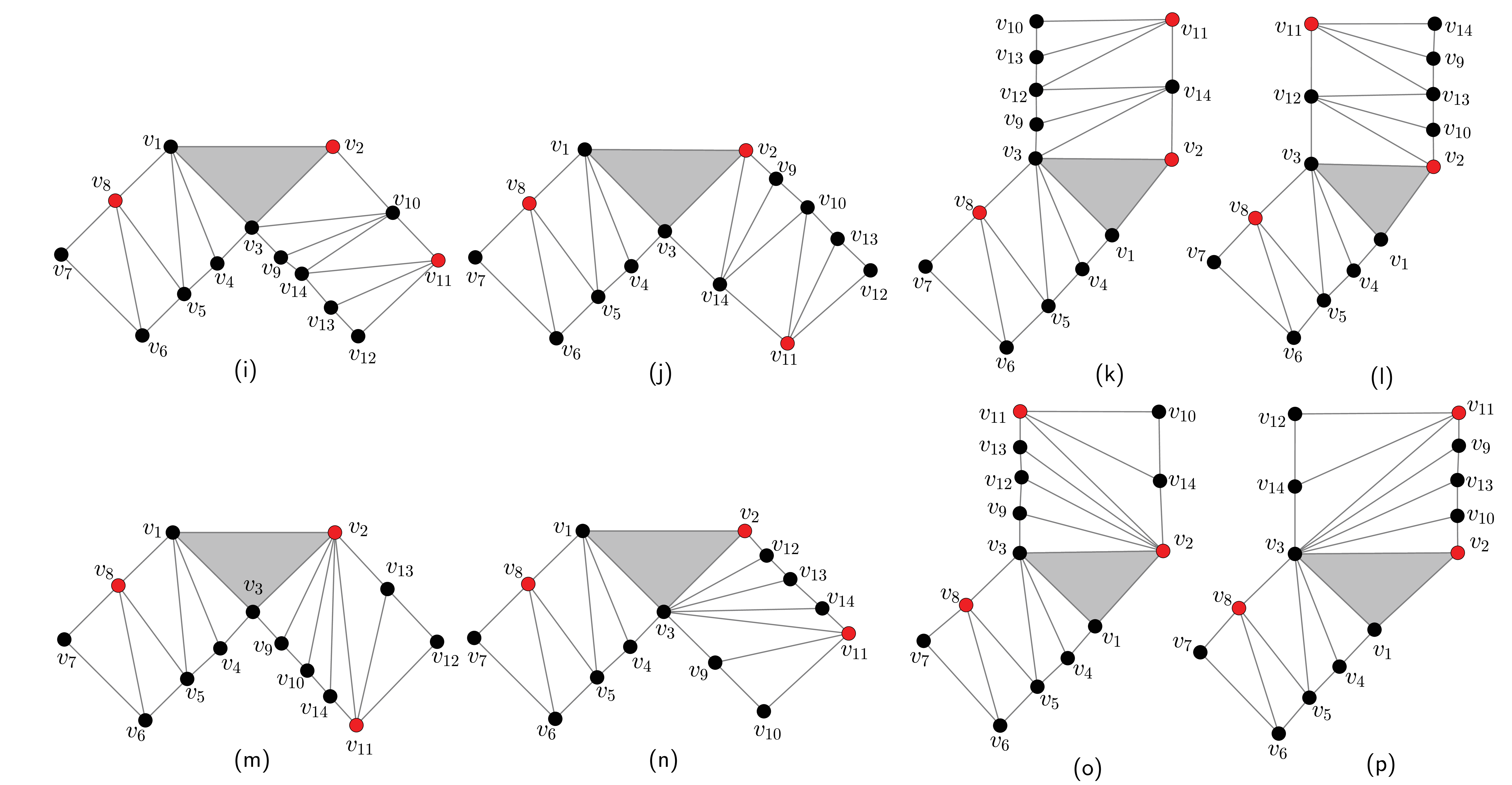}
			\caption{The regions of $G$ corresponding to tree $T_{27}$. The red vertices show a 2DD-set of $G[V(G)\setminus V(G')]$.}\label{tree_27}
		\end{center}
	\end{figure}
	
	\begin{figure}[htbp]
		\begin{center}
			\includegraphics[scale=0.16]{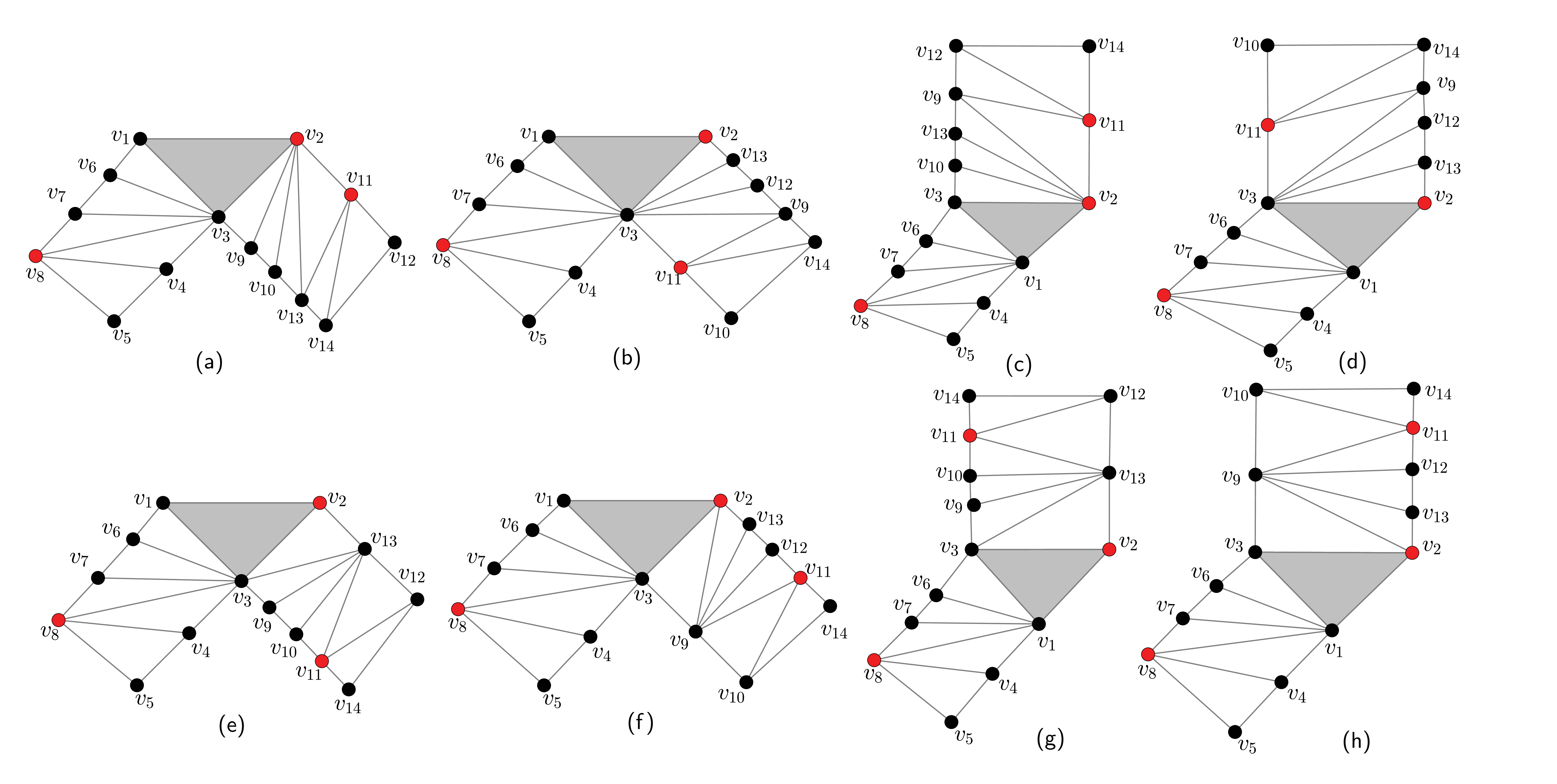}
			\includegraphics[scale=0.16]{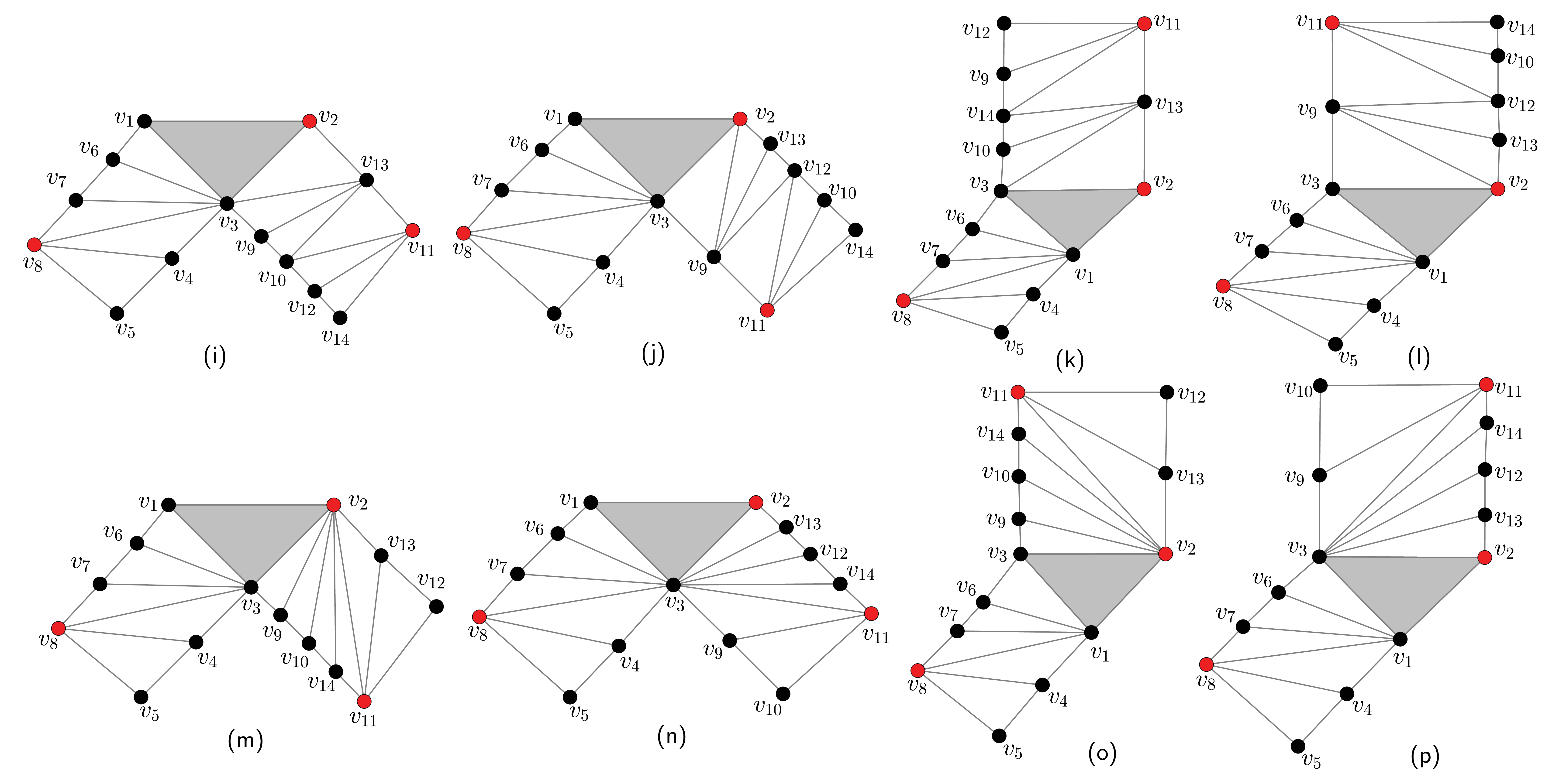}
			{\vspace{-.5cm}}
			\caption{The regions of $G$ corresponding to tree $T_{27}$. The red vertices show a 2DD-set of $G[V(G)\setminus V(G')]$.}\label{tree_27_1}
		\end{center}
	\end{figure}
	
	\begin{figure}[htbp]
		\begin{center}
			\includegraphics[scale=0.16]{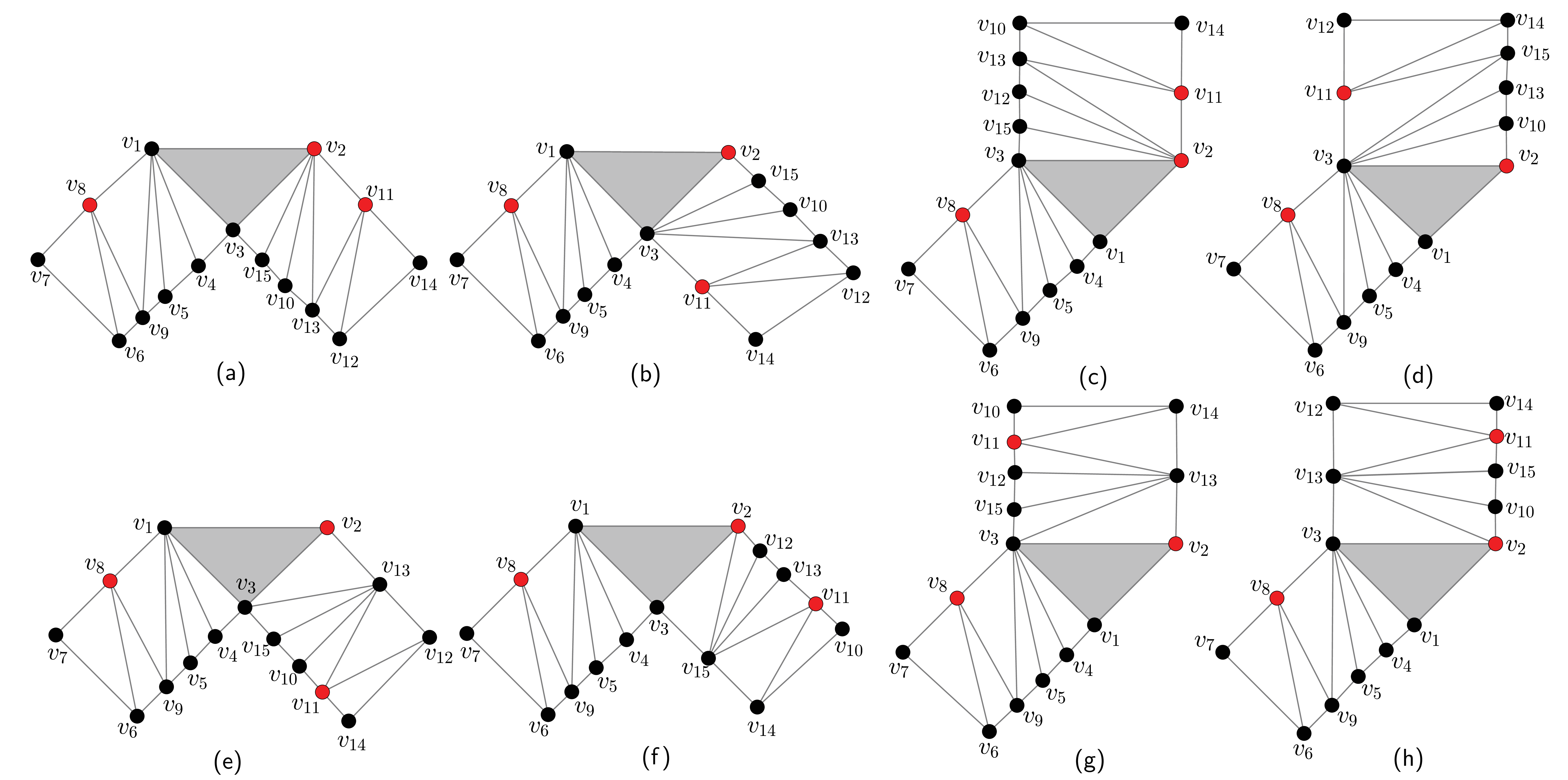}
			\includegraphics[scale=0.16]{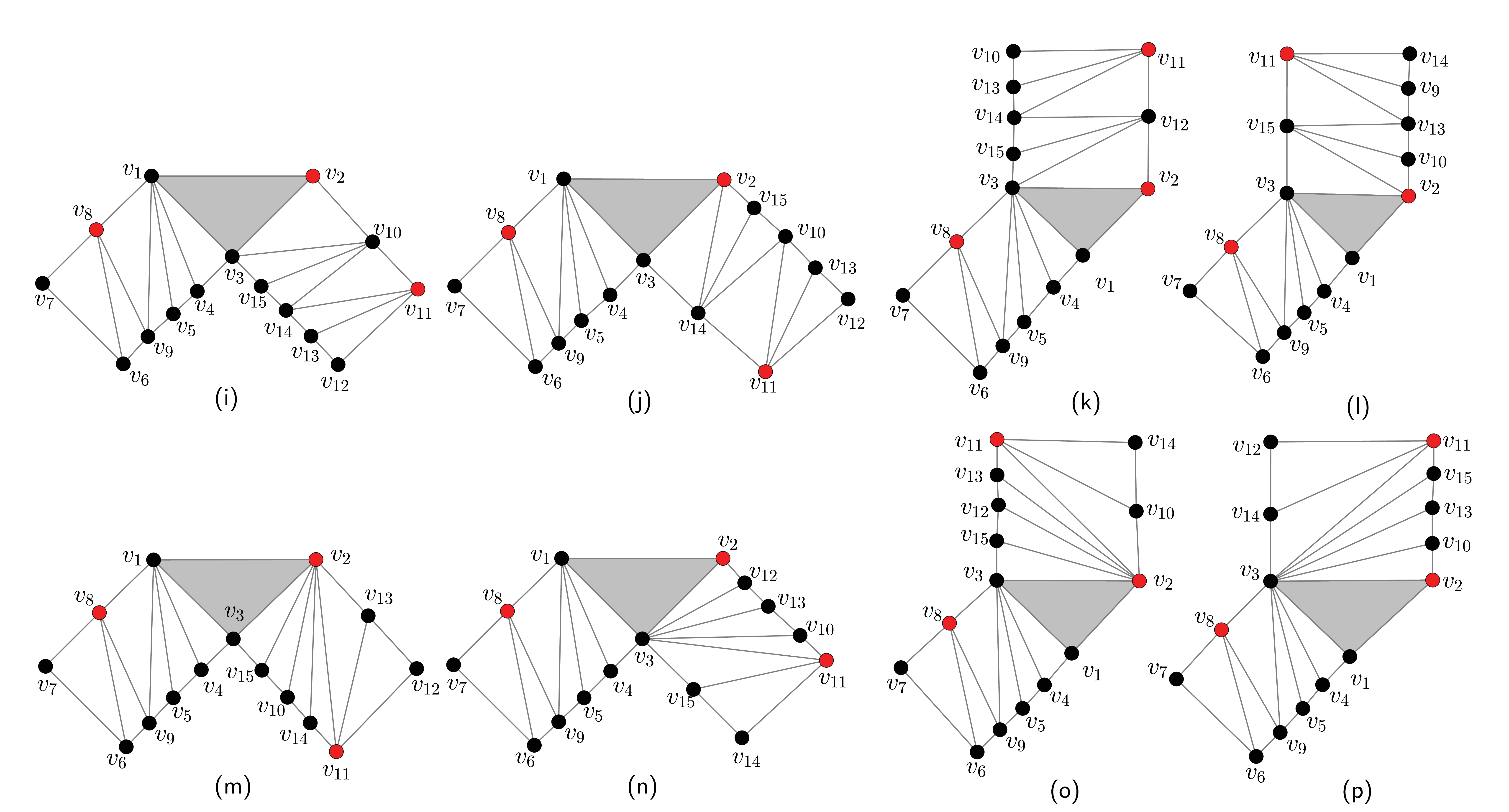}
			{\vspace{-.7cm}}
			\includegraphics[scale=0.16]{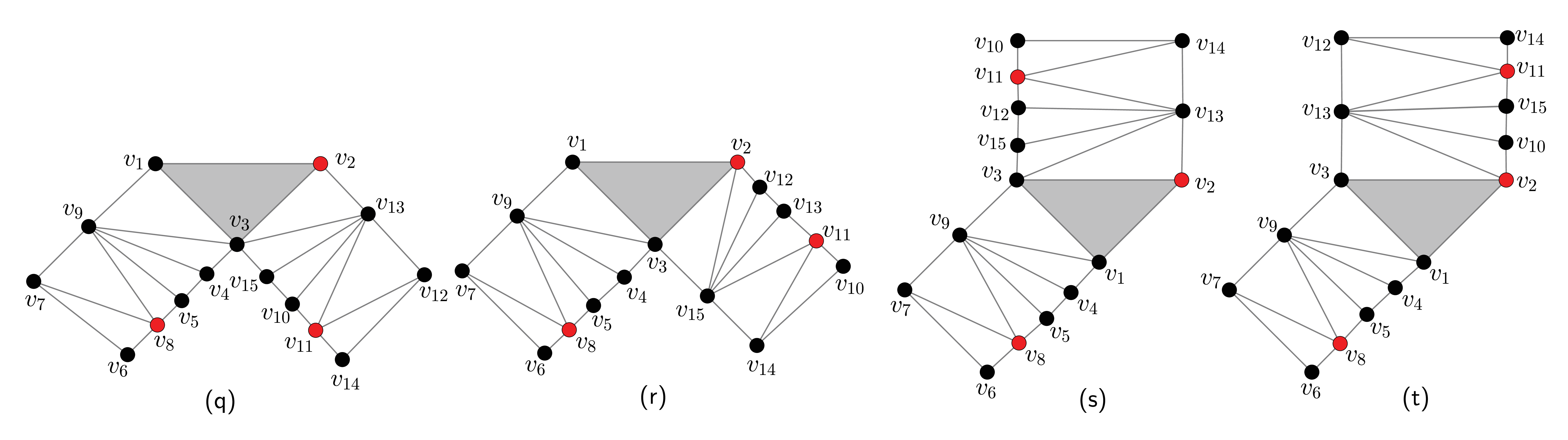}
			\caption{The regions of $G$ corresponding to tree $T_{28}$. The red vertices show a 2DD-set of $G[V(G)\setminus V(G')]$.}\label{tree_28}
		\end{center}
	\end{figure}
	
	\begin{figure}[htbp]
		\begin{center}
			\includegraphics[scale=0.16]{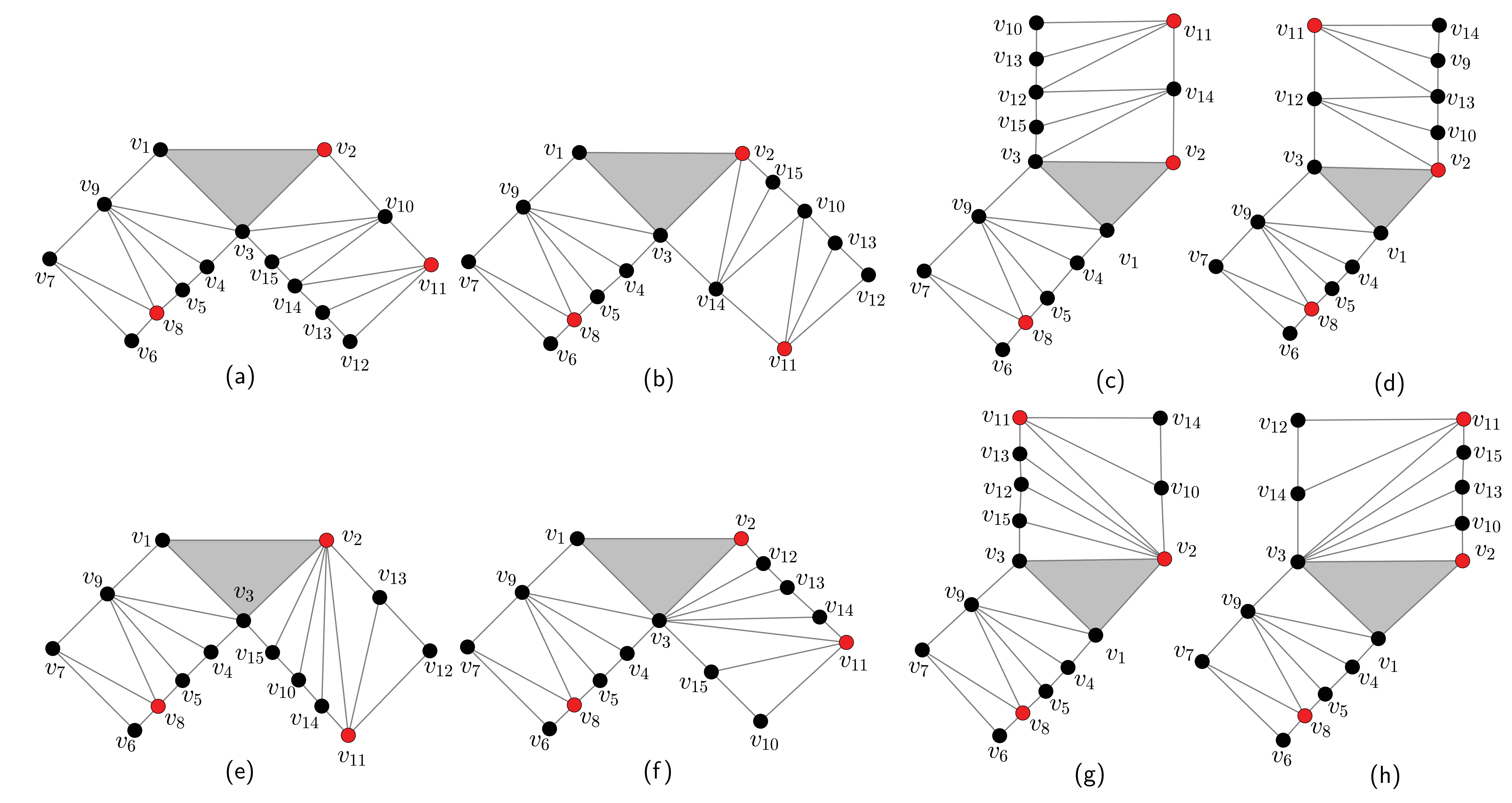}
			\includegraphics[scale=0.16]{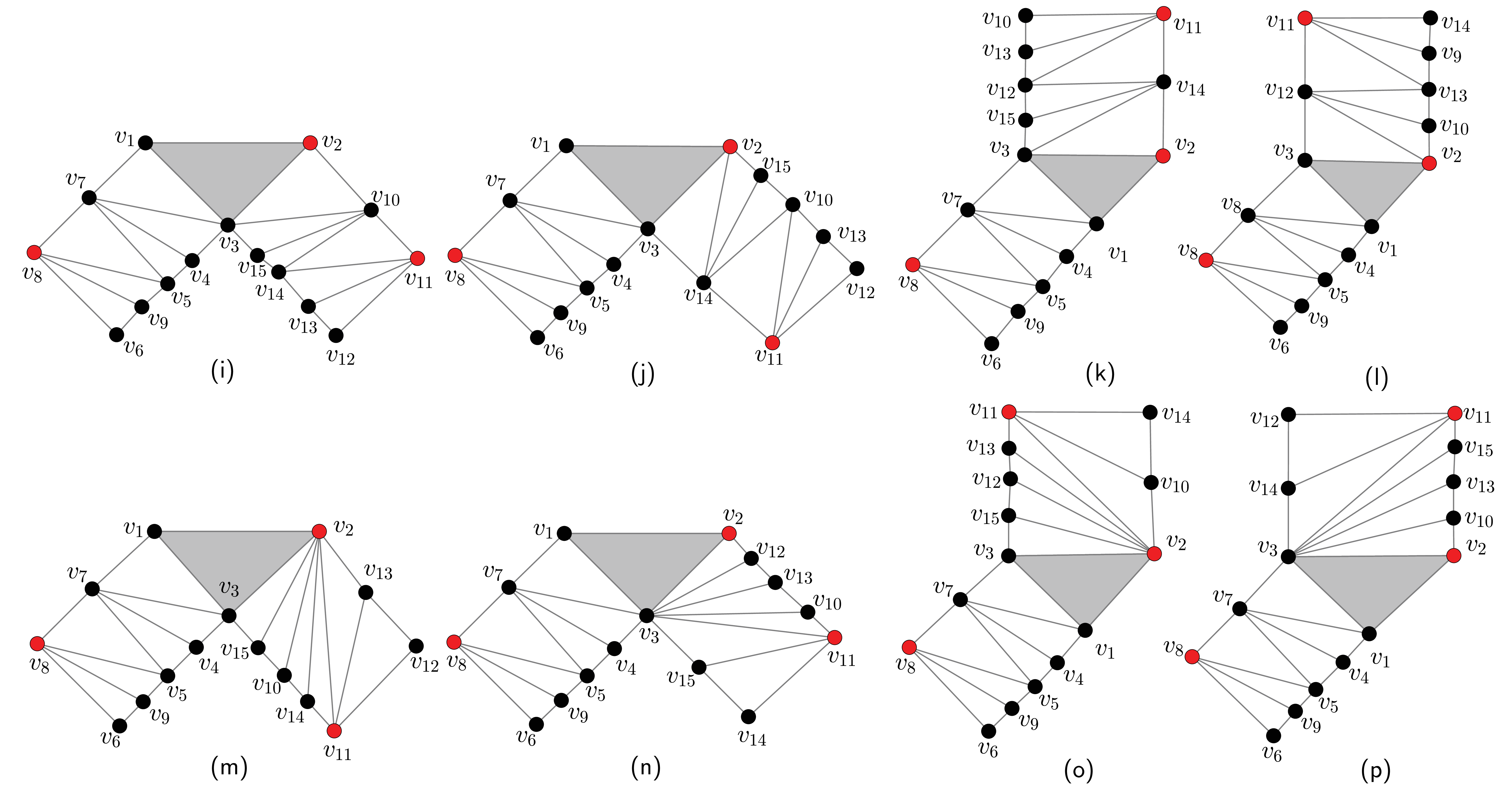}
			\includegraphics[scale=0.16]{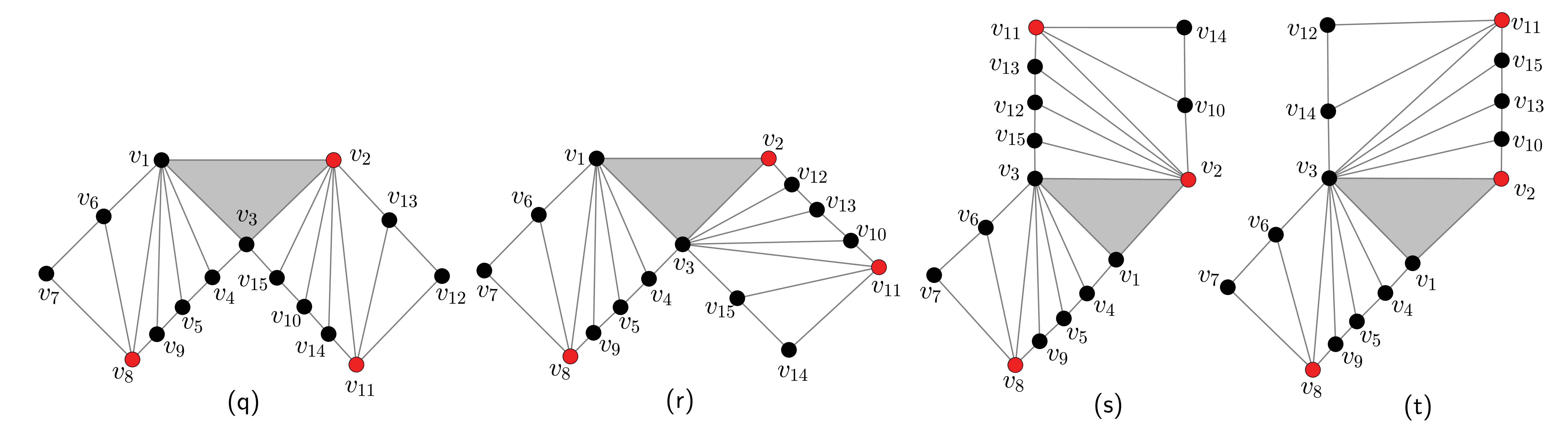}
			{\vspace{-.5cm}}
			\caption{The regions of $G$ corresponding to tree $T_{28}$. The red vertices show a 2DD-set of $G[V(G)\setminus V(G')]$.}\label{tree_28_1}
		\end{center}
	\end{figure}

	\begin{claim}\label{tree27}
		The tree $T_{27}$ is not a maximal subtree of $T$.
	\end{claim}
	\begin{proof}[Proof of Claim~\ref{tree27}]
		Suppose, to the contrary, that $T_{27}$ is a maximal subtree of $T$, and so $T_{27} = T_v$ where $v$ denotes the root of the subtree $T_v$.  We infer that the subgraph of $G$ associated with $T_{27}$ is obtained from either (i) the region $H_1$ by triangulating the region $v_1v_3v_4v_5v_6v_7v_8$  according to \Cref{lem1}\ref{5distance} as illustrated in Figure~\ref{tree_27}(a)-(p) or (ii) the region $H_2$ by triangulating the region $v_1v_3v_4v_5v_8v_7v_6$ according to \Cref{lem1}\ref{5distance} as illustrated in Figure~\ref{tree_27_1}(a)-(p), where we let $V(T_v) = \{v_1,v_2,v_{3}\}$ be the (shaded) triangle in $G$ associated with the vertex~$v$. In the following, we present arguments that work in each cases.
		
		Let $G'$ be the mop of order $n'$ obtained from $G$ by deleting the vertices $V_3^{14}$, and let $G'$ have $k'$ vertices of degree~$2$. We note that $n' = n-12$ and $k' = k - 1$, and $v_1v_{2}$ is an outer edge of $G'$. If $2 \le n' \le 4$, then $\{v_2,v_8,v_{11}\}$ is a 2DD-set of $G$, and hence $\gamma_2^d(G)\le 3 \le  \floor*{\frac{2}{9}(n+k)}$, a contradiction. If $5 \le n' \le 7$, then by \cref{obs3}, there exists a 2DD-set $D'$ of $G'$ such that $v_2\in D'$ and $|D'|=2$. Therefore, $D'\cup \{v_8,v_{11}\}$ is a 2DD-set of $G$, and so $\gamma_2^d(G)\le 4 \le \floor*{\frac{2}{9}(n+k)}$, a contradiction. Hence, $n' \ge 8$. Let $G_1$ be a graph of order $n_1$ obtained from $G'$ by contracting the edge $v_1v_{2}$ to form a new vertex $x$ in $G_1$, and let $G_1$ have $k_1$ vertices of degree~$2$. By Lemma~\ref{key}, $G_1$ is a mop. Since $n' \ge 8$, we note that $n_1 = n' - 1 \ge 7$. Further we note that $n_1 = n - 13$ and $k_1 \le k-1$. By the minimality of the mop $G$, we have $\gamma_2^d(G_1) \le \floor*{\frac{2}{9}(n_1+k_1)} \le \floor*{\frac{2}{9}(n-13+k-1)} \le \floor*{\frac{2}{9}(n+k)}-3$. Let $D_1$ be a $\gamma_2^d$-set of $G_1$. If $x \in D_1$, then let $D = (D_1 \setminus \{x\}) \cup \{v_1,v_2,v_8,v_{11}\}$. If $x \notin D_1$, then let $D = D_1 \cup \{v_2,v_8,v_{11}\}$. In both cases $D$ is a 2DD-set of $G$, and so $\gamma_2^d(G) \le |D| \le |D_1| + 3 \le \floor*{\frac{2}{9}(n+k)}$, a contradiction.
	\end{proof}

	\begin{claim}\label{tree28}
		The tree $T_{28}$ is not a maximal subtree of $T$.
	\end{claim}
	\begin{proof}[Proof of Claim~\ref{tree28}]
		Suppose, to the contrary, that $T_{28}$ is a maximal subtree of $T$, and so $T_{28} = T_v$ where $v$ denotes the root of the subtree $T_v$.  We infer that the subgraph of $G$ associated with $T_{28}$ is obtained from either (i) the region $H_{5}$ by triangulating the region $v_1v_3v_4v_5v_9v_6v_7v_8$  according to \Cref{lem1}\ref{6distance} as illustrated in Figure~\ref{tree_28}(a)-(p) or (ii) the region $H_6$ by triangulating the region $v_1v_3v_4v_5v_8v_6v_7v_9$ according to \Cref{lem1}\ref{6distance} as illustrated in Figure~\ref{tree_28}(q)-(t) and Figure~\ref{tree_28_1}(a)-(h)  or (iii) the region $H_7$ by triangulating the region $v_1v_3v_4v_5v_9v_6v_8v_7$ according to \Cref{lem1}\ref{6distance} as illustrated in Figure~\ref{tree_28_1}(i)-(p) or (iv) the region $H_8$ by triangulating the region $v_1v_3v_4v_5v_9v_8v_7v_6$ according to \Cref{lem1}\ref{6distance} as illustrated in Figure~\ref{tree_28_1}(q)-(t), where we let $V(T_v) = \{v_1,v_2,v_{3}\}$ be the (shaded) triangle in $G$ associated with the vertex~$v$. In the following, we present arguments that work in each cases.
		
		Let $G'$ be the mop of order $n'$ obtained from $G$ by deleting the vertices $V_3^{15}$, and let $G'$ have $k'$ vertices of degree~$2$. We note that $n' = n-13$ and $k' = k - 1$. If $2 \le n' \le 4$, then $\{v_2,v_8,v_{11}\}$ is a 2DD-set of $G$, and hence $\gamma_2^d(G)\le 3 \le  \floor*{\frac{2}{9}(n+k)}$, a contradiction. If $5 \le n' \le 7$, then by \cref{obs3}, there exists a 2DD-set $D'$ of $G'$ such that $v_2\in D'$ and $|D'|=2$. Therefore, $D'\cup \{v_8,v_{11}\}$ is a 2DD-set of $G$, and so $\gamma_2^d(G) \le 4 \le \floor*{\frac{2}{9}(n+k)}$, a contradiction. Hence, $n' \ge 8$. By the minimality of the mop $G$, we have $\gamma_2^d(G_1) \le \floor*{\frac{2}{9}(n'+k')} \le \floor*{\frac{2}{9}(n-13+k-1)} \le \floor*{\frac{2}{9}(n+k)}-3$. Let $D_1$ be a $\gamma_2^d$-set of $G_1$ and let $D = D_1 \cup \{v_2,v_8,v_{11}\}$. The set $D$ is a 2DD-set of $G$, and so $\gamma_2^d(G) \le |D| \le |D_1| + 3 \le \floor*{\frac{2}{9}(n+k)}$, a contradiction.
	\end{proof}
	
	We now return to the proof of Theorem~\ref{thm:main}. By Claims~\ref{tree1}-\ref{tree28}, we conclude that $T$ does not contains any tree $T_i$ shown in Figure~\ref{trees} as a subtree for $i\in [28]$, a contradiction to Claim~\ref{cl-2}. We deduce, therefore, that our supposition that Theorem~\ref{thm:main} is false is incorrect. Hence every maximal outerplanar of order $n \ge 7$ with $k$ vertices of degree~$2$ satisfies $\gamma_2^d(G) \le  \floor*{\frac{2}{9}(n+k)}$. This completes the proof of Theorem~\ref{thm:main}.	
\end{proof}

\section{Conclusion}
\label{Sect:conclude}

In this section, we show that upper bound shown in Theorem~\ref{thm:main} is tight. Note that each graph $G_i$ shown in Figure~\ref{tightexamples} has $\gamma_t^d(G_i) =  \floor*{\frac{2}{9}(n+k)}$, where $i\in[6]$.

\begin{figure}[H]
	\begin{center}
		\includegraphics[scale=.51]{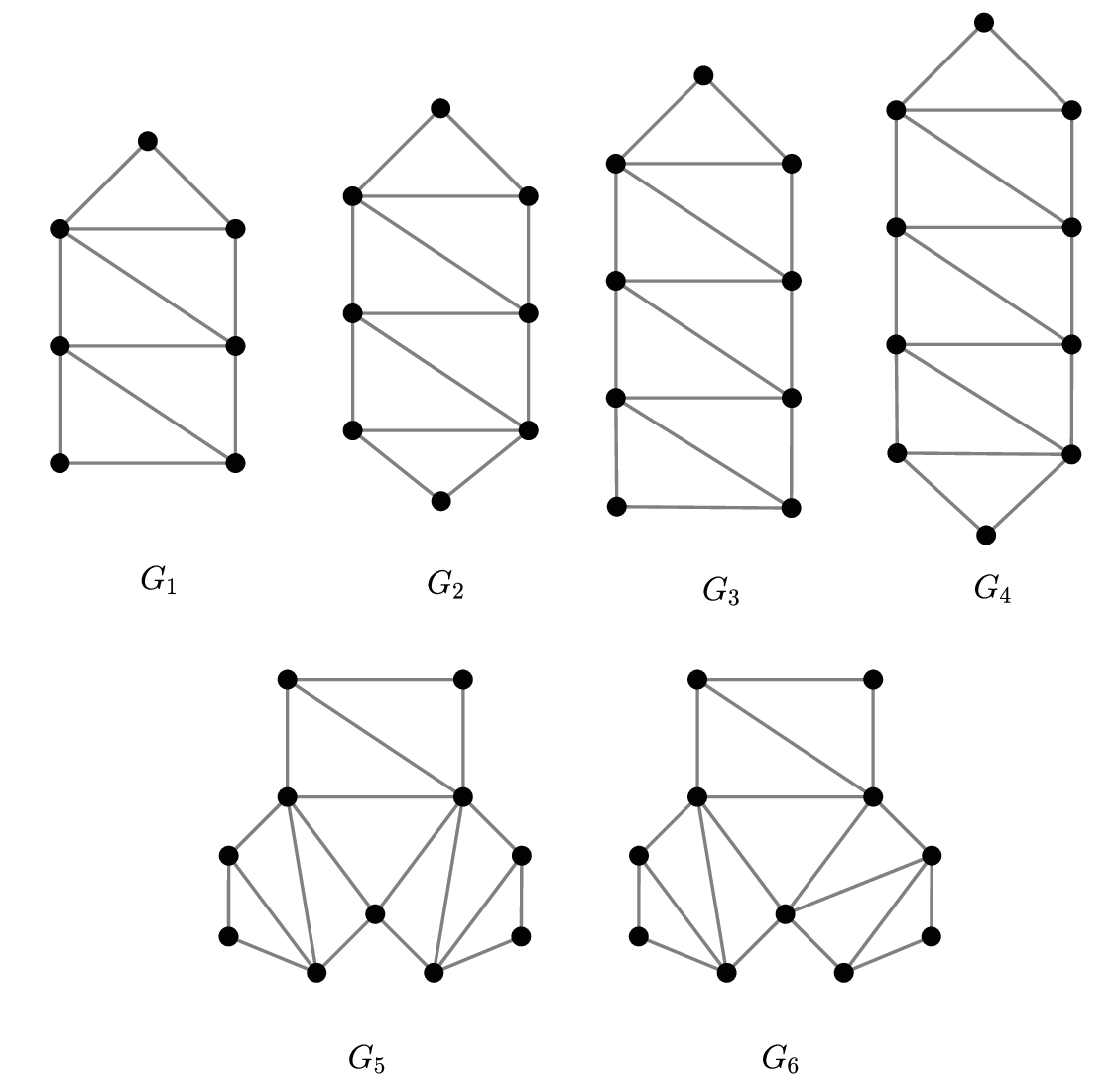}
		\caption{The graphs $G_1, G_2, G_3, G_4, G_5$ and $G_6$.} \label{tightexamples}
	\end{center}
\end{figure}

\section*{Declarations}

\noindent{\bf Conflict of interest} The authors do not have any financial or non financial interests that are directly or indirectly related to the work submitted for publication.

\noindent{\bf Data availability}
No data was used for the research described in this paper.

\noindent{\bf Acknowledgments} Research of Michael A. Henning was supported in part by the South African National Research Foundation (grants 132588, 129265) and the University of Johannesburg.

\end{document}